\documentclass{article}
\usepackage{Preamble}
\usepackage{graphicx} 

\usepackage{amsmath,amssymb,amsfonts,amsthm,mathtools,esint,mathrsfs,bm}
\usepackage{framed}
\usepackage{tikz}
\usepackage[font=small,labelfont=bf]{caption}
\usepackage{subcaption}
\usetikzlibrary{arrows}
\usetikzlibrary{matrix}
\usepackage{graphicx}
\usepackage{float}
\usepackage{subcaption}
\graphicspath{ {images/} }
\usepackage{color}
\usepackage{enumerate}
\usepackage{soul} 

\usepackage{geometry}
\definecolor{ForestGreen}{rgb}{0.0, 0.27, 0.13}
\definecolor{MidnightBlue}{rgb}{0.1, 0.1, 0.44}
\usepackage[colorlinks]{hyperref}   
\hypersetup{
	linkcolor = MidnightBlue, citecolor = ForestGreen   
}

\newcommand{\vertiii}[1]{{\left\vert\kern-0.25ex\left\vert\kern-0.25ex\left\vert #1 
    \right\vert\kern-0.25ex\right\vert\kern-0.25ex\right\vert}}		
\newcommand{\e}{\varepsilon}		
\newcommand{\Ediv}{E_{\varepsilon,\Div}}
\newcommand{\Ecur}{E_{\varepsilon,\Cur}}
\newcommand{\ediv}{e_{\varepsilon,\Div}}
\newcommand{\ecur}{e_{\varepsilon,\Cur}}
\newcommand{\Edivn}{E_{\varepsilon_{n},\Div}} 
\newcommand{\Ecurn}{E_{\varepsilon_{n},\Cur}} 
\newcommand{\grad}{\nabla}		
\newcommand{\R}{\mathbb{R}}		

\newcommand{\Ht}{H^1_T}	
\newcommand{\Hloc}{H_{{\mbox{\scriptsize loc}}}}	
\newcommand{\w}{\rightharpoonup}		
\DeclareMathOperator{\Div}{div}		
\DeclareMathOperator{\Cur}{curl}		
\DeclareMathOperator{\Curj}{\Cur_{j}}   

\newtheorem{thm}{Theorem}[section]
\newtheorem{prop}[thm]{Proposition}
\newtheorem{lem}[thm]{Lemma}
\newtheorem{cor}[thm]{Corollary}
\newtheorem*{assum*}{Assumption}			
\newtheorem{rem}[thm]{Remark}
\numberwithin{equation}{section}

\usepackage[capitalise]{cleveref}
\usepackage{authblk}

\title{A priori estimates and $\eta-$compactness for anisotropic Ginzburg-Landau minimizers with tangential anchoring}

\date{\today}

\begin{document}

\author[1]{Lia Bronsard\footnote{bronsard@mcmaster.ca }}
\author[2]{Andrew Colinet\footnote{acolinet@mun.ca}}
\author[3,1]{Dominik Stantejsky\footnote{dominik.stantejsky@univ-lorraine.fr}}
\author[1]{Lee van Brussel \footnote{vanbrulw@mcmaster.ca}}
\affil[1]{Department of Mathematics and Statistics, McMaster University, Hamilton, ON L8S 4L8 Canada}
\affil[2]{Department of Computer Science, Memorial University of Newfoundland, St. John's, NL A1C 5S7 Canada}
\affil[3]{Universit\'e de Lorraine, Institut \'Elie Cartan de Lorraine, UMR 7502 CNRS,  54506 Vand\oe uvre-l\`es-Nancy Cedex, France}

\maketitle

\begin{abstract}
We consider minimizers $u_\varepsilon$ of the Ginzburg-Landau energy with quadratic divergence or curl penalization on a simply-connected two-dimensional domain $\Omega$.
On the boundary, strong tangential anchoring is imposed.
We prove a priori estimates for $u_\varepsilon$ in $L^\infty$ uniform in $\varepsilon$ and that the Lipschitz constant of $u_\varepsilon$ blows up like $\varepsilon^{-1}$.
We then deduce compactness for a subsequence that converges to an $\mathbb{S}^1-$valued map with either one interior point defect or two boundary half-defects. We conclude our study with a proof that no boundary vortices can occur in the divergence penalized case.\\
\linebreak
\textbf{Keywords:}  Ginzburg-Landau energy, defects, boundary index, divergence and curl penalization, mixed Dirichlet-Robin boundary condition. \hfill\hfill
\linebreak
\end{abstract}

\tableofcontents


\section{Introduction}

In this paper, we consider an energy minimization problem motivated by two-dimensional nematic liquid crystals. 
The energy consists of the Ginzburg-Landau energy with additional elastic distortions, paired with the requirement that energy minimizing configurations adhere to a strong tangential anchoring condition along the boundary.
The primary purpose of this work is to fill a gap in existing results pertaining to the type (and number) of vortices one can expect from this setup in the length scale limit $\e\to 0$.
Results are known in the (classical) case of the Ginzburg-Landau energy with Dirichlet condition \cite{bbh} and strong tangential anchoring \cite{ABC,abv}.
The energy with additional elastic term as considered in this article has been investigated in \cite{colbert2013analysis,gomist} subject to a strong Dirichlet boundary condition.
The main result of this paper is summarized in Theorem~\ref{thm:maindiv}.
As a necessary preliminary, we also prove a priori bounds for minimizers completing the analysis from \cite{bcs24}. \\

The mathematical framework for this problem is as follows. We assume that a two-dimensional sample of nematic liquid crystal occupies a bounded, simply connected region $\Omega\subset\mathbb{R}^2$ with  $C^{3,1}-$boundary $\Gamma:=\partial\Omega$. 
The material is characterized by a relaxed \emph{director field} $u:\Omega\to\mathbb{R}^2$ which assigns to each $x\in\Omega$ a vector representing the preferred molecular alignment. 
The class of admissible configurations we consider is defined by 
\begin{equation*}
\Ht(\Omega)\coloneqq\left\{u\in H^1(\Omega;\mathbb{R}^2):\ \langle u,n\rangle=0\ \mbox{on}\ \Gamma\right\}
\, ,
\end{equation*}
where $\langle u,n\rangle$ is the inner product of $u$ with the outward unit normal vector $n\in\mathbb{S}^1$ to $\Gamma$. 
We also use the shorter notation $u_n \coloneqq\langle u,n\rangle$.
In this way, energy minimizing directors remain parallel to the unit tangent vector $\tau\coloneqq{}n^{\perp}$ to $\Gamma$, with the possibility of orientational change, and therefore permitting the existence of boundary vortices. 
Thus, the problem of boundary vortices amounts to observing the sign of $u_{\tau}\coloneqq\langle u,\tau\rangle$ along $\Gamma$.\\

The two energies studied in this work are of Ginzburg-Landau type with the addition of either a $(\Div u)^2$ or $(\Cur u)^2$ term as a means to penalize molecular \emph{splay} and \emph{bend} respectively:
\begin{equation}\label{def:DivEnergy}
\Ediv(u)\coloneqq\frac{1}{2}\int_{\Omega}\Bigl(|\grad u|^2+k(\Div u)^2+\frac{1}{2\e^2}\bigl(1-|u|^2\bigr)^2\Bigr)\mathrm{d}x
\end{equation}
\begin{equation}\label{def:CurlEnergy}
\Ecur(u)\coloneqq\frac{1}{2}\int_{\Omega}\Bigl(|\grad u|^2+k(\Cur u)^2+\frac{1}{2\e^2}\bigl(1-|u|^2\bigr)^2\Bigr)\mathrm{d}x
\, ,
\end{equation}
where $\e>0$ is the length scale parameter and $k>0$. 
Taking our minimization class to be $\Ht(\Omega)$, we obtain our desired minimization problems
\begin{equation}\label{eq:minprobST}
\inf\left\{\Ediv(u):u\in\Ht(\Omega;\R^{2})\right\},\ \inf\left\{\Ecur(u):u\in\Ht(\Omega;\R^{2})\right\}.
\end{equation}
Associated with these minimization problems are the Euler--Lagrange equations
\begin{equation}\label{eq:ELdivweak}
    \int_{\Omega}\!{}\nabla{}u_{\e}:\nabla{}v
    +k\int_{\Omega}\!{}(\Div u_{\e})(\Div v)
    =\frac{1}{\e^{2}}\int_{\Omega}\!{}
    u_{\e}\cdot{}v(1-|u_{\e}|^{2})
\end{equation}
and
\begin{equation}\label{eq:ELcurlweak}
   \int_{\Omega}\!{}\nabla{}u_{\e}:\nabla{}v
    +k\int_{\Omega}\!{}(\Cur u_{\e})(\Cur v)
    =\frac{1}{\e^{2}}\int_{\Omega}\!{}
    u_{\e}\cdot{}v(1-|u_{\e}|^{2}) 
\end{equation}
for all $v\in{}H_{T}^{1}(\Omega;\mathbb{R}^{2})$.
Formally, one can write the strong formulation of these equations and one gets
\begin{equation}\label{eq:ELdivstrong}
\left\{
\begin{alignedat}{2}
-\Delta u-k\nabla\Div u &=\frac{1}{\e^2}u(1-|u|^2)\quad && \mbox{in}\ \Omega,\\
u_{n}&=0\quad && \mbox{on}\ \Gamma,\\
\partial_nu_{\tau}&=0\quad && \mbox{on}\ \Gamma,
\end{alignedat}
\right.
\end{equation}

\begin{equation}\label{eq:ELcurlstrong}
\left\{
\begin{alignedat}{2}
-\Delta u-k\nabla^{\perp}\Cur u &=\frac{1}{\e^2}u(1-|u|^2)\quad && \mbox{in}\ \Omega,\\
u_{n}&=0\quad && \mbox{on}\ \Gamma,\\
\partial_n u_{\tau}+k\Cur u&=0\quad && \mbox{on}\ \Gamma.
\end{alignedat}
\right.
\end{equation}
Noting the decomposition of $u$ in the frame $\{n,\tau\}$ i.e.\ $u=u_\tau \tau+u_n n$ one can write $\Cur u = \partial_n u_\tau - \partial_\tau u_n + \kappa u_\tau$, where $\kappa$ is the curvature of the boundary.
Hence $0=\partial_n u_{\tau}+k\Cur u$ is equivalent to $(1+k)\partial_n u_\tau + k\kappa u_\tau = k\partial_\tau u_n = 0$, where in the last step we used that $u_n=0$ on $\Gamma$. 
The boundary condition for $u_\tau$ can thus be seen as a non-homogeneous Robin condition, see Subsection~\ref{subsec:coord} for more details.
This mixed Dirichlet-Robin boundary condition is the reason why one cannot apply the a priori results from \cite{bcs24} which only handles the Dirichlet-Neumann case coming from the divergence penalization \eqref{eq:ELdivstrong}.
The reflection-extension method used to prove the $L^\infty$ and Lipschitz bounds needs to account for the additional boundary term which we show can be absorbed into the differential operator while preserving ellipticity.
This is made precise in Section~\ref{sec:bounds}.\\

The physical motivation for studying these problems is twofold:
Many mathematical and physical results assume a simplified elastic energy which typically consists only of the Dirichlet energy of the deformation \cite{FuStaYoYo, LeKiKi, ball2011orientability, majzar,baumanparkphillips12, alama2015weak, alama2020thin, alama2021boojum,ACS21,ACS24, Alama_2016, Alama_2018}.
However, some materials such as lyotropic chromonic liquid crystals \cite{material,ZhoNaYu} posses properties that require a more precise description of their elastic characteristics which lead towards the inclusion of anisotropic elastic terms into the variational models \cite{Tasinkevych2012,Liu2013, bcmmssv2024, Kolapa_2022, gomist}.\\

The necessity for studying the type of boundary conditions as in this article comes from the anchoring transition results of Volovik \& Lavrentovich \cite{volovik1983topological} and the mathematical insight provided in Colbert-Kelly \& Phillips \cite{colbert2013analysis} for Smectic C* liquid crystals. 
Indeed, the situation investigated in \cite{volovik1983topological} with tangential anchoring yields two boundary defects, leading us to study the condition $u_n=0$ along $\Gamma$. 
Due to this boundary condition, a $\Ht$-minimizer $u_{\e}$ near a nontrivial boundary vortex $q\in\Gamma$ for $\Ediv$ or $\Ecur$ will either change orientation across $q$ or it will not. 
More precisely, we mean that along an interior arc connecting two nearby boundary points on either side of $q$, the phase of $u_{\e}$ along this arc changes either by an odd multiple of $\pi$ or even multiple of $\pi$ with respect to the positively oriented unit tangent vector $\tau(q)$. 
To capture this turning behaviour, an integer quantity called the \emph{boundary index} $D(q)\in\mathbb{Z}$ is defined--as the analogue of \emph{degree} for interior vortices.
The reader is directed to \cite[Section 4]{abv} for details on this construction, the general idea is as follows. 
We assume $|u_{\e}|\geq \frac{1}{2}$ in some small annular region centered at $q\in\Gamma$. 
Since $u_{\e}/|u_{\e}|=\pm\tau$ along the boundary of this annular region that coincides with $\Gamma$, we make the observation that the function $u_{\e}^2$ (when viewed as a complex-valued function in polar form) maintains the same orientation on either side of said boundary components. 
We then interpolate $u_{\e}^2$ across $q$, which defines a function with even degree along the boundary of the closed ball $\partial(\overline{B_r(q)}\cap\overline{\Omega})$. 
The boundary index of $u_{\e}$ is then defined as one half of this quantity. \\

To analyze the interaction of boundary vortices with bend and splay penalization, we take inspiration from \cite{colbert2013analysis}, where they consider an energy of the form 
\begin{equation}\label{def:energy-CKPh}
\frac{1}{2}\int_{\Omega}\left(k_s(\Div u)^2+k_b(\Cur u)^2+\frac{1}{2\e^2}\left(1-|u|^2\right)^2\right)\mathrm{d}x
\end{equation}
with non-equal splay $k_s>0$ and bend $k_b>0$ moduli and Dirichlet boundary data $u=g\in\mathbb{S}^1$ on $\partial\Omega$. 
It is found that the elastic energy density above can be rewritten in one of the two forms:
\begin{equation*}
k_s(\Div u)^2+k_b(\Cur u)^2=
\begin{cases}
k_b|\grad u|^2+(k_s-k_b)(\Div u)^2+2k_b\det(\grad u) & \mbox{if}\ k_s>k_b,\\
k_s|\grad u|^2+(k_b-k_s)(\Cur u)^2+2k_s\det(\grad u) & \mbox{if}\ k_b>k_s.
\end{cases}
\end{equation*}
In either case, the Dirichlet boundary data simplifies their energy as the Jacobian term integrates to $\pi$ times the winding number of $g$ along $\Gamma$. 
Our energies \eqref{def:DivEnergy} and \eqref{def:CurlEnergy} do not contain the Jacobian contribution, which would be needed to make the connection with \eqref{def:energy-CKPh} since due to our boundary condition, the Jacobian may contribute nontrivial energy in case of boundary vortices. \\

With this framework and physical motivation in hand, we are now ready to state our main result:
\begin{thm}\label{thm:maindiv}
Let $\Omega\subset\mathbb{R}^2$ be open, bounded, and simply-connected with $C^{3,1}-$boundary $\Gamma$. 
Let $\{u_{\e}\}_{\e>0}$ be a sequence of $H^1_T(\Omega;\mathbb{R}^2)$-minimizers for $\Ediv$ or $\Ecur$. 
Then there is a subsequence $\e_n\to 0$ and a finite set of point singularities $\Sigma\subset\overline{\Omega}$ such that 
\begin{equation*}
    u_{\e_n}\w u_0 \quad\mbox{weakly in}\ \Hloc^1(\overline{\Omega}\setminus\Sigma;\mathbb{R}^2)
    \, ,
\end{equation*}
where $u_0\in \Hloc^1(\overline{\Omega}\setminus\Sigma;\mathbb{R}^2)$ with $|u_0|=1$ $\mathcal{H}^2-$almost everywhere and $\langle u_0,n\rangle=0$ $\mathcal{H}^1-$almost everywhere on $\Gamma$. 
The degree and boundary index associated to each vortex for $u_0$ in $\Sigma$ is equal to one. Moreover,

\begin{enumerate}[(i)]
    \item $\Sigma=\{p\}\subset\Omega$ in the case of divergence penalization, and
    \item $\Sigma=\{p\}\subset\Omega$ or $\Sigma=\{q_1,q_2\}\subset\Gamma$ in the case of curl penalization. 
\end{enumerate}
\end{thm}

In the case of divergence penalization, our theorem guarantees a single interior vortex. However, the theorem does not give direct insight to precisely which vortex type is expected in the case of curl penalization. Numerical simulations suggest curl penalization favours boundary defects, see Figure~\ref{fig:disk:strAnch:div_curl} and \ref{fig:peanut:strAnch:div_curl}. 
This is due to the freedom of the boundary condition, which cannot be obtained by a smooth tangential Dirichlet condition, see Figure \ref{fig:disk:Dirichlet:curl}.\\

\begin{figure}
\begin{center}
\includegraphics[scale=0.22, angle=90]{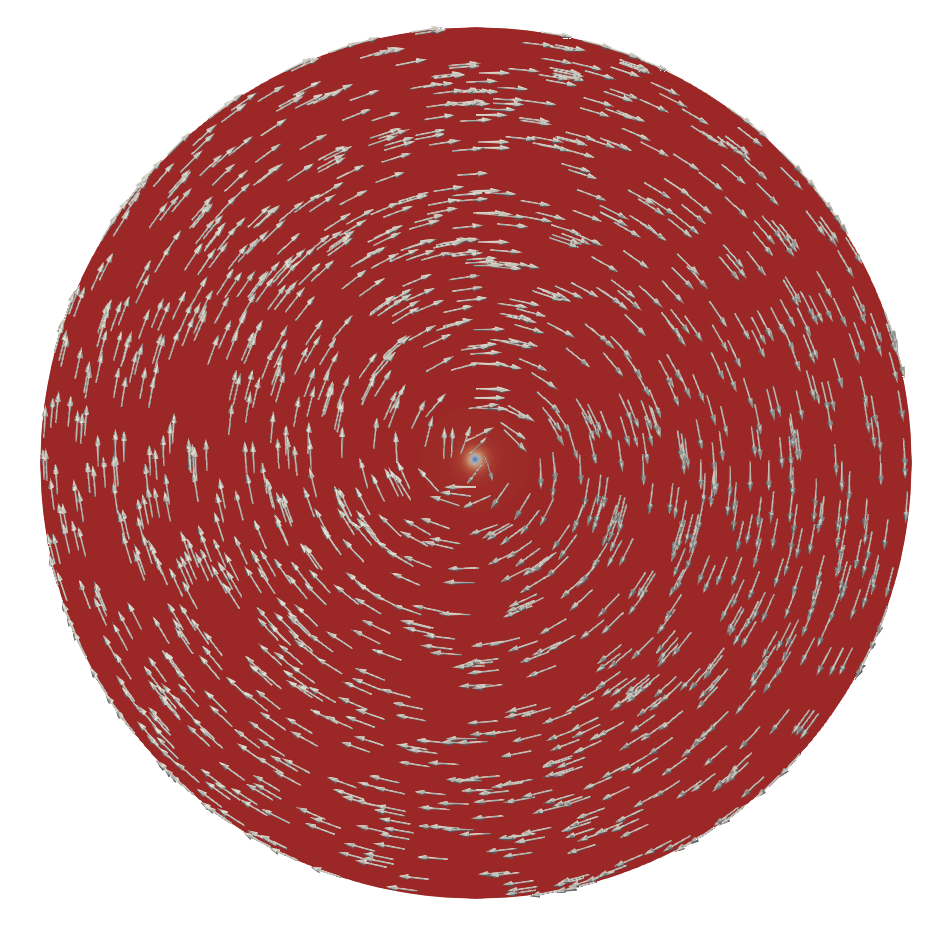} 
\includegraphics[scale=0.22, angle=90]{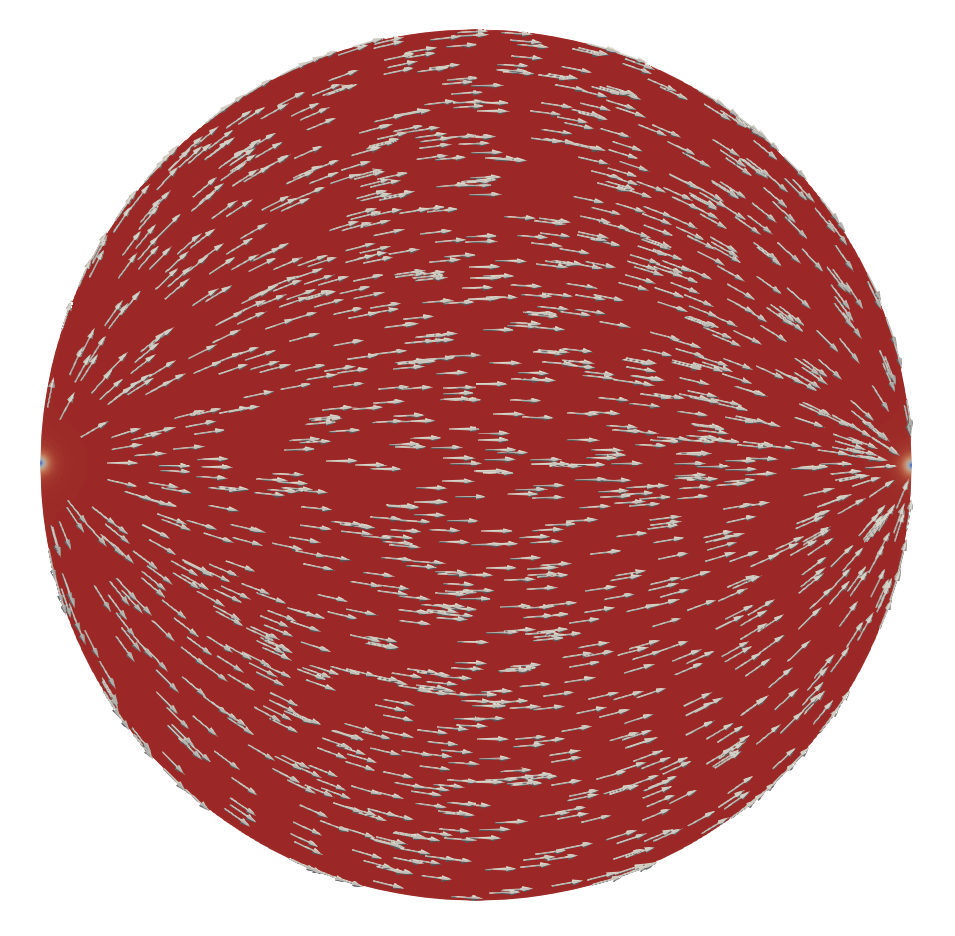} 
\caption{Plot of energy minimal vector fields $u_\e$ of $E_\e$ with divergence penalization (left) and curl penalization (right) for $k=1$ and $\e=0.01$ on the unit disk subject to the boundary condition $\langle u_\e,x\rangle=0$. 
The divergence penalized minimizers shows a divergence-free interior degree $1$ vortex, while in the curl-penalized case two opposing half-vortices on the boundary are preferred.}
\label{fig:disk:strAnch:div_curl}
\end{center}
\end{figure}

\begin{figure}
\begin{center}
\includegraphics[scale=0.22]{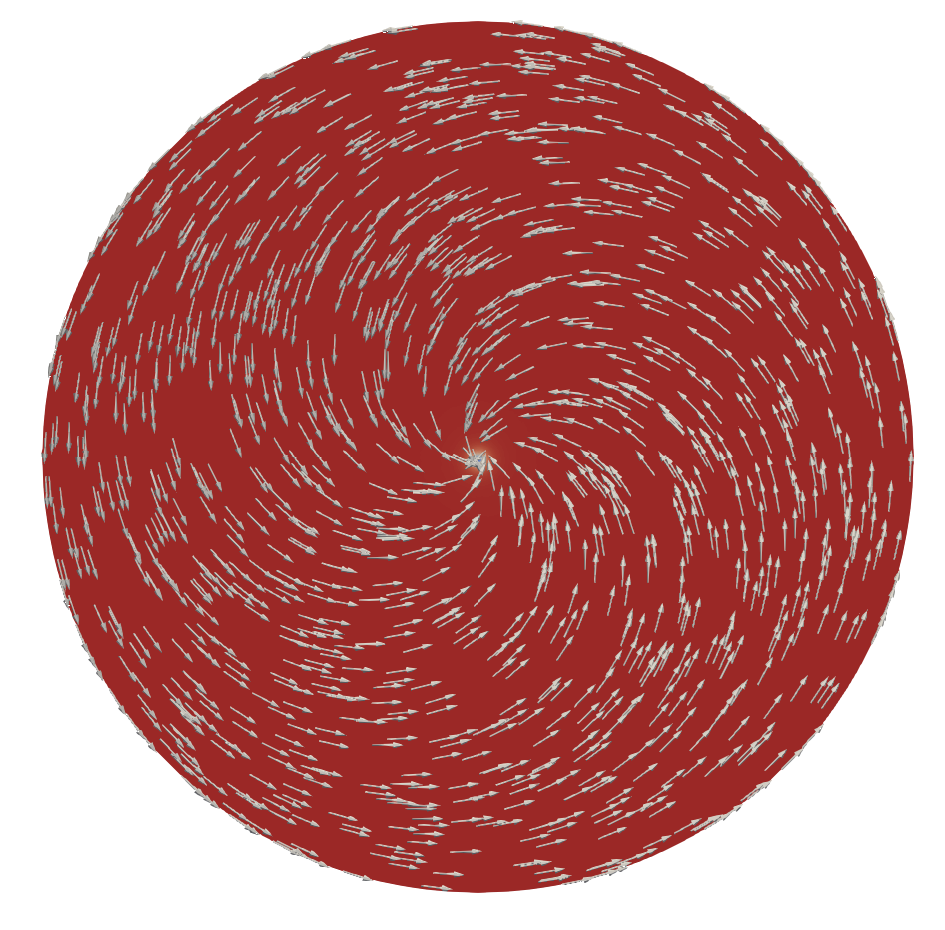} 
\caption{Plot of a minimizer $u_\e$ of $E_\e$ with curl-penalization and $u_\e(x)=x^\perp$ imposed on the boundary. 
Note that this is not the energy minimal configuration for $\langle u_\e,x\rangle=0$ (see Figure \ref{fig:disk:strAnch:div_curl}). 
The Dirichlet condition forces an interior defect of type $\frac{x}{|x|}$ at the origin, see also \cite{LKM2006}.}
\label{fig:disk:Dirichlet:curl}
\end{center}
\end{figure}

\begin{figure}
\begin{center}
\includegraphics[scale=0.25, angle=90]{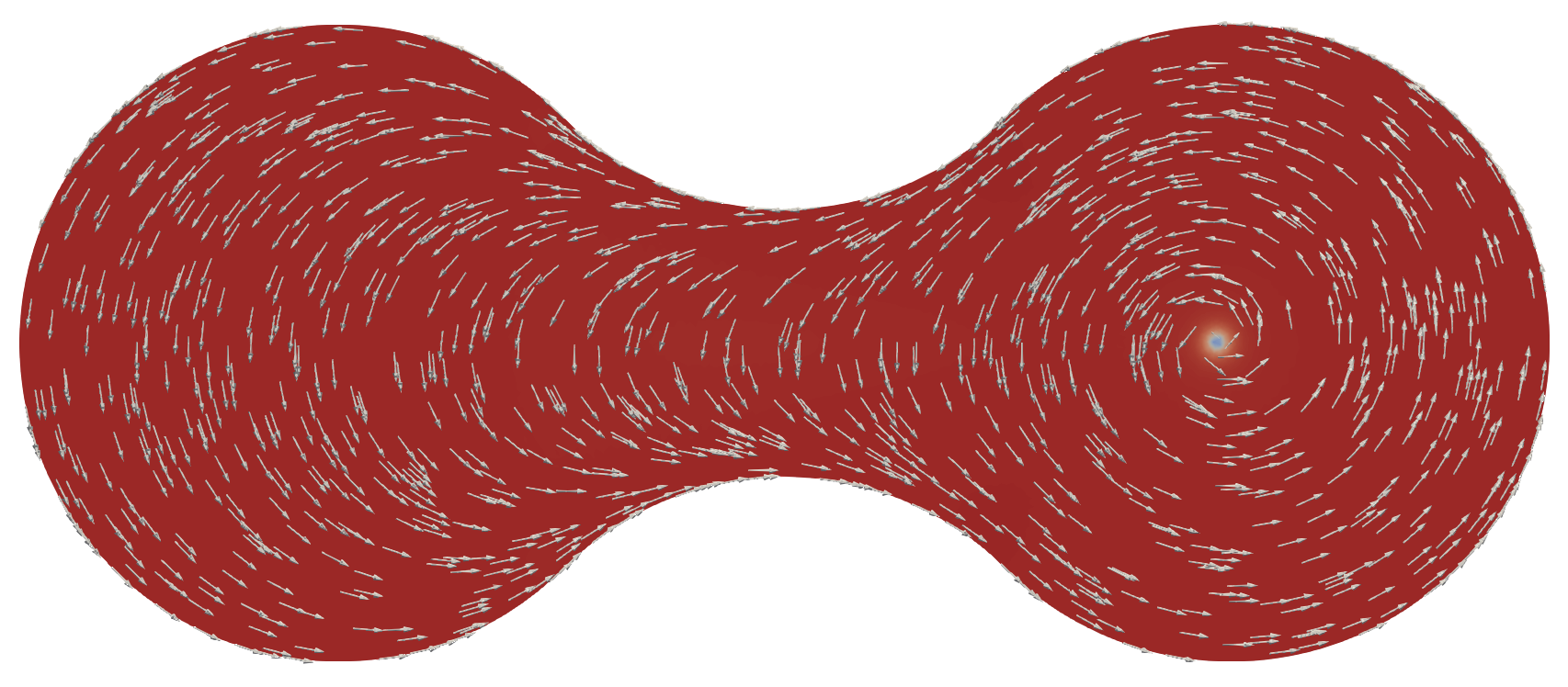} 
\includegraphics[scale=0.25, angle=90]{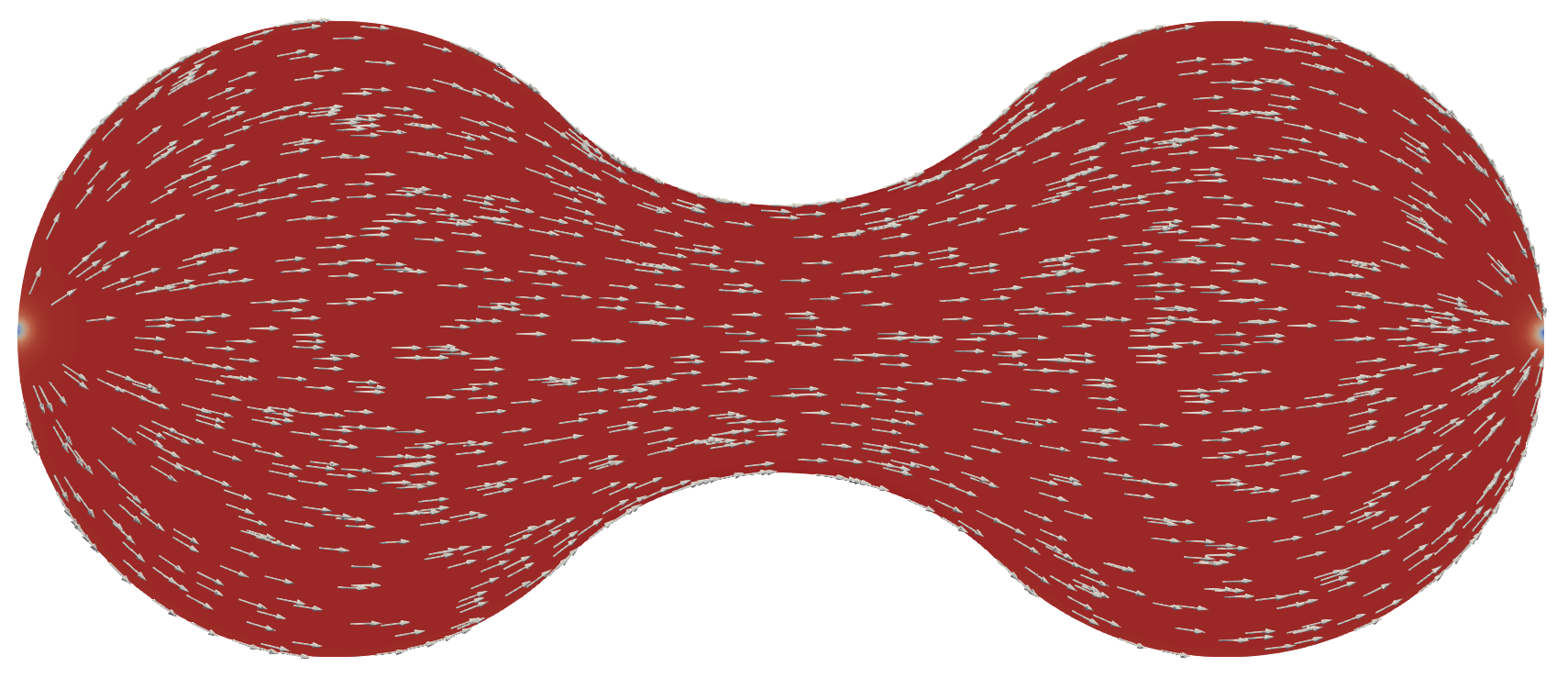} 
\caption{Plot of energy minimal vector fields $u_\e$ of $E_\e$ with divergence penalization (left) and curl penalization (right) for $k=1$ and $\e=0.01$ on a peanut shaped domain subject to the boundary condition $\langle u_\e,n\rangle=0$. 
For $\e$ sufficiently small, the divergence penalized minimizers shows a divergence-free interior degree $1$ vortex. 
In the curl-penalized case, two half-vortices on the boundary placed as far from each other as possible are optimal.}
\label{fig:peanut:strAnch:div_curl}
\end{center}
\end{figure}

The proof of Theorem~\ref{thm:maindiv} is presented in the following three sections of this article. 
In each section, we first analyze the case of divergence-penalized minimizers, followed by the necessary modifications for curl-penalized minimizers. 
Beginning in Section~\ref{sec:bounds}, we obtain an upper bound for minimizing energies and then provide a priori $L^{\infty}$-bounds for minimizers and their gradients. 
These a priori bounds are then used in Section~\ref{sec:eta} to prove an $\eta$-compactness result for both divergence and curl-dependent energies. 
The proof concludes with Section~\ref{sec:lowerbound} where a matching lower bound for the energies is found away from defects and the nonexistence of boundary vortices in the case of divergence penalization is shown.

\paragraph{Acknowledgment}
We would like to thank Stan Alama for his helpful discussions during Lee van Brussel's doctoral thesis on which part of this article is based.
LB is supported by an NSERC discovery grant. Part of this work was carried out while AC and DS were affiliated with McMaster University.

\section{Energy Upper Bound and a priori Estimates}\label{sec:bounds}

We begin this section by obtaining an upper bound for the energies of an $\Ht$-minimizer using competitors constructed by Colbert-Kelly and Phillips in \cite{colbert2013analysis} for both divergence and curl penalizations.
\begin{prop}\label{prop:energyupper}
If $u_{\e}$ is an $\Ht$-minimizer of $\Ediv$, then there is a constant $C>0$ independent of $\e$ such that
\begin{equation}\label{eq:Eupperbound}
\Ediv(u_{\e})\leq \pi|\ln\e|+C.
\end{equation}
Similarly, if $u_{\e}$ is an $\Ht$-minimizer of $\Ecur$, then inequality \eqref{eq:Eupperbound} holds for $\Ecur(u_{\e})$.
\end{prop}

\begin{proof}
By \cite[Proposition 2.1]{colbert2013analysis} there exists a minimizer $v$ for $\Ediv$ over the function space
\begin{equation*}
\{v\in H^1(\Omega;\mathbb{R}^2):v=\tau\ \mbox{on}\ \Gamma\}\subset \Ht(\Omega).
\end{equation*} 
By the set inclusion above, $\Ediv(u_{\e})\leq \Ediv(v)$ and noting $\deg(\tau;\Gamma)=1$, applying \cite[Proposition 2.1]{colbert2013analysis} to $\Ediv(v)$ yields
\begin{equation*}
\Ediv(u_{\e})\leq \Ediv(v)\leq \pi|\ln\e|+C
\end{equation*}
as desired. The proof for $\Ecur$ is identical.
\end{proof}

Another crucial ingredient for the proofs in Sections~\ref{sec:eta} and \ref{sec:lowerbound} are estimates on the $L^\infty-$ and Lipschitz norms of minimizers $u_\e$ of \eqref{def:DivEnergy} or \eqref{def:CurlEnergy}.
In the case of divergence penalization, i.e.\ \eqref{def:DivEnergy}, these estimates have been proven in \cite{bcs24}:

\begin{lem}[\cite{bcs24}]\label{lem:a-priori-div}
Let $\partial\Omega$ be of class $C^{3,1}$.
There exist constants $C_1, C_2,  \varepsilon_{0}>0$, only depending on
$k$ and $\Omega$ such that for all
$\varepsilon\in (0,\varepsilon_{0})$ and any minimizer $u_\varepsilon$ of
\eqref{def:DivEnergy} subject to the boundary condition $\langle u,n\rangle=0$, it holds
\begin{equation*}
\Vert u_\varepsilon\Vert_{L^\infty(\Omega;\mathbb{R}^{2})}
\ \leq \ 
C_1\quad \mbox{and}\quad \Vert
\nabla u_\varepsilon\Vert_{L^\infty(\Omega;\mathbb{R}^{2})}
\ \leq \ 
\frac{C_2}{\varepsilon}.
\end{equation*}
\end{lem}

In the curl-penalized case, we are going to prove the analogous results. We have the following Theorem:

\begin{thm}\label{thm:a-priori-curl}
Let $\partial\Omega$ be of class $C^{3,1}$.
There exist constants $C_1, C_2,  \varepsilon_{0}>0$, only depending on
$k$ and $\Omega$ such that for all
$\varepsilon\in (0,\varepsilon_{0})$ and any minimizer $u_\varepsilon$ of
\eqref{def:CurlEnergy} subject to the boundary condition $\langle u,n\rangle=0$, it holds
\begin{equation*}
\Vert u_\varepsilon\Vert_{L^\infty(\Omega;\mathbb{R}^{2})}
\ \leq \ 
C_1\quad \mbox{and}\quad \Vert
\nabla u_\varepsilon\Vert_{L^\infty(\Omega;\mathbb{R}^{2})}
\ \leq \ 
\frac{C_2}{\varepsilon}.
\end{equation*}
\end{thm}

The proof of this theorem is inspired by \cite{bcs24} but needs some important modifications to handle the Dirchlet-Robin boundary condition \eqref{eq:ELcurlstrong} instead of the Dirichlet-Neumann in the divergence penalized case \eqref{eq:ELdivstrong}.
In particular, a fundamental difference is that the Dirichlet-Neumann boundary conditions from \cite{bcs24} ensured that the extended function satisfied an elliptic equation of the same form while the Dirichlet-Robin boundary conditions can only establish this to highest order.
To compensate for this, additional first-order terms must be included in the elliptic operator.
This obstruction is handled rigorously in Subsection~\ref{subsec:glueing} and Appendix~\ref{app:calculations}.
As a result of the additional terms, the ``glueing" procedure requires an
addendum in order to extend to this setting.
Fortunately, ellipticity is determined by the highest order terms of the PDE and so the strategy from \cite{bcs24} is still admissible.

The rest of this section is devoted to the proof of Theorem~\ref{thm:a-priori-curl}.
As in \cite{bcs24}, we first introduce local coordinates that allow us to rewrite the PDE in \eqref{eq:ELcurlstrong} in a geometrically adapted frame which can then be extended via a reflection across the boundary $\partial\Omega$, see Subsection~\ref{subsec:coord}.
It is in this extension, that one needs to account for the different boundary condition, giving rise to an additional term in the extended PDE, see Lemma~\ref{lem:pdeext} in Subsection~\ref{subsec:glueing} and Appendix~\ref{app:calculations}.
To apply elliptic regularity estimates, we check in Subsection~\ref{subsec:ellip} that the differential operator of the extended PDE satisfies the Legendre-Hadamard condition, which allows us to conclude the proof.
Finally, in Subsection~\ref{subsec:regproof} we combine
the previous work and obtain our desired conclusion.
In addition, in Subsection \ref{subsec:glueing} we provide an
extended discussion surrounding elliptic regularity results
pertaining to elliptic systems of PDEs.
While results of this nature can be obtained by piecing together
known theorems we centralize this and provide a reference.


\subsection{Local Coordinates}\label{subsec:coord}

Here we provide a few remarks regarding the choice of local coordinates about $\Gamma=\partial\Omega$.
These will be used to define the extension.
A formal construction has been provided in \cite{ABC} so we only outline the properties needed for calculations.\\

As we will work frequently in a neighbourhood of $\Gamma$,
it will be convenient to introduce a suitable coordinate system.
In order to define this, we parametrize $\Gamma$ by its
arclength, $L$, using a positively oriented $C^{3,1}$ curve
$\gamma=(\gamma_{1}(y_{1}),\gamma_{2}(y_{1}))$
where $\gamma\colon\mathbb{R}\slash{}L\mathbb{Z}\to\Gamma$.
Using this, we define the positively oriented Frenet frame
$\{n,\tau\}$ where $n$ is taken to be the outward unit normal and
$\tau$ is an appropriate unit tangent vector.
Specifically, we define these quantities, for
$y_{1}\in\mathbb{R}\slash{}L\mathbb{Z}$, by
\begin{equation*}
    \tau(y_{1})\coloneqq\gamma'(y_{1}),\hspace{15pt}
    n(y_{1})\coloneqq-(\tau(y_{1}))^{\perp}.
\end{equation*}
This frame satisfies the sign-modified Frenet-Serret formulas
\begin{align*}
\left\{
\begin{array}{l}
\partial_{\tau}n = \kappa\tau\\[0.5em]
\partial_{\tau}\tau=-\kappa n
\end{array}
\right.
\end{align*}
where $\kappa=\kappa(x)$ is the curvature of the boundary at $x\in\Gamma$.
Using the positive frame $\{n,\tau\}$ defined near $\Gamma$ we may introduce a coordinate system defined in a tubular
neighbourhood of $\Gamma$.
To do this we begin by introducing the $C^{2,1}$ map
$X\colon(\mathbb{R}\slash{}L\mathbb{Z})\times(0,r_{0})\to
\Omega_{r_{0}}$ defined
by
\begin{equation}\label{def:precoord}
    X(y_{1},y_{2})\coloneqq\gamma(y_{1})-y_{2}n(y_{1}).
\end{equation}
Next, we cover a tubular neighbourhood of $\Gamma$ by sets $\{\mathcal{U}_{j}\}_{j=1}^{N}$, of the form $X(B_{r_{1}}(b_{j},0)\cap\{y_{2}>0\})$ for $j=1,2,\ldots,N$ where $r_{1}>0$ is a suitably chosen constant
dependent only on quantities related to $\Omega$, and
$b_{j}$ for $j=1,2,\ldots,N$ are appropriate centers on
$\mathbb{R}\slash{}L\mathbb{Z}$ chosen to form a cover of $\Gamma$.
Then, we define local coordinate
maps $\psi_{j}\colon{}B_{r_{1},+}(0)\to\mathcal{U}_{j}$ by
\begin{equation}\label{def:coord}
	\psi_{j}(y_{1},y_{2})\coloneqq{}X(b_{j}+y_{1},y_{2})
\end{equation}
and
\begin{equation*}
	B_{r_{1},+}(0)\coloneqq{}\bigl\{y\in{}B_{r_{1}}(0):y_{2}>0\bigr\}.
\end{equation*}
Since we will typically work inside a single coordinate chart and
translation will not affect the quantities we need to calculate
we will, for notational convenience, assume $b_{j}=0$.
As it will be helpful for calculations, we note that
\begin{equation*}
	\tau'(y_{1})=-\kappa(y_{1})n(y_{1}),\hspace{20pt}
	n'(y_{1})=\kappa(y_{1})\tau(y_{1})
\end{equation*}
as well as that
\begin{align*}
	\nabla\psi_{j}(y)&=\Bigl[(1-y_{2}\kappa(y_{1}))\tau(y_{1}),-n(y_{1})\Bigr],\\
	\nabla\psi_{j}^{-1}(x)&=
	\begin{pmatrix}
		\bigl(1-(\psi_{j}^{-1})^{2}(x)\kappa\bigl((\psi_{j}^{-1})^{1}(x)\bigr)\bigr)^{-1}\tau\bigl((\psi_{j}^{-1})^{1}(x)\bigr)^{T}\\
		-n\bigl((\psi_{j}^{-1})^{1}(x)\bigr)^{T}
	\end{pmatrix}.
\end{align*}
We observe that we can express a map $u\colon\Omega\to\mathbb{R}^{2}$ using these coordinates as
\begin{equation*}
	u(x)=\widetilde{u}_{\tau}\bigl((\psi_{j}^{-1})(x)\bigr)\tau\bigl((\psi_{j}^{-1})^{1}(x)\bigr)
	+\widetilde{u}_{n}\bigl((\psi_{j}^{-1})(x)\bigr)n\bigl((\psi_{j}^{-1})^{1}(x)\bigr)
\end{equation*}
where $x\in\mathcal{U}_{j}$ and $j=1,2,\ldots,N$.
We compute that
\begin{align}
    \partial_{x_{i}}u_{\e}(x)=&
    \bigl[\partial_{x_{i}}\widetilde{\psi}_{j}^{-1}(x)^{T}\nabla(\widetilde{u}_{\e})_{\tau}\bigl(\widetilde{\psi}_{j}^{-1}(x)\bigr)\bigr]
    \tau\bigl((\widetilde{\psi}_{j}^{-1})^{1}(x)\bigr)
    \label{eq:partialderiv}\\
    &-(\widetilde{u}_{\e})_{\tau}\bigl(\widetilde{\psi}_{j}^{-1}(x)\bigr)
    \partial_{x_{i}}(\widetilde{\psi}_{j}^{-1})^{1}(x)
    \kappa\bigl((\widetilde{\psi}_{j}^{-1})^{1}(x)\bigr)
    n\bigl((\widetilde{\psi}_{j}^{-1})^{1}(x)\bigr)
    \nonumber\\
    &+\bigl[\partial_{x_{i}}\widetilde{\psi}_{j}^{-1}(x)^{T}\nabla(\widetilde{u}_{\e})_{n}\bigl(\widetilde{\psi}_{j}^{-1}(x)\bigr)\bigr]n\bigl((\widetilde{\psi}_{j}^{-1})^{1}(x)\bigr)
    \nonumber
    \\
    &+(\widetilde{u}_{\e})_{n}\bigl(\widetilde{\psi}_{j}^{-1}(x)\bigr)
    \partial_{x_{i}}(\widetilde{\psi}_{j}^{-1})^{1}(x)
    \kappa\bigl((\widetilde{\psi}_{j}^{-1})^{1}(x)\bigr)
    \tau\bigl((\widetilde{\psi}_{j}^{-1})^{1}(x)\bigr).\nonumber
\end{align}
As a consequence we have
\begin{align}
	\nabla{}u(x)n(x)&=-\partial_{y_{2}}\widetilde{u}_{\tau}\bigl(\psi_{j}^{-1}(x)\bigr)\tau\bigl((\psi_{j}^{-1})^{1}(x)\bigr)
	-\partial_{y_{2}}\widetilde{u}_{n}\bigl(\psi_{j}^{-1}(x)\bigr)n\bigl((\psi_{j}^{-1})^{1}(x)\bigr),
    \label{eq:normderiv}\\
	\Cur u(x)&=
	-\partial_{y_{2}}\widetilde{u}_{\tau}\bigl(\psi_{j}^{-1}(x)\bigr)
	-\frac{\partial_{y_{1}}\widetilde{u}_{n}\bigl(\psi_{j}^{-1}(x)\bigr)}{1-(\psi_{j}^{-1})^{2}(x)\kappa\bigl((\psi_{j}^{-1})^{1}(x)\bigr)}
	+\frac{\kappa\bigl((\psi_{j}^{-1})^{1}(x)\bigr)\widetilde{u}_{\tau}\bigl(\psi_{j}^{-1}(x)\bigr)}
	{1-(\psi_{j}^{-1})^{2}(x)\kappa\bigl((\psi_{j}^{-1})^{1}(x)\bigr)}.
    \label{eq:curlderiv}
\end{align}
In particular, if $x\in\Gamma$ it holds $(\psi_{j}^{-1})^{2}(x)=0$ and thus we obtain
\begin{align*}
	\nabla{}u(x)n(x)&=-\partial_{y_{2}}\widetilde{u}_{\tau}\bigl(\psi_{j}^{-1}(x)\bigr)\tau\bigl((\psi_{j}^{-1})^{1}(x)\bigr)
	-\partial_{y_{2}}\widetilde{u}_{n}\bigl(\psi_{j}^{-1}(x)\bigr)n\bigl((\psi_{j}^{-1})^{1}(x)\bigr),\\
	\Cur u(x)&=
	-\partial_{y_{2}}\widetilde{u}_{\tau}\bigl(\psi_{j}^{-1}(x)\bigr)
	-\partial_{y_{1}}\widetilde{u}_{n}\bigl(\psi_{j}^{-1}(x)\bigr)
	+\kappa\bigl((\psi_{j}^{-1})^{1}(x)\bigr)\widetilde{u}_{\tau}\bigl(\psi_{j}^{-1}(x)\bigr).
\end{align*}
For convenience, when $x\in\mathcal{U}_{j}$ for $j=1,2,\ldots,N$, we introduce the notation
\begin{equation*}
    \tau(x)\coloneqq\tau\bigl((\psi_{j}^{-1})^{1}(x)\bigr),\hspace{5pt}
    n(x)\coloneqq{}n\bigl((\psi_{j}^{-1})^{1}(x)\bigr)
    \hspace{5pt}
\end{equation*}
as well as
\begin{equation*}
    u_{\tau}(x)\coloneqq{}\widetilde{u}_{\tau}\bigl(\psi_{j}^{-1}(x)\bigr),\hspace{5pt}
    u_{n}(x)\coloneqq{}\widetilde{u}_{n}\bigl(\psi_{j}^{-1}(x)\bigr).
\end{equation*}
Using this notation we have that for $x\in\mathcal{U}_{j}$ and $j=1,2,\ldots,N$ that
\begin{equation}\label{eq:frenetdecomp}
\begin{split}
u&=u_nn+u_{\tau}\tau,\\[0.5em]
\partial_nu&=\partial_nu_nn+\partial_nu_{\tau}\tau,\\[0.5em]
\partial_{\tau}u&=-\kappa u_{\tau}n+\partial_{\tau}u_{\tau}\tau.
\end{split}
\end{equation}
Through decompositions \eqref{eq:frenetdecomp}, an expansion of $\Cur u$ provides a more familiar view of the boundary conditions associated to \eqref{eq:ELcurlstrong}, namely that they can be rewritten as the Dirichlet-Robin pair
\begin{equation*}
\left\{
\begin{alignedat}{2}
u_{n}&=0\quad && \mbox{on}\ \Gamma,\\
(1+k)\partial_nu_{\tau}+k\kappa u_{\tau}&=0\quad && \mbox{on}\ \Gamma.
\end{alignedat}
\right.
\end{equation*}
Through a similar formula we may define an extension of \eqref{def:coord}, denoted
$\widetilde{\psi}_{j}\colon{}B_{r_{1}}(0)\to\widetilde{\mathcal{U}}_{j}$, where $\{\widetilde{\mathcal{U}}_{j}\}_{j=1}^{N}$ cover a two-sided
tubular neighbourhood of $\Gamma$.
By adjoining one additional set, $\widetilde{\mathcal{U}}_{0}$, which covers all points of $\Omega$ of sufficient distance from the boundary,
we obtain a cover of an open set $\widetilde{\Omega}$ which contains $\Omega$.
Pairing $\widetilde{\mathcal{U}}_{0}$ with the identity map we obtain an atlas for $\widetilde{\Omega}$.
We refer to \cite{ABC,bcs24} for more details of the construction.

If $u_{\varepsilon}$ is a minimizer of \eqref{def:CurlEnergy} in $H_{T}^{1}(\Omega;\mathbb{R}^{2})$ then we define an extension,
$U_{\varepsilon}$, in the following way:
\begin{equation}\label{def:extension}
	U_{\varepsilon}(x)\coloneqq
	\begin{cases}
		u_{\varepsilon}(x)&\text{if }x\in\Omega\, , \\
		A(x)u_{\varepsilon}\bigl(R(x)\bigr)&\text{if }x\in\widetilde{\Omega}\setminus\Omega\, ,
	\end{cases}
\end{equation}
where, for $x\in\widetilde{\mathcal{U}}_{j}$, 
\begin{align}
	A(x)&\coloneqq{}I_{2}-2n\bigl((\widetilde{\psi}_{j}^{-1})^{1}(x)\bigr)n\bigl((\widetilde{\psi}_{j}^{-1})^{1}(x)\bigr)^{T}\, ,\\
	R(x)&\coloneqq\widetilde{\psi}_{j}\Bigl(\Bigl[I_{2}-2\mathbf{e}_{2}\mathbf{e}_{2}^{T}\Bigr]\widetilde{\psi}_{j}^{-1}(x)\Bigr)
    \, .\label{eq:BdMirror}
\end{align}


\subsection{PDE Extension}\label{subsec:glueing}

In this subsection, we extend the PDE \eqref{eq:ELcurlstrong} satisfied by $u_{\varepsilon}$ on $\Omega$ to one satisfied by $U_{\varepsilon}$ on the extended domain $\widetilde{\Omega}$.
To do this,
we provide the main points of the necessary
computations needed to prove this.
For the interested reader, we provide an appendix containing additional
computational details.\\

Similar to \cite{bcs24}, we introduce a few preliminary
definitions and calculations needed in the proof before preceding.
Since $\widetilde{\mathcal{U}}_{0}\subset\Omega$ and
$U_{\e}=u_{\e}$ here
then we can focus on demonstrating that a PDE is satisfied by
$U_{\e}$ on $\bigcup_{j=1}^{N}\widetilde{\mathcal{U}}_{j}$.
Notice that we may further reduce the problem to demonstrating the glueing
to functions whose support is contained in
$\widetilde{\mathcal{U}}_{j}$ for $j=1,2,\ldots,N$ by appealing to a
partition of unity.
Using this reduction, we may assume that the test function, $v$, satisfies
$v\in{}H_{0}^{1}(\widetilde{\mathcal{U}}_{j};\mathbb{R}^{2})$
for some $j=1,2,\ldots,N$.
In addition, we may make use of the local coordinate descriptions
of $U_{\e}$ and $u_{\e}$ on each of $\widetilde{\mathcal{U}}_{j}$ for
$j=1,2,\ldots,N$.
We express $u_{\e}$ on $\widetilde{\mathcal{U}}_{j}$ for
$j=1,2,\ldots,N$ as
\begin{equation*}
    u_{\e}(x)=
    (\widetilde{u}_{\e})_{\tau}\bigl(\widetilde{\psi}_{j}^{-1}(x)\bigr)
    \tau\bigl((\widetilde{\psi}_{j}^{-1})^{1}(x)\bigr)
    +(\widetilde{u}_{\e})_{n}\bigl(\widetilde{\psi}_{j}^{-1}(x)\bigr)
    n\bigl((\widetilde{\psi}_{j}^{-1})^{1}(x)\bigr)
\end{equation*}
and we observe that equations \eqref{eq:partialderiv},
\eqref{eq:normderiv}, and \eqref{eq:curlderiv} all hold.
Notice that these identities remain valid for any suitable function $v$ when
expressed in tangent-normal coordinates.

Next, for $j=1,2,\ldots,N$ we define the function
$\sigma_{j}\colon\widetilde{\mathcal{U}}_{j}\to\widetilde{\mathcal{U}}_{j}$ by
\begin{equation*}
    \sigma_{j}(x)\coloneqq
    \begin{cases}
        x& \text{for }x\in\widetilde{\mathcal{U}}_{j}\cap\Omega\\
        R(x)&
        \text{for }x\in\widetilde{\mathcal{U}}_{j}\setminus\Omega
    \end{cases}
\end{equation*}
where $R$ is as defined in \eqref{eq:BdMirror}.
This function will be used on the enlarged coordinate chart $\widetilde{\mathcal{U}}_{j}$ to swap between an exterior point to $\Omega$ into its interior counterpart while leaving interior points invariant.
This function will be imperative for the glueing argument as we
intend to make use of the structure of the extension as well as
information valid in the interior.
We also introduce, for $j=1,2,\ldots,N$,
the function $\mathfrak{R}_{j}$, defined on
$\widetilde{\mathcal{U}}_{j}\times\mathbb{R}^{2}$ by
\begin{equation*}
    \mathfrak{R}_{j}(x,z)\coloneqq
    \Bigl[\tau\bigl((\widetilde{\psi}_{j}^{-1})^{1}(x)\bigr)
    \tau\bigl((\widetilde{\psi}_{j}^{-1})^{1}(x)\bigr)^{T}-
    n\bigl((\widetilde{\psi}_{j}^{-1})^{1}(x)\bigr)
    n\bigl((\widetilde{\psi}_{j}^{-1})^{1}(x)\bigr)^{T}
    \Bigr]z,
\end{equation*}
which corresponds to the function $R$ written in the original coordinate
system.
We observe that
\begin{equation*}
    \nabla\sigma_{j}(\sigma_{j}(x))\nabla\sigma_{j}(\sigma_{j}(x))^{T}=
    \begin{cases}
        I_{2}&\text{for }x\in\widetilde{\mathcal{U}}_{j}\cap\Omega,\\
        \nabla\widetilde{\psi}_{j}\bigl(\widetilde{\psi}_{j}^{-1}(x)\bigr)
        \mathcal{M}(x)
        \nabla\widetilde{\psi}_{j}\bigl(\widetilde{\psi}_{j}^{-1}(x)\bigr)^{T}&\text{for }x\in\widetilde{\mathcal{U}}_{j}\setminus\Omega,
    \end{cases}
\end{equation*}
where
\begin{equation*}
    \mathcal{M}(x)\coloneqq
    \begin{pmatrix}
            \frac{1}{\bigl(1-|(\widetilde{\psi}_{j}^{-1})^{2}(x)|
            \kappa\bigl((\widetilde{\psi}_{j}^{-1})^{1}(x)\bigr)\bigr)^{2}}& 0\\
            0& 1
    \end{pmatrix}.
\end{equation*}
Corresponding to this matrix we introduce the inner product defined for $x\in\widetilde{\mathcal{U}}_{j}$, for $j=1,2,\ldots,N$, by
\begin{equation*}
    \bigl<v,w\bigr>_{j}\coloneqq{}
    |\det(\nabla\sigma_{j}(x))|
    v^{T}\nabla\sigma_{j}\bigl(\sigma_{j}(x)\bigr)
    \nabla\sigma_{j}\bigl(\sigma_{j}(x)\bigr)^{T}w
    \, ,
\end{equation*}
where $v,w\in\mathbb{R}^{2}$.
This inner product will enter when verifying that the PDEs glue properly and to simplify the notation hereafter.
In addition, we introduce the following \emph{distortion factor}
\begin{equation*}
    \mathcal{D}_{j}(x)\coloneqq
    \begin{cases}
        1& \text{for }x\in\widetilde{\mathcal{U}}_{j}\cap\Omega,\\
        \frac{1-(\widetilde{\psi}_{j}^{-1})^{2}(x)\kappa\bigl((\widetilde{\psi}_{j}^{-1})^{1}(x)\bigr)}
        {1+(\widetilde{\psi}_{j}^{-1})^{2}(x)\kappa\bigl((\widetilde{\psi}_{j}^{-1})^{1}(x)\bigr)},
        &\text{for }x\in\widetilde{\mathcal{U}}_{j}\setminus\Omega,
    \end{cases}
\end{equation*}
that accounts for the deformation due to the change of variables from outside to the inside of the domain.
Using this distortion factor, we define
\begin{equation*}
    \Curj(w)(x)\coloneqq
    |\det(\nabla\sigma_{j}(x))|^{\frac{1}{2}}
    \biggl[
    \partial_{n}\biggl(w(x)\cdot\tau\bigl((\widetilde{\psi}_{j}^{-1})^{1}(x)\bigr)\biggr)
    -\mathcal{D}_{j}(x)\partial_{\tau}\biggl(w(x)\cdot{}n\bigl((\widetilde{\psi}_{j}^{-1})^{1}(x)\bigr)\biggr)
    \biggr]
\end{equation*}
for $x\in\widetilde{\mathcal{U}}_{j}$
for $j=1,2,\ldots,N$ and functions,
$w$, of appropriate regularity.
In particular, we use this notation to denote a quantity resembling curl
but including a compensating distortion factor.
This compensating factor ultimately arises from metric distortion in
the exterior of the domain and can be computed from the traditional
Euclidean curl in the interior of the domain after a reflection
change of variables.

Finally, for notational convenience, we let
$\mathcal{G}_{j}\colon\widetilde{\Omega}\times{}B_{r_{1}}(0)
\to{}M_{2\times2}(\mathbb{R})$, for $j=1,2,\ldots,N$,
denote the matrix-valued functions given by
\begin{equation*}
    \mathcal{G}_j(x,y)\coloneqq
    \begin{pmatrix}
        \frac{1}{(1-y_{2}\kappa(y_{1}))^{2}}& 0\\
        0& 1
    \end{pmatrix}
    \, ,
\end{equation*}
where $\kappa(y_1)$ denotes the curvature of $\Gamma=\partial\Omega$ at the point $x=\widetilde{\psi}_j(y_1,0)$.
With this notation in place we are ready to glue the PDEs together.\\

\begin{lem}\label{lem:pdeext}
    Suppose $\Omega\subset\mathbb{R}^{2}$ is an open, bounded, and
    simply connected set with $C^{3,1}$ boundary.
    Suppose also that $u_{\e}$ is a minimizer to \eqref{def:CurlEnergy}
    and that $U_{\e}$ is its corresponding extension to $\widetilde{\Omega}$
    defined as in \eqref{def:extension}.
    Then for $j=1,2,\ldots,N$ we have that $U_{\e}$ satisfies
    \begin{align}
    &\sum_{i=1}^{2}\int_{\widetilde{\mathcal{U}}_{j}}\!{}
    \bigl<\nabla{}U_{\e}^{i},\nabla{}v^{i}\bigr>_{j}
    +k\int_{\widetilde{\mathcal{U}}_{j}}\!{}
    (\Curj U_{\e})(\Curj v)
    \nonumber\\
    &+k\int_{\widetilde{\mathcal{U}}_{j}}\!{}
    |\det(\nabla\sigma_{j}(x))|\mathcal{D}_{j}(x)
    \biggl[\frac{\kappa\bigl((\widetilde{\psi}_{j}^{-1})^{1}(x)\bigr)U_{\e}(x)\cdot\tau\bigl((\widetilde{\psi}_{j}^{-1})^{1}(x)\bigr)}
    {1-(\widetilde{\psi}_{j}^{-1})^{2}(x)\kappa\bigl((\widetilde{\psi}_{j}^{-1})^{1}(x)\bigr)}\biggr]
    \nabla\bigl[v(x)\cdot{}n\bigl(\widetilde{\psi}_{j}^{-1}(x)\bigr)\bigr]\cdot{}n\bigl(\widetilde{\psi}_{j}^{-1}(x)\bigr)
    \nonumber\\
    =&\int_{\widetilde{\mathcal{U}}_{j}\cap\Omega}\!{}
    \frac{U_{\e}(x)\cdot{}v(x)}{\varepsilon^{2}}(1-|U_{\e}(x)|^{2})
    \nonumber\\
    &+\int_{\widetilde{\mathcal{U}}_{j}\setminus\Omega}\!{}
    |\det(\nabla\sigma_{j}(x))|
    \frac{\mathfrak{R}_{j}\bigl(x,U_{\e}(x)\bigr)\cdot{}v(x)}{\varepsilon^{2}}(1-|U_{\e}(x)|^{2})
    \nonumber\\
    &+\int_{\widetilde{\mathcal{U}}_{j}}\!{}\widetilde{\mathcal{F}}_{j}
    \bigl(x,U_{\e}(x),\nabla{}U_{\e}(x)\bigr)\cdot{}v(x)
    \label{def:ExtendedPDE}
\end{align}
where $\widetilde{F}_{j}$ satisfies
\begin{equation}\label{eq:carathmapApp}
    |\widetilde{\mathcal{F}}_{j}(x,z,p)|\le{}C(\Omega,k)\bigl[1+|z|+|p|\bigr],\hspace{15pt}
    |\nabla_{z,p}\widetilde{\mathcal{F}}_{j}(x,z,p)|\le{}C(\Omega,k).
\end{equation}
\end{lem}

\begin{proof}
    First, we express \eqref{eq:ELcurlweak} in terms of tangent and
    normal components.
    Using the identities discussed before the statement of this lemma
    in \eqref{eq:ELcurlweak} in addition to expressing $u_{\e}$ and $v$
    in tangent-normal coordinates, integrating by parts, and using that
    $\widetilde{v}_{n}=0$ gives
    \begin{align*}
        &\int_{\mathcal{U}_{j}}\!{}
        \biggl[
        \frac{\partial_{y_{1}}(\widetilde{u}_{\e})_{\tau}\bigl(\widetilde{\psi}_{j}^{-1}(x)\bigr)
        \partial_{y_{1}}\widetilde{v}_{\tau}\bigl(\widetilde{\psi}_{j}^{-1}(x)\bigr)}
        {\bigl(1-(\widetilde{\psi}_{j}^{-1})^{2}(x)\kappa\bigl((\widetilde{\psi}_{j}^{-1})^{1}(x)\bigr)\bigr)^{2}}
        +
        \partial_{y_{2}}(\widetilde{u}_{\e})_{\tau}\bigl(\widetilde{\psi}_{j}^{-1}(x)\bigr)
        \partial_{y_{2}}\widetilde{v}_{\tau}\bigl(\widetilde{\psi}_{j}^{-1}(x)\bigr)
        \biggr]
        \\
        &+k\int_{\mathcal{U}_{j}}
        \biggl[
        \partial_{y_{2}}(\widetilde{u}_{\e})_{\tau}\bigl(\widetilde{\psi}_{j}^{-1}(x)\bigr)
        +\frac{\partial_{y_{1}}(\widetilde{u}_{\e})_{n}\bigl(\widetilde{\psi}_{j}^{-1}(x)\bigr)}{1-(\widetilde{\psi}_{j}^{-1})^{2}(x)\kappa\bigl((\widetilde{\psi}_{j}^{-1})^{1}(x)\bigr)}
        \Bigr]
        \Bigl[\partial_{y_{2}}\widetilde{v}_{\tau}\bigl(\widetilde{\psi}_{j}^{-1}(x)\bigr)\biggr]
        \\
        &+k\int_{\mathcal{U}_{j}}\!{}
        \biggl[\frac{\kappa\bigl((\widetilde{\psi}_{j}^{-1})^{1}(x)\bigr)(\widetilde{u}_{\e})_{\tau}\bigl(\widetilde{\psi}_{j}^{-1}(x)\bigr)}{1-(\widetilde{\psi}_{j}^{-1})^{2}(x)\kappa\bigl((\widetilde{\psi}_{j}^{-1})^{1}(x)\bigr)}\biggr]
        \Bigl[\partial_{y_{2}}\widetilde{v}_{\tau}\bigl(\widetilde{\psi}_{j}^{-1}(x)\bigr)\Bigr]
        \\
        =&\int_{\mathcal{U}_{j}}\!{}
        \frac{(\widetilde{u}_{\e})_{\tau}\bigl(\widetilde{\psi}_{j}^{-1}(x)\bigr)
        \widetilde{v}_{\tau}\bigl(\widetilde{\psi}_{j}^{-1}(x)\bigr)}
        {\varepsilon^{2}}(1-|u|^{2})
        +\int_{\mathcal{U}_{j}}\!{}F_{j,\tau}\bigl(x,u_{\e}(x),\nabla{}u_{\e}(x)\bigr)
        \widetilde{v}_{\tau}\bigl(\widetilde{\psi}_{j}^{-1}(x)\bigr)
    \end{align*}
    and
    \begin{align*}
        &\int_{\mathcal{U}_{j}}\!{}
       \biggl[
        \frac{\partial_{y_{1}}(\widetilde{u}_{\e})_{n}\bigl(\psi_{j}^{-1}(x)\bigr)
        \partial_{y_{1}}\widetilde{v}_{n}\bigl(\psi_{j}^{-1}(x)\bigr)}
        {\bigl(1-(\psi_{j}^{-1})^{2}(x)\kappa\bigl((\psi_{j}^{-1})^{1}(x)\bigr)\bigr)^{2}}
        +
        \partial_{y_{2}}(\widetilde{u}_{\e})_{n}\bigl(\psi_{j}^{-1}(x)\bigr)
        \partial_{y_{2}}\widetilde{v}_{n}\bigl(\psi_{j}^{-1}(x)\bigr)
        \biggr]\\
        &+k\int_{\mathcal{U}_{j}}
        \biggl[
        \partial_{y_{2}}(\widetilde{u}_{\e})_{\tau}\bigl(\widetilde{\psi}_{j}^{-1}(x)\bigr)
        +\frac{\partial_{y_{1}}(\widetilde{u}_{\e})_{n}\bigl(\widetilde{\psi}_{j}^{-1}(x)\bigr)}{1-(\widetilde{\psi}_{j}^{-1})^{2}(x)\kappa\bigl((\widetilde{\psi}_{j}^{-1})^{1}(x)\bigr)}
        \biggr]
        \biggl[\frac{\partial_{y_{1}}\widetilde{v}_{n}\bigl(\widetilde{\psi}_{j}^{-1}(x)\bigr)}{1-(\widetilde{\psi}_{j}^{-1})^{2}(x)\kappa\bigl((\widetilde{\psi}_{j}^{-1})^{1}(x)\bigr)}
        \biggr]\\
        =&\int_{\mathcal{U}_{j}}\!{}
        \frac{(\widetilde{u}_{\e})_{n}\bigl(\psi_{j}^{-1}(x)\bigr)
        \widetilde{v}_{n}\bigl(\psi_{j}^{-1}(x)\bigr)}
        {\varepsilon^{2}}(1-|u|^{2})
        +\int_{\mathcal{U}_{j}}\!{}F_{j,n}\bigl(x,u_{\e}(x),\nabla{}u_{\e}(x)\bigr)
        \widetilde{v}_{n}\bigl(\psi_{j}^{-1}(x)\bigr)
        \, ,
    \end{align*}
    where $F_{j,\tau}$ and $F_{j,n}$ are determined by the remaining
    integrands that do not involve any derivatives of the test
    function, see \eqref{app:def-Fjt} and \eqref{app:def-Fjn} for their definition.
    Notice that $F_{j,\tau}$ and $F_{j,n}$ satisfy
    \begin{equation}\label{eq:carathmap}
        \max\{|F_{j,\tau}(x,z,p)|,|F_{j,n}(x,z,p)|\}
        \le{}C(\Omega,k)\bigl[1+|z|+|p|\bigr]
    \end{equation}
    and
    \begin{equation}\label{eq:carathmapgrad}
        \max\{|\nabla_{z,p}F_{j,\tau}(x,z,p)|,|\nabla_{z,p}F_{j,n}(x,z,p)|\}
        \le{}C(\Omega,k)
    \end{equation}
    for each $j=1,2,\ldots,N$ since all constituent terms also satisfy this.

    Combining our previous work now gives
    \begin{align*}
        &\sum_{i=1}^{2}\int_{\widetilde{\mathcal{U}}_{j}}\!{}
        \bigl<\nabla{}U_{\e}^{i},\nabla{}v^{i}\bigr>_{j}
        +k\int_{\widetilde{\mathcal{U}}_{j}}\!{}
        (\Curj U_{\e})(\Curj v)
        \nonumber\\
        &+k\int_{\widetilde{\mathcal{U}}_{j}}\!{}
        |\det(\nabla\sigma_{j}(x))|\mathcal{D}_{j}(x)
        \biggl[\frac{\kappa\bigl((\widetilde{\psi}_{j}^{-1})^{1}(x)\bigr)U_{\e}(x)\cdot\tau\bigl((\widetilde{\psi}_{j}^{-1})^{1}(x)\bigr)}
        {1-(\widetilde{\psi}_{j}^{-1})^{2}(x)\kappa\bigl((\widetilde{\psi}_{j}^{-1})^{1}(x)\bigr)}\biggr]
        \nabla\bigl[v(x)\cdot{}n\bigl(\widetilde{\psi}_{j}^{-1}(x)\bigr)\bigr]\cdot{}n\bigl(\widetilde{\psi}_{j}^{-1}(x)\bigr)
        \nonumber\\
        =&\int_{\widetilde{\mathcal{U}}_{j}\cap\Omega}\!{}
        \nabla(\widetilde{u}_{\e})_{\tau}\bigl(\widetilde{\psi}_{j}^{-1}(x)\bigr)^{T}
        \bigl[\mathcal{G}_{j}(x,\widetilde{\psi}_{j}^{-1}(x))\bigr]
        \bigl[\nabla\widetilde{v}_{\tau}\bigl(\widetilde{\psi}_{j}^{-1}(x)\bigr)+
        \nabla\widetilde{v}_{\tau}^{R}\bigl(\widetilde{\psi}_{j}^{-1}(x)\bigr)\bigr]\\
        &+
        \int_{\widetilde{\mathcal{U}}_{j}\cap\Omega}\!{}
        \nabla(\widetilde{u}_{\e})_{n}\bigl(\widetilde{\psi}_{j}^{-1}(x)\bigr)^{T}
        \bigl[\mathcal{G}_{j}(x,\widetilde{\psi}_{j}^{-1}(x))\bigr]
        \bigl[\nabla\widetilde{v}_{n}\bigl(\widetilde{\psi}_{j}^{-1}(x)\bigr)-
        \nabla\widetilde{v}_{n}^{R}\bigl(\widetilde{\psi}_{j}^{-1}(x)\bigr)\bigr]\\
        &+k\int_{\widetilde{\mathcal{U}}_{j}\cap\Omega}\!{}
        \biggl[
        \partial_{y_{2}}(\widetilde{u}_{\e})_{\tau}\bigl(\widetilde{\psi}_{j}^{-1}(x)\bigr)+
        \frac{\partial_{y_{1}}(\widetilde{u}_{\e})_{n}\bigl(\widetilde{\psi}_{j}^{-1}(x)\bigr)}
        {1-(\widetilde{\psi}_{j}^{-1})^{2}(x)\kappa\bigl((\widetilde{\psi}_{j}^{-1})^{1}(x)\bigr)}
        \biggr]
        \bigl[\partial_{y_{2}}\widetilde{v}_{\tau}\bigl(\widetilde{\psi}_{j}^{-1}(x)\bigr)-\partial_{y_{2}}(\widetilde{v}^{R})_{\tau}\bigl(\widetilde{\psi}_{j}^{-1}(x)\bigr)
        \bigr]
        \\
        &+k\int_{\widetilde{\mathcal{U}}_{j}\cap\Omega}\!{}
        \biggl[
        \partial_{y_{2}}(\widetilde{u}_{\e})_{\tau}\bigl(\widetilde{\psi}_{j}^{-1}(x)\bigr)+
        \frac{\partial_{y_{1}}(\widetilde{u}_{\e})_{n}\bigl(\widetilde{\psi}_{j}^{-1}(x)\bigr)}
        {1-(\widetilde{\psi}_{j}^{-1})^{2}(x)\kappa\bigl((\widetilde{\psi}_{j}^{-1})^{1}(x)\bigr)}
        \biggr]
        \frac{\partial_{y_{1}}\widetilde{v}_{\tau}\bigl(\widetilde{\psi}_{j}^{-1}(x)\bigr)+\partial_{y_{1}}(\widetilde{v}^{R})_{\tau}\bigl(\widetilde{\psi}_{j}^{-1}(x)\bigr)}
        {1-(\widetilde{\psi}_{j}^{-1})^{2}(x)\kappa\bigl((\widetilde{\psi}_{j}^{-1})^{1}(x)\bigr)}
        \\
        &+k\int_{\widetilde{\mathcal{U}}_{j}\cap\Omega}\!{}
        \biggl[\frac{\kappa\bigl((\widetilde{\psi}_{j}^{-1})^{1}(x)\bigr)(\widetilde{u}_{\e})_{\tau}\bigl(\widetilde{\psi}_{j}^{-1}(x)\bigr)}{1-(\widetilde{\psi}_{j}^{-1})^{2}(x)\kappa\bigl((\widetilde{\psi}_{j}^{-1})^{1}(x)\bigr)}\biggr]
        \Bigl[\partial_{y_{2}}\widetilde{v}_{\tau}\bigl(\widetilde{\psi}_{j}^{-1}(x)\bigr)
        -\partial_{y_{2}}(\widetilde{v}^{R})_{\tau}\bigl(\widetilde{\psi}_{j}^{-1}(x)\bigr)\Bigr]
        \\
        &+\int_{\widetilde{\mathcal{U}}_{j}\setminus\Omega}\!{}
        \widetilde{F}_{j}\bigl(x,U_{\e}(x),\nabla{}U_{\e}(x)\bigr)\cdot
        v(x).
    \end{align*}
    Since the operation
    \begin{equation*}
        w\mapsto{}\mathfrak{R}_{j}(x,w(\sigma_{j}(x))),
    \end{equation*}
    for suitable functions, is an involution for each
    $x\in\widetilde{\mathcal{U}}_{j}$ then
    we may define the even part relative to this involution, denoted
    $w_{E}$, by
    \begin{equation*}
        w_{E}(x)\coloneqq
        \frac{w(x)+\mathfrak{R}_{j}(x,w(\sigma_{j}(x)))}{2}.
    \end{equation*}
    Observe that
    \begin{equation*}
        (w_{E})_{\tau}(x)=
        \frac{\widetilde{w}_{\tau}(\widetilde{\psi}_{j}^{-1}(x))
        +\widetilde{w}_{\tau}(\widetilde{\psi}_{j}^{-1}(R(x)))}{2},
        \hspace{15pt}
        (w_{E})_{n}(x)=
        \frac{\widetilde{w}_{n}(\widetilde{\psi}_{j}^{-1}(x))
        -\widetilde{w}_{n}(\widetilde{\psi}_{j}^{-1}(R(x)))}{2}
    \end{equation*}
    and that for $x\in\Gamma$ satisfying
    $x=\widetilde{\psi}_{j}(y_{1},0)$
    \begin{equation*}
        (w_{E})_{n}(x)
        =\frac{\widetilde{w}_{n}(\widetilde{\psi}_{j}^{-1}(x))
        -\widetilde{w}_{n}(\widetilde{\psi}_{j}^{-1}(x))}{2}
        =\frac{\widetilde{w}_{n}(y_{1},0)
        -\widetilde{w}_{n}(y_{1},0)}{2}
        =0.
    \end{equation*}
    Thus, $w_{E}$ only has tangential part along $\Gamma$.
    Using this notation in the previous calculation combined with
    the PDE satisfied by $u_{\e}$ we find that
    \begin{align*}
        &
        \sum_{i=1}^{2}\int_{\widetilde{\mathcal{U}}_{j}}\!{}
        \bigl<\nabla{}U_{\e}^{i}(x),\nabla{}v^{i}(x)\bigr>_{j}
        +k\int_{\widetilde{\mathcal{U}}_{j}}\!{}
        \Curj(U_{\e})(x)\Curj(v)(x)\\
        &+k\int_{\widetilde{\mathcal{U}}_{j}}\!{}
        |\det(\nabla\sigma_{j}(x))|\mathcal{D}_{j}(x)
        \biggl[\frac{\kappa\bigl((\widetilde{\psi}_{j}^{-1})^{1}(x)\bigr)U_{\e}(x)\cdot\tau\bigl((\widetilde{\psi}_{j}^{-1})^{1}(x)\bigr)}
        {1-(\widetilde{\psi}_{j}^{-1})^{2}(x)\kappa\bigl((\widetilde{\psi}_{j}^{-1})^{1}(x)\bigr)}\biggr]
        \nabla\bigl[v(x)\cdot{}n\bigl(\widetilde{\psi}_{j}^{-1}(x)\bigr)\bigr]\cdot{}n\bigl(\widetilde{\psi}_{j}^{-1}(x)\bigr)
        \\
        =&2\int_{\widetilde{\mathcal{U}}_{j}\cap\Omega}\!{}
            \frac{U_{\e}\cdot{}v_{E}}{\varepsilon^{2}}(1-|U_{\e}|^{2})
            +\int_{\widetilde{\mathcal{U}}_{j}}\!{}
            \widetilde{\mathcal{F}}_{j}\bigl(x,U_{\e}(x),\nabla{}U_{\e}(x)\bigr)\cdot
            v(x),
    \end{align*}
    where $\widetilde{\mathcal{F}}_{j}$ combines $F_{j,\tau}$,
    $F_{j,n}$, and $\widetilde{F}_{j}$.
    Noting that
    \begin{equation*}
        U(x)\cdot\bigl[\mathfrak{R}_{j}\bigl(x,v(\sigma_{j}(x))\bigr)\bigr]=\mathfrak{R}_{j}(x,U(x))\cdot{}v(\sigma_{j}(x))
        \, ,
    \end{equation*}
    and changing variables gives
    \begin{align*}
        &\sum_{i=1}^{2}\int_{\widetilde{\mathcal{U}}_{j}}\!{}
        \bigl<\nabla{}U_{\e}^{i},\nabla{}v^{i}\bigr>_{j}
        +k\int_{\widetilde{\mathcal{U}}_{j}}\!{}
        (\Curj U_{\e})(\Curj v)
        \nonumber\\
        &+k\int_{\widetilde{\mathcal{U}}_{j}}\!{}
        |\det(\nabla\sigma_{j}(x))|\mathcal{D}_{j}(x)
        \biggl[\frac{\kappa\bigl((\widetilde{\psi}_{j}^{-1})^{1}(x)\bigr)U_{\e}(x)\cdot\tau\bigl((\widetilde{\psi}_{j}^{-1})^{1}(x)\bigr)}
        {1-(\widetilde{\psi}_{j}^{-1})^{2}(x)\kappa\bigl((\widetilde{\psi}_{j}^{-1})^{1}(x)\bigr)}\biggr]
        \nabla\bigl[v(x)\cdot{}n\bigl(\widetilde{\psi}_{j}^{-1}(x)\bigr)\bigr]\cdot{}n\bigl(\widetilde{\psi}_{j}^{-1}(x)\bigr)
        \nonumber\\
        =&\int_{\widetilde{\mathcal{U}}_{j}\cap\Omega}\!{}
        \frac{U_{\e}(x)\cdot{}v(x)}{\varepsilon^{2}}(1-|U_{\e}(x)|^{2})
        \nonumber\\
        &+\int_{\widetilde{\mathcal{U}}_{j}\setminus\Omega}\!{}
        |\det(\nabla\sigma_{j}(x))|
        \frac{\mathfrak{R}_{j}\bigl(x,U_{\e}(x)\bigr)\cdot{}v(x)}{\varepsilon^{2}}(1-|U_{\e}(x)|^{2})
        \nonumber\\
        &+\int_{\widetilde{\mathcal{U}}_{j}}\!{}\widetilde{\mathcal{F}}_{j}
        \bigl(x,U_{\e}(x),\nabla{}U_{\e}(x)\bigr)\cdot{}v(x).
    \end{align*}
\end{proof}


\subsection{Ellipticity}\label{subsec:ellip}

From the calculations of Subsection \ref{subsec:glueing} we see that the operator has changed for
$x\in\widetilde{\Omega}\setminus\Omega$.
In particular, it is worth noting that the operator now includes a first order term which did not appear when considering the divergence penalized case \cite{bcs24}.
In this subsection, we show that the new operator is still elliptic in the sense of Legendre-Hadamard.
The importance of this subsection is that, while the operator has
changed, the elliptic structure of the PDE remains intact.
As a result, one may apply elliptic regularity in order to obtain
control of stronger norms.\\

\begin{lem}\label{lem:ExtEllipticity}
    Suppose $\Omega\subset\mathbb{R}^{2}$ is an open, bounded, and
    simply connected set with $C^{3,1}$ boundary.
    Suppose also that $u_{\e}$ is a minimizer to \eqref{def:CurlEnergy}
    and that $U_{\e}$ is its corresponding extension to $\widetilde{\Omega}$
    defined as in \eqref{def:extension}.
    Then $U_{\e}$ weakly solves an elliptic PDE.
\end{lem}

\begin{proof}
    Since the modifications to the operator all occur for $x\in\widetilde{\Omega}\setminus\Omega$ and the operator was elliptic to start, it suffices to verify ellipticity only on $\widetilde{\Omega}\setminus\Omega$.
    In addition, since ellipticity is a pointwise condition and invariant under a change of variables, it is enough to verify ellipticity in
    local coordinates within the coordinate charts $\widetilde{\mathcal{U}}_{j}\setminus\Omega$ for each $j=1,2,\ldots,N$.
    As a result, we will use the weak formulation of the PDE in coordinates
    found for the extension in Subsection \ref{subsec:glueing}.
    First, notice that we can rewrite
    \begin{equation*}
    	\sum_{i=1}^{2}\int_{\widetilde{\mathcal{U}}_{j}\setminus\Omega}\!{}
        \bigl<\nabla{}U_{\varepsilon}^{i}(x),\nabla{}v^{i}(x)\bigr>_{j}
    \end{equation*}
    in the coordinates introduced in Subsection~\ref{subsec:coord} as
    \begin{align*}
        &\int_{\widetilde{\mathcal{U}}_{j}\setminus\Omega}\!{}
        |\det(\nabla\sigma_{j}(x))|
        \nabla(\widetilde{u}_{\e})_{\tau}\bigl(\widetilde{\psi}_{j}^{-1}(R(x))\bigr)^{T}
        \bigl[[I_{2}-2\mathbf{e}_{2}\mathbf{e}_{2}^{T}]\mathcal{G}_{j}(x,\widetilde{\psi}_{j}^{-1}(R(x)))\bigr]
        \nabla\widetilde{v}_{\tau}\bigl(\widetilde{\psi}_{j}^{-1}(x)\bigr)\\
        &-\int_{\widetilde{\mathcal{U}}_{j}\setminus\Omega}\!{}
        |\det(\nabla\sigma_{j}(x))|
        \nabla\widetilde({u}_{\e})_{n}\bigl(\widetilde{\psi}_{j}^{-1}(R(x))\bigr)^{T}
        \bigl[[I_{2}-2\mathbf{e}_{2}\mathbf{e}_{2}^{T}]\mathcal{G}_{j}(x,\widetilde{\psi}_{j}^{-1}(R(x)))\bigr]
        \nabla\widetilde{v}_{n}\bigl(\widetilde{\psi}_{j}^{-1}(x)\bigr)\\
        &+\int_{\widetilde{\mathcal{U}}_{j}\setminus\Omega}\!{}
        |\det(\nabla\sigma_{j}(x))|
        \kappa\bigl((\widetilde{\psi}_{j}^{-1})^{1}(x)\bigr)
        \widetilde{v}_{n}\bigl(\widetilde{\psi}_{j}^{-1}(x)\bigr)
        \nabla(\widetilde{u}_{\e})_{\tau}\bigl(\widetilde{\psi}_{j}^{-1}(R(x))\bigr)^{T}
        \bigl[[I_{2}-2\mathbf{e}_{2}\mathbf{e}_{2}^{T}]\mathcal{G}_{j}(x,\widetilde{\psi}_{j}^{-1}(R(x)))\bigr]
        \mathbf{e}_{1}\\
        &-\int_{\widetilde{\mathcal{U}}_{j}\setminus\Omega}\!{}
        |\det(\nabla\sigma_{j}(x))|
        \kappa\bigl((\widetilde{\psi}_{j}^{-1})^{1}(x)\bigr)
        \widetilde{v}_{\tau}\bigl(\widetilde{\psi}_{j}^{-1}(x)\bigr)
        \nabla(\widetilde{u}_{\e})_{n}\bigl(\widetilde{\psi}_{j}^{-1}(R(x))\bigr)^{T}
        \bigl[[I_{2}-2\mathbf{e}_{2}\mathbf{e}_{2}^{T}]\mathcal{G}_{j}(x,\widetilde{\psi}_{j}^{-1}(R(x)))\bigr]
        \mathbf{e}_{1}\\
        &-\sum_{i=1}^{2}\int_{\widetilde{\mathcal{U}}_{j}\setminus\Omega}\!{}
        |\det(\nabla\sigma_{j}(x))|
        \frac{\kappa\bigl((\widetilde{\psi}_{j}^{-1})^{1}(x)\bigr)
        n^{i}\bigl((\widetilde{\psi}_{j}^{-1})^{1}(x)\bigr)}
        {\bigl(1+(\widetilde{\psi}_{j}^{-1})^{2}(x)\kappa\bigl((\widetilde{\psi}_{j}^{-1})^{1}(x)\bigr)\bigr)^{2}}
        (\widetilde{u}_{\e})_{\tau}\bigl(\widetilde{\psi}_{j}^{-1}(R(x))\bigr)
        \Bigl[
        \nabla\widetilde{\psi}_{j}\bigl(\widetilde{\psi}_{j}^{-1}(x)\bigr)\mathbf{e}_{1}\Bigr]\cdot{}\nabla{}v(x)\\
        &+\sum_{i=1}^{2}\int_{\widetilde{\mathcal{U}}_{j}\setminus\Omega}\!{}
        |\det(\nabla\sigma_{j}(x))|
        \frac{\kappa\bigl((\widetilde{\psi}_{j}^{-1})^{1}(x)\bigr)
        \tau^{i}\bigl((\widetilde{\psi}_{j}^{-1})^{1}(x)\bigr)}
        {\bigl(1+(\widetilde{\psi}_{j}^{-1})^{2}(x)\kappa\bigl((\widetilde{\psi}_{j}^{-1})^{1}(x)\bigr)\bigr)^{2}}
        (\widetilde{u}_{\e})_{n}\bigl(\widetilde{\psi}_{j}^{-1}(R(x))\bigr)
        \Bigl[
        \nabla\widetilde{\psi}_{j}\bigl(\widetilde{\psi}_{j}^{-1}(x)\bigr)\mathbf{e}_{1}\Bigr]\cdot{}\nabla{}v(x).
    \end{align*}
    Second, we can rewrite the integral
    \begin{equation*}
    	k\int_{\widetilde{\mathcal{U}}_{j}\setminus\Omega}\!{}
    	(\Curj U_{\varepsilon})(\Curj v)
    \end{equation*}
    as
    \begin{align*}
        &k\int_{\widetilde{\mathcal{U}}_{j}\setminus\Omega}\!{}
        |\det(\nabla\sigma_{j}(x))|
        \partial_{y_{2}}(\widetilde{U}_{\e})_{\tau}\bigl(\widetilde{\psi}_{j}^{-1}(R(x))\bigr)
        \partial_{y_{2}}\widetilde{v}_{\tau}\bigl(\widetilde{\psi}_{j}^{-1}(x)\bigr)\\
        +&k\int_{\widetilde{\mathcal{U}}_{j}\setminus\Omega}\!{}
        |\det(\nabla\sigma_{j}(x))|
        \frac{
        \partial_{y_{1}}(\widetilde{U}_{\e})_{n}\bigl(\widetilde{\psi}_{j}^{-1}(R(x))\bigr)
        \partial_{y_{2}}\widetilde{v}_{\tau}\bigl(\psi_{j}^{-1}(x)\bigr)}
        {1+(\widetilde{\psi}_{j}^{-1})^{2}(x)\kappa\bigl((\widetilde{\psi_{j}^{-1}})^{1}(x)\bigr)}\\
        +&k\int_{\widetilde{\mathcal{U}}_{j}\setminus\Omega}\!{}
        |\det(\nabla\sigma_{j}(x))|
        \frac{\partial_{y_{2}}(\widetilde{U}_{\e})_{\tau}\bigl(\widetilde{\psi}_{j}^{-1}(R(x))\bigr)\partial_{y_{1}}\widetilde{v}_{n}\bigl(\psi_{j}^{-1}(x)\bigr)}
        {1+(\widetilde{\psi}_{j}^{-1})^{2}(x)\kappa\bigl((\widetilde{\psi_{j}^{-1}})^{1}(x)\bigr)}\\
        +&k\int_{\widetilde{\mathcal{U}}_{j}\setminus\Omega}\!{}
        |\det(\nabla\sigma_{j}(x))|
        \frac{\partial_{y_{1}}(\widetilde{U}_{\e})_{n}\bigl(\widetilde{\psi}_{j}^{-1}(R(x))\bigr)\partial_{y_{1}}\widetilde{v}_{n}\bigl(\psi_{j}^{-1}(x)\bigr)}
        {\bigl(1+(\widetilde{\psi}_{j}^{-1})^{2}(x)\kappa\bigl((\widetilde{\psi_{j}^{-1}})^{1}(x)\bigr)\bigr)^{2}}.
    \end{align*}
    Finally, we can rewrite the integral
    \begin{equation*}
    	k\int_{\widetilde{\mathcal{U}}_{j}}\!{}
        |\det(\nabla\sigma_{j}(x))|\mathcal{D}_{j}(x)
        \biggl[\frac{\kappa\bigl((\widetilde{\psi}_{j}^{-1})^{1}(x)\bigr)U_{\e}(x)\cdot\tau\bigl((\widetilde{\psi}_{j}^{-1})^{1}(x)\bigr)}
        {1-(\widetilde{\psi}_{j}^{-1})^{2}(x)\kappa\bigl((\widetilde{\psi}_{j}^{-1})^{1}(x)\bigr)}\biggr]
        \nabla\bigl[v(x)\cdot{}n\bigl(\widetilde{\psi}_{j}^{-1}(x)\bigr)\bigr]\cdot{}n\bigl(\widetilde{\psi}_{j}^{-1}(x)\bigr),
    \end{equation*}
    as
    \begin{align*}
        -k\int_{\widetilde{\mathcal{U}}_{j}}\!{}
        |\det(\nabla\sigma_{j}(x))|
        \frac{\kappa\bigl((\widetilde{\psi}_{j}^{-1})^{1}(x)\bigr)
        (U_{\e})_{\tau}\bigl(\psi_{j}^{-1}(x)\bigr)
        \partial_{y_{2}}\widetilde{v}_{n}\bigl(\widetilde{\psi}_{j}^{-1}(x)\bigr)}
        {1+(\widetilde{\psi}_{j}^{-1})^{2}(x)\kappa\bigl((\widetilde{\psi_{j}^{-1}})^{1}(x)\bigr)}.
    \end{align*}
    Summing these representations and changing coordinates using the change of variables $x=\widetilde{\psi}_{j}(y)$ leads to
    \begin{align*}
    	&\int_{B_{r_{1},-}(0)}\!{}
    	\biggl[\frac{\partial_{y_{1}}(\widetilde{U}_{\varepsilon})_{\tau}(y)
    	\partial_{y_{1}}\widetilde{v}_{\tau}(y)}
    	{(1+y_{2}\kappa(y_{1}))^{2}}
    	+(k+1)\partial_{y_{2}}(\widetilde{U}_{\varepsilon})_{\tau}(y)
    	\partial_{y_{2}}\widetilde{v}_{\tau}(y)\biggr]|\det(\nabla\widetilde{\psi}_{j}(x))|\\
    	+&\int_{B_{r_{1},-}(0)}\!{}\biggl[
    	\frac{(k+1)\partial_{y_{1}}(\widetilde{U}_{\varepsilon})_{n}(y)
    	\partial_{y_{1}}\widetilde{v}_{n}(y)}
    	{(1+y_{2}\kappa(y_{1}))^{2}}
    	+\partial_{y_{2}}(\widetilde{U}_{\varepsilon})_{n}(y)
    	\partial_{y_{2}}\widetilde{v}_{n}(y)
        \biggr]|\det(\nabla\widetilde{\psi}_{j}(x))|\\
    	+&k\int_{B_{r_{1},-}(0)}\!{}\biggl[
        \frac{\partial_{y_{2}}(\widetilde{U}_{\varepsilon})_{\tau}(y)
    	\partial_{y_{1}}\widetilde{v}_{n}(y)
        +\partial_{y_{1}}(\widetilde{U}_{\varepsilon})_{n}(y)
    	\partial_{y_{2}}\widetilde{v}_{\tau}(y)}{1+y_{2}\kappa(y_{1})}
        \biggr]|\det(\nabla\widetilde{\psi}_{j}(x))|\\
    	+&
        \text{lower order terms}.
    \end{align*}
    Next we check ellipticity of the leading order terms.
    We refer to Definition 3.36 in Subsection 3.41 of \cite{GiMa} for notation regarding PDE systems.
    From the above we can read off the coefficients $A_{ij}^{\alpha\beta}$ of the elliptic operator. 
    For $\alpha,\beta,i,j\in\{1,2\}$ it holds
    \begin{alignat*}{3}
            &
            A^{1,1}_{1,1}
            =
            \frac{|\det(\nabla\widetilde{\psi}_{j}(x))|}
            {(1+y_{2}\kappa(y_{1}))^{2}},
            \quad
            &&
            A^{2,2}_{1,1}
            =
            (1+k)|\det(\nabla\widetilde{\psi}_{j}(x))|,
            \quad
            &&
            A^{1,1}_{2,2}
            =
            \frac{(1+k)|\det(\nabla\widetilde{\psi}_{j}(x))|}
            {(1+y_{2}\kappa(y_{1}))^{2}},\\
            &
            A^{2,2}_{2,2}
            =
            |\det(\nabla\widetilde{\psi}_{j}(x))|,
            \quad
            &&
            A^{2,1}_{2,1}
            =
            \frac{k|\det(\nabla\widetilde{\psi}_{j}(x))|}{1+y_{2}\kappa(y_{1})},
            \quad
            &&
            A^{1,2}_{1,2}
            =
            \frac{k|\det(\nabla\widetilde{\psi}_{j}(x))|}{1+y_{2}\kappa(y_{1})},
    \end{alignat*}
    and we have $A_{ij}^{\alpha\beta}=0$ otherwise.
    These coefficients are Lipschitz continuous and the Arithmetic-Geometric inequality gives that
    \begin{align*}
        \sum_{k,m,i,j=1}^{2}A_{i,j}^{k,m}\xi_{k}\xi_{m}\eta_{i}\eta_{j}
        =&\frac{|\det(\nabla\widetilde{\psi}_{j}(x))|\xi_{1}^{2}|\eta|^{2}}
        {(1+y_{2}\kappa(y_{1}))^{2}}
        +|\det(\nabla\widetilde{\psi}_{j}(x))|\xi_{2}^{2}|\eta|^{2}\\
        &+k|\det(\nabla\widetilde{\psi}_{j}(x))|\xi_{2}^{2}\eta_{1}^{2}
        +\frac{k|\det(\nabla\widetilde{\psi}_{j}(x))|\xi_{1}^{2}\eta_{2}^{2}}{(1+y_{2}\kappa(y_{1}))^{2}}\\
        &+2|\det(\nabla\widetilde{\psi}_{j}(x))|\bigl[\sqrt{k}\xi_{2}\eta_{1}\bigr]
        \biggl[\frac{\sqrt{k}\xi_{1}\eta_{2}}{1+y_{2}\kappa(y_{1})}\biggr]\\
        \ge&|\det(\nabla\widetilde{\psi}_{j}(x))|\biggl[\frac{\xi_{1}^{2}|\eta|^{2}}{(1+y_{2}\kappa(y_{1}))^{2}}
        +\xi_{2}^{2}|\eta|^{2}\biggr]\\
        \ge&\frac{|\det(\nabla\widetilde{\psi}_{j}(x))|}{\bigl(1+r_{1}\left\|\kappa\right\|_{L^{\infty}((-r_{1},r_{1}))}\bigr)^{2}}|\xi|^{2}|\eta|^{2}.
    \end{align*}
    By the construction of the tubular neighbourhood and the fact that
    $\widetilde{\psi}_{j}$ is invertible over $\widetilde{\mathcal{U}}_{j}$ we
    can find a constant $C_{0}>0$ dependent only on $\Omega$ such that
    \begin{equation*}
        |\det(\nabla\widetilde{\psi}_{j}(x))|\ge{}C_{0}
    \end{equation*}
    for all $j=1,2,\ldots,N$.
    Then we have that
    \begin{equation*}
        \sum_{k,m,i,j=1}^{2}A_{i,j}^{k,m}\xi_{k}\xi_{m}\eta_{i}\eta_{j}
        \ge{}\frac{C_{0}}{\bigl(1+r_{1}\left\|\kappa\right\|_{L^{\infty}((-r_{1},r_{1}))}\bigr)^{2}}|\xi|^{2}|\eta|^{2}.
    \end{equation*}
\end{proof}

\subsection{Proof of Theorem \ref{thm:a-priori-curl}}
\label{subsec:regproof}
We now have all the necessary ingredients to prove Theorem~\ref{thm:a-priori-curl} and will now dedicate a subsection to supplying the argument.
As the proof closely follows \cite{bcs24} we only point out the relevant modifications hereafter.
We also provide an extended discussion which illustrates how to obtain regularity results for a Legendre-Hadamard elliptic system of equations.

\begin{proof}[Proof of Theorem~\ref{thm:a-priori-curl}]
    Noting that we only have to show the bounds close to the boundary, pick any $x_0\in\Gamma$ and perform the rescaling $\frac{1}{\e}(\widetilde{\Omega}-x_0)$ and $\widehat{U}_{\e}(z)\coloneqq{}U_{\e}(x_0+\e z)$, where $\widetilde{\Omega}$ and $U_{\e}$ are the extended domain and extended minimizer from Subsection~\ref{subsec:coord}.
    Using Subsection~\ref{subsec:ellip}, one sees that the PDE from Subsection~\ref{subsec:glueing} is elliptic in the sense of Legendre-Hadamard, so we would like to apply results from elliptic regularity.

    For the convenience of the reader we provide an extended discussion here: 
    \begin{enumerate}
        \item
            $W^{2,2}$ interior regularity for $L^{2}$ data is established e.g.\ in Theorem $4.11$ of \cite{Amb18}.
            We note that while the proof of the Theorem utilizes Legendre ellipticity it actually extends to Legendre-Hadamard ellipticity due to Theorem $4.4$ of
            \cite{GiMa}.
        \item
            Interior regularity of second derivative of the solution to the extended PDE from Subsection~\ref{subsec:glueing} (based on \eqref{eq:ELcurlstrong}) for $L^1$ data can be obtained through Calderon-Zygmund estimates.
    \end{enumerate}
    From here, we appeal to Stampaccia interpolation, to establish $W^{2,p}$ regularity given $L^{p}$ data for $1<p<2$.
    Finally, we proceed with a duality argument to deduce regularity in $W^{2,p}$ given $L^{p}$ data for $2<p<\infty$.
    In particular, we obtain a uniform bound in $W^{2,\frac43}$ after demonstrating that the PDE data is uniformly bounded in $L^{\frac{4}{3}}$.
    For this, we require a uniform $L^4-$bound of $\widehat{U}_{\e}$ which can be obtained as follows:
    First show that for the classical Ginzburg-Landau energy, defined by
    \begin{equation*}
    G_{\e}(u)
    \coloneqq
    \frac{1}{2}\int_{\Omega}\Bigl(|\grad u|^2+\frac{1}{2\e^2}\bigl(1-|u|^2\bigr)^2\Bigr)\mathrm{d}x
    \, ,
    \end{equation*}
    the energy $G_{\sqrt{2}\e}$ of $u_\e$ is bounded from above by $\pi|\ln\e|+C$, which follows as in \cite[Section 4]{bcs24} (based on \cite{colbert2013analysis}), with the minor modification that one has to replace the divergence by the curl and the test function $\rho_\e(x)\frac{(x-x_0)^\perp}{|x-x_0|}$ by $\rho_\e(x)\frac{x-x_0}{|x-x_0|}$.
    Then one derives a lower bound matching the logarithmic order with the precise constant $\pi$ as in \cite[Section 4]{bcs24}, based on \cite{abv}.
    Both upper and lower bound together imply that 
    \begin{align}\label{eq:uniform_bound_L4}
    \frac{1}{8\e^2}\int_\Omega (1-|u_\e|^2)^2 \dx x
    \ \leq \
    C
    \, ,
    \end{align}
    for a constant $C>0$ independent of $\e$ from which a uniform $L^{4}$ bound of $u_{\e}$ and hence also for $U_{\e}$ and $\widehat{U}_{\e}$ follows.
    We conclude the proof using Sobolev-Embedding Theorem and Morrey's inequality to control the $L^\infty-$norm of $\widehat{U}_{\e}$ (and hence of $u_{\e}$) by the  $W^{2,\frac43}-$norm which is in turn controlled by the $L^4-$norm of $\widehat{U}_{\e}$ for which we have a uniform bound in $\e$.
    For the Lipschitz bound, one needs to iterate the procedure to obtain a control in $W^{2,p}$ for $p>2$.
\end{proof}


\section{\texorpdfstring{$\eta$}{}-Compactness}\label{sec:eta}

In this section we prove an $\eta$-compactness result for solutions of \eqref{eq:ELdivstrong} and \eqref{eq:ELcurlstrong}. 
\begin{thm}[$\eta$-Compactness]\label{thm:eta}
Let $\frac{3}{4}\leq \beta<\gamma<1$. There exists constants $\eta$, $\tilde{C}$, $\e_0>0$ such that for any solution $u_{\e}$ of \eqref{eq:ELdivstrong} with $\e\in(0,\e_0)$, if $x_0\in\overline{\Omega}$ and $\Ediv(u_{\e};\omega_{\e^{\beta}}(x_0))\leq\eta |\ln\e|$, then
\begin{align}
&|u_{\e}|\geq \frac{1}{2}\quad \mbox{in}\ \omega_{\e^{\gamma}}(x_0)\label{ineq:half-divstrong}\quad \mbox{and}\\[0.5em]
&\frac{1}{4\e^2}\int_{\omega_{\e^{\gamma}}(x_0)}(1-|u_{\e}|^2)^2\,\mathrm{d}x\leq \tilde{C}\eta\label{ineq:ceta-divstrong}.
\end{align}
The same conclusions hold for solutions $u_{\e}$ of \eqref{eq:ELcurlstrong} with $\Ecur(u_{\e})\leq \eta|\ln\e|$.
\end{thm}
The proof for Theorem \ref{thm:eta} is broken into two subsections, focusing on divergence and curl penalization separately.


\subsection{Divergence Penalization}\label{sec:diveta}
Provided in Lemma \ref{lem:bigestimate-divstrong} below is an estimate required for the proof of Theorem \ref{thm:eta}. To set up the Lemma, we begin by stating a Pokhozhaev-type identity for solutions of \eqref{eq:ELdivstrong} which is obtained via integrating by parts against a smooth function.\\\\
Define
\begin{align*}
\ediv(u)\coloneqq\frac{1}{2}|\nabla u|^2+\frac{k}{2}(\Div u)^2+\frac{1}{4\e^2}\bigl(1-|u|^2\bigr)^2.
\end{align*}

\begin{prop}\label{prop:Pohdiv}
Let $\psi\in C^{2}(\Omega;\mathbb{R}^2)$. If $u$ is a solution of \eqref{eq:ELdivstrong}, then
\begin{multline}\label{PohDiv}
\int_{\partial\omega_r}\bigl\{\ediv(u)\langle\psi,n\rangle-\langle \partial_nu+k(\Div u)n,\psi\cdot\nabla u\rangle\bigr\}\mathrm{d}s\\=\int_{\omega_r}\Bigl\{\ediv(u)\Div \psi-\sum_{j,l}\partial_{x_{j}}\psi^{l}\langle\partial_{x_j}u,\partial_{x_{l}}u\rangle-k(\Div u)\sum_{i=1}^2\langle\partial_{x_i}\psi,\nabla u^i\rangle\Bigr\}\mathrm{d}x.
\end{multline}
\end{prop}
\begin{proof}
Taking the inner product of equation \eqref{eq:ELdivstrong} with $\psi\cdot\nabla u$ and integrating by parts over $\omega_r$, \cite[Lemma 4.2]{alama2015weak} gives
\begin{multline}\label{Pohkzero}
\int_{\partial\omega_r}\left\{\left(\frac{1}{2}|\nabla u|^2+\frac{1}{4\e^2}\left(1-|u|^2\right)^2\right)\langle \psi,n\rangle-\langle\partial_nu,\psi\cdot\nabla u\rangle\right\} \mathrm{d}s\\
=\int_{\omega_r}\left\{\left(\frac{1}{2}|\nabla u|^2+\frac{1}{4\e^2}\left(1-|u|^2\right)^2\right)\Div\psi-\sum_{j,l}\partial_{x_{j}}\psi^{l}\langle\partial_{x_j}u,\partial_{x_{l}}u\rangle\right\}\mathrm{d}x.
\end{multline}
in the case where $k=0$. To incorporate the divergence term, we compute
\begin{equation*}
-\int_{\omega_r}\langle \psi\cdot\nabla u,\nabla \Div u\rangle\,\mathrm{d}x=\int_{\omega_r}(\Div u)(\Div (\psi\cdot\nabla u))\,\mathrm{d}x-\int_{\partial\omega_r}(\Div u)\langle \psi\cdot\nabla u,n\rangle\,\mathrm{d}s.
\end{equation*}
Expanding the interior integrand, we have
\begin{equation*}
(\Div u)(\Div (\psi\cdot\nabla u))=\frac{1}{2}\langle \psi,\nabla(\Div u)^2\rangle+(\Div u)\sum_{i=1}^2\langle\partial_{x_i}\psi,\nabla u^i\rangle
\end{equation*}
and
\begin{equation*}
\int_{\omega_r}\langle \psi,\nabla(\Div u)^2\rangle\,\mathrm{d}x=\int_{\partial\omega_r}(\Div u)^2\langle\psi,n\rangle\,\mathrm{d}s-\int_{\omega_r}(\Div u)^2\Div\psi\,\mathrm{d}x.
\end{equation*}
Multiplying through by $k$ and sorting the integrals by boundary and interior type together with \eqref{Pohkzero} yields \eqref{PohDiv}.
\end{proof}

For $x_0\in\overline{\Omega}$, define as in \cite{struwe1994asymptotic, moser2003} the function
\begin{equation*}
F_{\Div}(r)\coloneqq{}F_{\Div}(r;x_0,u,\e)=r\int_{\partial B_r(x_0)\cap\Omega}\ediv(u)\,\mathrm{d}s.
\end{equation*}
\begin{lem}\label{lem:bigestimate-divstrong}
Let $x_0\in\overline{\Omega}$. There exists constants $C(k)>0$ and $r_0(k)>0$ such that for $\e\in(0,1)$ and $r\in(0,r_0)$ we have\\[0.5em]
\mbox{(1)}\ If $x_0\in\Omega$ and $\overline{\omega_r(x_0)}\cap\Gamma=\varnothing$,
\begin{align}\label{intinterior-divstrong}
\frac{1}{4\e^2}\int_{\omega_r}\bigl(1-|u|^2\bigr)^2\,\mathrm{d}x\leq C\biggl[r\int_{\omega_r}\frac{1}{2}\left(|\nabla u|^2+k(\Div u)^2\right) dx+F_{\Div}(r)\biggr],
\end{align}
\mbox{(2)}\ If $x_0\in\Gamma$,
\begin{align}\label{intboundstrong-divstrong}
\begin{split}
\frac{1}{4\e^2}\int_{\omega_r}&\bigl(1-|u|^2\bigr)^2\,\mathrm{d}x\leq C\biggl[r\int_{\omega_r}\frac{1}{2}\left(|\nabla u|^2+ k(\Div u)^2\right) dx+F_{\Div}(r)+\frac{r^2}{\e}\biggr].
\end{split}
\end{align}
\end{lem}
\begin{proof}[Proof of Lemma \ref{lem:bigestimate-divstrong}]
\ \\[0.5em]
\underline{Step 1: $x_0\in\Omega$}\\[0.5em]
Assume first $\omega_r=B_r(x_0)\subset\Omega$. Let $n$ and $\tau$ represent the unit outward normal and tangent vectors to $\partial\omega_r=\partial B_r(x_0)$ respectively and define the vector field $X=x-x_0$. Then, $|X|\leq r$ for all $x\in\omega_r$ with $X_n=\langle X,n\rangle =r$ on $\partial\omega_r$ and $X_{\tau}=\langle X,\tau\rangle=0$ on $\partial\omega_r$. To obtain \eqref{intinterior-divstrong}, consider the Pohosaev identity \eqref{PohDiv} and take $\psi=X$.\\[0.5em]
\underline{Estimates on $\partial\omega_r$}:\\[0.5em]
The lefthand side of \eqref{PohDiv} can be written as the sum of the following three integrals:
\begin{align*}
I_1&=\int_{\partial\omega_r}\biggl\{\frac{1}{2}|\nabla u|^2\langle X,n\rangle-\langle\partial_nu,X\cdot\nabla u\rangle\biggr\}\mathrm{d}s,\\[0.5em]
I_2&=\int_{\partial\omega_r}\biggl\{\frac{k}{2}(\Div u)^2\langle X,n\rangle-\langle k(\Div u)n,X\cdot\nabla u\rangle\biggr\}\mathrm{d}s,\\[0.5em]
I_3&=\frac{1}{4\e^2}\int_{\partial\omega_r}\bigl(1-|u|^2\bigr)^2\langle X,n\rangle\,\mathrm{d}s=\frac{r}{4\e^2}\int_{\partial\omega_r}\bigl(1-|u|^2\bigr)^2\,\mathrm{d}s.
\end{align*}
Since $X=rn$ on $\partial\omega_r$, we have that $X\cdot \nabla u=r\partial_nu$. The first integral has estimate
\begin{equation*}
I_1=r\int_{\partial\omega_r}\biggl\{\frac{1}{2}|\nabla u|^2-\langle\partial_nu,\partial_nu\rangle\biggr\}\mathrm{d}s\leq r\int_{\partial\omega_r}\frac{1}{2}|\nabla u|^2\,\mathrm{d}s.
\end{equation*}
For $I_2$, we use Cauchy-Schwarz:
\begin{align*}
I_2\leq r\int_{\partial\omega_r}\biggl\{\frac{k}{2}(\Div u)^2+k|(\Div u)n||\partial_nu|\biggr\}\mathrm{d}s\leq Cr\int_{\partial\omega_r}\biggl\{\frac{k}{2}(\Div u)^2+|\nabla u|^2\biggr\}\mathrm{d}s.
\end{align*}
Therefore there is a constant $C>0$ so that
\begin{align*}
I_1+I_2+I_3\leq Cr\int_{\partial\omega_r}\frac{1}{2}\biggl\{|\nabla u|^2+k(\Div u)^2+\frac{1}{2\e^2}(1-|u|^2)^2\biggr\}\mathrm{d}s=CF_{\Div}(r).
\end{align*}
\underline{Estimates in $\omega_r$}:\\[0.5em]
The righthand side of \eqref{PohDiv} can be written as the sum of the three integrals
\begin{align*}
J_1&=\int_{\omega_r}\biggl\{\frac{1}{2}|\nabla u|^2\Div X-\sum_{j,l}X_{x_j}^{l}\langle\partial_{x_j}u,\partial_{x_{l}}u\rangle\biggr\}\mathrm{d}x,\\[0.5em]
J_2&=\int_{\omega_r}\biggl\{\frac{k}{2}(\Div u)^2\Div X-k(\Div u)\sum_{i=1}^2\langle\partial_{x_i}X,\nabla u^i\rangle\biggr\}\mathrm{d}x,\\[0.5em]
J_3&=\frac{1}{4\e^2}\int_{\omega_r}\bigl(1-|u|^2\bigr)^2\Div X\,\mathrm{d}x.
\end{align*}
Since $X_{x_j}^l=\delta_{jl}$ and $\Div X=2>2-r$,
\begin{equation*}
J_1\geq \int_{\omega_r}\biggl\{|\nabla u|^2-\frac{r}{2}|\nabla u|^2-\sum_{j,l}\delta_{jl}\langle\partial_{x_j}u,\partial_{x_{l}}u\rangle\biggr\}\mathrm{d}x=-r\int_{\omega_r}\frac{1}{2}|\nabla u|^2\,\mathrm{d}x.
\end{equation*}
Similarly for $J_2$,
\begin{equation*}
J_2\geq \int_{\omega_r}\biggl\{k(\Div u)^2-\frac{k r}{2}(\Div u)^2-k(\Div u)\sum_{i=1}^2\langle\partial_{x_i}X,\nabla u^i\rangle\biggr\}\mathrm{d}x=-r\int_{\omega_r}\frac{k}{2}(\Div u)^2\,\mathrm{d}x.
\end{equation*}
Lastly, for $J_3$ we use $\Div X>1$ to get
\begin{align*}
J_3\geq \frac{1}{4\e^2}\int_{\omega_r}\bigl(1-|u|^2\bigr)^2\,\mathrm{d}x.
\end{align*}
Putting everything together, 
\begin{equation*}
\frac{1}{4\e^2}\int_{\omega_r}\bigl(1-|u|^2\bigr)^2\,\mathrm{d}x-r\int_{\omega_r}\biggl\{\frac{1}{2}|\nabla u|^2+\frac{k}{2}(\Div u)^2\biggr\}\mathrm{d}x\leq J_1+J_2+J_3\leq CF_{\Div}(r)
\end{equation*}
which proves inequality \eqref{intinterior-divstrong}.\\[0.5em]
\underline{Step 2: $x_0\in\Gamma$}\\[0.5em]
Let $r_0>0$ be chosen small enough so that $\Gamma\cap B_r(x_0)$ consists of a single smooth arc satisfying $|\Gamma_r|\leq Cr$ for all $0<r\leq r_0$. As in \cite{alama2015weak} we let $\mathcal{N}$ be a $2r_0$-neighbourhood of $\Gamma$, and by taking $r_0$ smaller if necessary, there exists a vector field $X\in C^2(\mathcal{N};\mathbb{R}^2)$ satisfying
\begin{align}
&\langle X,n\rangle=X_n=0\quad \mbox{for all}\ x\in\Gamma_r, \label{Xnormal}\\[0.5em]
&|X-(x-x_0)|\leq C|x-x_0|^2\quad \mbox{for all}\ x\in\omega_r, \label{Xbound}\\[0.5em]
& |\partial_{x_i}X^j-\delta_{ij}|\leq C|x-x_0|\quad \mbox{for all}\ x\in\omega_r, \label{DXbound}
\end{align} 
for a constant $C>0$ and for any $x_0\in\Gamma$. 
To obtain inequality \eqref{intboundstrong-divstrong}, we consider \eqref{PohDiv} with $\psi=X$ and find estimates for several of its terms. Using the decomposition $\partial\omega_r=\Gamma_r\cup(\partial B_r(x_0)\cap\Omega)$, it will be convenient to perform these estimates on $\Gamma_r$ and $\partial B_r(x_0)\cap\Omega$ separately.\\[0.5em]
\underline{Estimates Along $\Gamma_r$}:\\[0.5em]
Using \eqref{Xnormal}, one may write $X\cdot\nabla u=X_{\tau}\partial_{\tau}u$ which reduces the lefthand side of \eqref{PohDiv} to
\begin{equation*}
-\int_{\Gamma_r}\langle\partial_nu,X_{\tau}\partial_{\tau}u\rangle\,\mathrm{d}s-\int_{\Gamma_r}\langle k(\Div u)n,X_{\tau}\partial_{\tau}u\rangle\,\mathrm{d}s.
\end{equation*}
For the first integrand, representations \eqref{eq:frenetdecomp} along with the boundary conditions from \eqref{eq:ELdivstrong} and the pointwise results of \cite{bcs24} can be used in the proof of \cite[Lemma 3.3]{abv} to give a constant $c_1$ independent of $\e$ such that 
\begin{equation*}
|\langle\partial_nu,X_{\tau}\partial_{\tau}u\rangle|\leq |X_{\tau}|\frac{c_1}{\e}.
\end{equation*}
The second integrand can be written 
\begin{equation*}
\langle k(\Div u)n,X_{\tau}\partial_{\tau}u\rangle=X_{\tau}\langle k(\Div u)n,-\kappa u_{\tau}n+\partial_{\tau}u_{\tau}\tau\rangle=-k \kappa X_{\tau}(\Div u)u_{\tau}
\end{equation*}
where we recall $\kappa=\kappa(x)$ is the curvature function for $\Gamma$, which is uniformly bounded independent of $\e$. By the pointwise bounds of \cite{bcs24} we have
\begin{equation*}
|u_{\tau}|\leq C_1,\quad \mbox{and}\quad |\Div u|\leq |u^1_{x_1}|+|u^2_{x_2}|\leq\frac{2C_2}{\e}
\end{equation*}
and therefore there is a constant $c_2$ independent of $\e$ such that
\begin{equation*}
|k \kappa X_{\tau}(\Div u)u_{\tau}|\leq |X_{\tau}|\frac{c_2}{\e}.
\end{equation*}
Putting these estimates together, there is a constant $c$ independent of $\e$ where
\begin{equation*}
|\langle\partial_nu+k(\Div u)n,X_{\tau}\partial_{\tau}u\rangle|\leq |X_{\tau}|\frac{c}{\e}.
\end{equation*}
Now, given $|X_{\tau}|\leq Cr$ and $|\Gamma_r|\leq Cr$ we have another constant $C$ (independent of $\e$) so that
\begin{equation*}
\left|\int_{\Gamma_r}\langle\partial_nu+k(\Div u)n,X\cdot\nabla u\rangle\,\mathrm{d}s\right|\leq \int_{\Gamma_r}|X_{\tau}|\frac{c}{\e}\,\mathrm{d}s\leq \frac{Cr^2}{\e}.
\end{equation*}
\underline{Estimates Along $\partial B_r(x_0)\cap\Omega$}:\\[0.5em]
The lefthand side of \eqref{PohDiv} along $\partial B_r(x_0)\cap\Omega$ can be written as the sum of
\begin{align*}
I_1&\coloneqq\int_{\partial B_r(x_0)\cap\Omega}\biggl\{\frac{1}{2}|\nabla u|^2\langle X,n\rangle-\langle\partial_nu,X\cdot\nabla u\rangle\biggr\}\mathrm{d}s,\\[0.5em]
I_2&\coloneqq\int_{\partial B_r(x_0)\cap\Omega}\biggl\{\frac{k}{2}(\Div u)^2\langle X,n\rangle-\langle k(\Div u)n,X\cdot\nabla u\rangle\biggr\}\mathrm{d}s,\\[0.5em]
I_3&\coloneqq\frac{1}{4\e^2}\int_{\partial B_r(x_0)\cap\Omega}(1-|u|^2)^2\langle X,n\rangle\,\mathrm{d}s.
\end{align*}
Following the proof of \cite[Lemma 3.3]{abv} exactly as written, integrals $I_1$ and $I_3$ have the respective estimates
\begin{equation*}
I_1\leq Cr\int_{\partial B_r(x_0)\cap\Omega}|\nabla u|^2\,\mathrm{d}s,\qquad I_3\leq \frac{Cr}{4\e^2}\int_{\partial B_r(x_0)\cap\Omega}\bigl(1-|u|^2\bigr)^2\,\mathrm{d}s.
\end{equation*}
To estimate $I_2$, we write $X=\langle X,n\rangle n+\langle X,\tau\rangle\tau$ so that
\begin{align*}
-\langle k(\Div u)n,X\cdot\nabla u\rangle=-k\langle X,n\rangle\langle(\Div u)n,\partial_nu\rangle-k\langle X,\tau\rangle\langle(\Div u)n,\partial_{\tau}u\rangle.
\end{align*}
Then
\begin{align*}
I_2&\leq k Cr\int_{\partial B_r(x_0)\cap\Omega}\biggl\{\frac{1}{2}(\Div u)^2+|(\Div u)n||\partial_nu|+|(\Div u)n||\partial_{\tau}u|\biggr\}\mathrm{d}s\\[0.5em]
&\leq k Cr\int_{\partial B_r(x_0)\cap\Omega}\biggl\{\frac{1}{2}(\Div u)^2+\frac{1}{2}(\Div u)^2+\frac{1}{2}|\partial_nu|^2+\frac{1}{2}(\Div u)^2+\frac{1}{2}|\partial_{\tau}u|^2\biggr\}\mathrm{d}s\\[0.5em]
&=Cr\int_{\partial B_r(x_0)\cap\Omega}\biggl\{\frac{3k}{2}(\Div u)^2+\frac{k}{2}|\nabla u|^2\biggr\}\mathrm{d}s.
\end{align*}
Thus, for $C>0$ large enough
\begin{align*}
I_1+I_2+I_3&\leq Cr\int_{\partial B_r(x_0)\cap\Omega}\frac{1}{2}\biggl\{|\nabla u|^2+k(\Div u)^2+\frac{1}{2\e^2}\bigl(1-|u|^2\bigr)^2\biggr\}\mathrm{d}s\\
&=CF_{\Div}(r)
\end{align*}
and therefore
\begin{equation*}
\int_{\partial\omega_r}\bigl\{\ediv(u)\langle X,n\rangle-\langle\partial_nu+k(\Div u)n,X\cdot\nabla u\rangle\bigr\}\mathrm{d}s\leq C\left[F_{\Div}(r)+\frac{r^2}{\e}\right].
\end{equation*}
\underline{Estimates in $\omega_r$}:\\[0.5em]
The righthand side of \eqref{PohDiv} can be written as the sum of the integrals
\begin{align*}
J_1&\coloneqq\int_{\omega_r}\biggl\{\frac{1}{2}|\nabla u|^2\Div X-\sum_{j,l}X_{x_j}^{l}\langle\partial_{x_j}u,\partial_{x_{l}}u\rangle\biggr\}\mathrm{d}x,\\[0.5em]
J_2&\coloneqq\int_{\omega_r}\biggl\{\frac{k}{2}(\Div u)^2\Div X-k(\Div u)\sum_{i=1}^2\langle\partial_{x_i}X,\nabla u^i\rangle\biggr\}\mathrm{d}x,\\[0.5em]
J_3&\coloneqq\frac{1}{4\e^2}\int_{\omega_r}\bigl(1-|u|^2\bigr)^2\Div X\,\mathrm{d}x,
\end{align*}
By the proof of \cite[Lemma 3.3]{abv} the integrals $J_1$ and $J_3$ have the respective estimates
\begin{equation*} 
J_1\geq -Cr\int_{\omega_r}\frac{1}{2}|\nabla u|^2\,\mathrm{d}x,\qquad 
J_3\geq \frac{1}{4\e^2}\int_{\omega_r}(1-|u|^2)^2\,\mathrm{d}x.
\end{equation*}
For $J_2$:
\begin{align*}
k(\Div u)\sum_{i=1}^2\langle\partial_{x_i}X,\nabla u^i\rangle & \leq k(\Div u)^2+k\sum_{i=1}^2\left[|X_{x_i}^i-1||u_{x_i}^i|^2+|X_{x_i}^i-1||u_{x_1}^1||u_{x_2}^2|\right]\\
&\qquad+k\sum_{\substack{j,l\\j\neq l}}\left(|X_{x_j}^{l}||u_{x_1}^1||u_{x_{l}}^j|+|X_{x_j}^{l}||u_{x_{l}}^j||u_{x_2}^2|\right).
\end{align*}
Applying condition \eqref{DXbound} on the necessary terms, there exists $C>0$ such that
\begin{align*}
k(\Div u)\sum_{i=1}^2\langle\partial_{x_i}X,\nabla u^i\rangle\leq k(\Div u)^2+Ck r|\nabla u|^2.
\end{align*}
Now, since
\begin{align*}
\Div X=X_{x_1}^1+X_{x_2}^2=2+(X_{x_1}^1-1)+(X_{x_2}^2-1)\geq 2 - 2Cr
\end{align*}
we write
\begin{align*}
J_2&\geq\int_{\omega_r}\biggl\{\frac{k}{2}(\Div u)^2\Div X-k(\Div u)^2-Ck r|\nabla u|^2\biggr\}\mathrm{d}x,\\[0.5em]
&\geq -Cr\int_{\omega_r}k(\Div u)^2\,\mathrm{d}x-Cr\int_{\omega_r}|\nabla u|^2\,\mathrm{d}x.
\end{align*}
Therefore we can find $C$ large enough so 
\begin{align*}
\frac{1}{4\e^2}\int_{\omega_r}\bigl(1-|u|^2\bigr)^2\,\mathrm{d}x-\frac{Cr}{2}\int_{\omega_r}(|\nabla u|^2+k(\Div u)^2)\,\mathrm{d}x\leq C\left[F_{\Div}(r)+\frac{r^2}{\e}\right].
\end{align*}
\end{proof}

\begin{proof}[Proof of Theorem \ref{thm:eta} - Divergence Case]
The case where $x_0\in\Omega$ and $\omega_{\e^{\beta}}(x_0)\cap\Gamma=\varnothing$ is shown in \cite[Lemma 2.3]{struwe1994asymptotic}. Therefore, it is sufficient to prove the result for $x_0\in\Gamma$. We begin by proving \eqref{ineq:ceta-divstrong}. Using the mean value theorem for integrals, there exists $r_{\e}\in(\e^{\gamma},\e^{\beta})$ such that 
\begin{equation*}
\eta\ln\frac{1}{\e}\geq\int_{\e^{\gamma}}^{\e^{\beta}}\frac{F_{\Div}(r)}{r}\,dr=F_{\Div}(r_{\e})\int_{\e^{\gamma}}^{\e^{\beta}}\frac{dr}{r}=F_{\Div}(r_{\e})(\gamma-\beta)\ln\frac{1}{\e}. 
\end{equation*}
Using $r=r_{\e}$ in \eqref{intboundstrong-divstrong}:
\begin{align*}
\frac{1}{4\e^2}\int_{\omega_{r_{\e}}(x_0)}\bigl(1-|u_{\e}|^2\bigr)^2\,\mathrm{d}x&\leq C\biggl[r_{\e}\int_{\omega_{r_{\e}}(x_0)}\biggl\{\frac{1}{2}|\nabla u|^2+\frac{k}{2}(\Div u)^2\biggr\}\,\mathrm{d}x+F_{\Div}(r_{\e})+\frac{r_{\e}^2}{\e}\biggr]\\[0.5em]
&\leq C\biggl[2\e^{3/4}\eta|\ln\e|+\frac{\eta}{\gamma-\beta}+4\e^{2\beta-1}\biggr]\\
&\leq \tilde{C}\eta
\end{align*}
and so \eqref{ineq:ceta-divstrong} holds since $r_{\e}>\e^{\gamma}$. Inequality \eqref{ineq:half-divstrong} can be obtained using a contradiction argument. Assume there is some $x_1\in \omega_{\e^{\gamma}}(x_0)$ such that $|u(x_1)|<1/2$. Employing Lemma \ref{lem:a-priori-div}, we have 
\begin{equation*}
|u(x)-u(x_1)|\leq\frac{C_2}{\e}|x-x_1|
\end{equation*}
and so there exists $c>0$ such that $|u(x)|\leq 3/4$ for all $x\in \omega_{c\e}(x_1)\subset \omega_{\e^{\gamma}}(x_0)$. Then
\begin{equation*}
\tilde{C}\eta\geq \frac{1}{4\e^2}\int_{\omega_{\e^{\gamma}}(x_0)}\bigl(1-|u|^2\bigr)^2 \mathrm{d}x\geq \frac{1}{4\e^2}\int_{\omega_{c\e}(x_1)}\bigl(1-|u|^2\bigr)^2 \mathrm{d}x\geq c'
\end{equation*}
for $c'>0$ independent of $\eta$ and $\e$. Taking $\eta$ smaller if necessary gives the contradiction.
\end{proof}


\subsection{Curl Penalization}
We proceed as in Section \ref{sec:diveta} to obtain analogous estimates and results.\\

Define the quantities
\begin{align*}
\ecur(u)&\coloneqq\frac{1}{2}|\nabla u|^2+\frac{k}{2}(\Cur u)^2+\frac{1}{4\e^2}\bigl(1-|u|^2\bigr)^2,\\
F_{\Cur}(r)&\coloneqq{}F_{\Cur}(r;x_0,u,\e)=r\int_{\partial B_r(x_0)\cap\Omega}\ecur(u)\,\mathrm{d}s
\end{align*}
where $x_0\in\overline{\Omega}$. The corresponding Pokhozhaev-type identity and integral estimate for curl-penalized solutions are as follows.\\
\begin{prop}\label{prop:Pohcurl}
Let $\psi\in C^{2}(\Omega;\mathbb{R}^2)$. If $u$ is a solution of \eqref{eq:ELcurlstrong}, then
\begin{multline}\label{PohCurl}
\int_{\partial\omega_r}\left\{\ecur(u)\langle\psi,n\rangle-\langle\partial_nu+k(\Cur u)\tau,\psi\cdot\nabla u\rangle\right\}\mathrm{d}s\\=\int_{\omega_r}\biggl\{\ecur(u)\Div \psi-\sum_{j,l}\psi_{x_j}^{l}\langle\partial_{x_j}u,\partial_{x_{l}}u\rangle-k(\Cur u)\sum_{i=1}^2(-1)^{i-1}\langle\partial_{x_i}\psi,\nabla u^{3-i}\rangle\biggr\}\mathrm{d}x.
\end{multline}
\end{prop}
\begin{proof}
We take the inner product of equation \eqref{eq:ELcurlstrong} with $\psi\cdot\nabla u$ and integrate by parts over $\omega_r$. When $k=0$, \cite[Lemma 4.2]{alama2015weak} applies and so \eqref{Pohkzero} holds in this case. To include the curl term, the quantity $-k\int_{\omega_r}\langle \psi\cdot\nabla u,\nabla^{\perp}\Cur u\rangle\,\mathrm{d}x$ is integrated by parts and combined with \eqref{Pohkzero} as in the proof of Proposition \ref{prop:Pohdiv}.
\end{proof}
\begin{lem}\label{lem:bigestimate-curlstrong}
Let $x_0\in\overline{\Omega}$. There exists constants $C(k)>0$ and $r_0(k)>0$ such that for $\e\in(0,1)$ and $r\in(0,r_0)$ we have\\[0.5em]
\mbox{(1)}\ If $x_0\in\Omega$ and $\overline{\omega_r(x_0)}\cap\Gamma=\varnothing$,
\begin{align}\label{intinterior-curlstrong}
\frac{1}{4\e^2}\int_{\omega_r}\bigl(1-|u|^2\bigr)^2\,\mathrm{d}x\leq C\biggl[r\int_{\omega_r}\frac{1}{2}\left(|\nabla u|^2+k(\Cur u)^2\right) dx+F_{\Cur}(r)\biggr],
\end{align}
\mbox{(2)}\ If $x_0\in\Gamma$,
\begin{align}\label{intboundstrong-curlstrong}
\begin{split}
\frac{1}{4\e^2}\int_{\omega_r}&\bigl(1-|u|^2\bigr)^2\,\mathrm{d}x\leq C\biggl[r\int_{\omega_r}\frac{1}{2}\left(|\nabla u|^2+ k(\Cur u)^2\right) dx+F_{\Cur}(r)+\frac{r^2}{\e}\biggr].
\end{split}
\end{align}
\end{lem}
\begin{proof}[Proof of Lemma \ref{lem:bigestimate-curlstrong}] As before, we proceed in two steps.\\[0.5em]
\underline{Step 1: $x_0\in\Omega$}\\[0.5em]
To obtain \eqref{intinterior-curlstrong}, we take $X$ to be the vector field as described in Lemma \ref{lem:bigestimate-divstrong} and set $\psi=X$ in equation (\ref{PohCurl}).\\[0.5em]
\underline{Estimates Along $\partial\omega_r$}:\\[0.5em]
The lefthand side of (\ref{PohCurl}) can be written as the sum of integrals $I_4+I_5+I_6$ where
\begin{align*}
I_4&\coloneqq\int_{\partial\omega_r}\biggl\{\frac{1}{2}|\nabla u|^2\langle X,n\rangle-\langle\partial_nu,X\cdot\nabla u\rangle\biggr\}\mathrm{d}s,\\[0.5em]
I_5&\coloneqq\int_{\partial\omega_r}\biggl\{\frac{k}{2}(\Cur u)^2\langle X,n\rangle-\langle k(\Cur u)\tau,X\cdot\nabla u\rangle\biggr\}\mathrm{d}s,\\[0.5em]
I_6&\coloneqq\frac{1}{4\e^2}\int_{\partial\omega_r}\bigl(1-|u|^2\bigr)^2\langle X,n\rangle\,\mathrm{d}s.
\end{align*}
Integrals $I_4$ and $I_6$ are handled identically to that of $I_1$ and $I_3$ of Lemma \ref{lem:bigestimate-divstrong}, so we estimate $I_5$ only. Using Cauchy-Schwarz:
\begin{align*}
I_5&=r\int_{\partial\omega_r}\biggl\{\frac{k}{2}(\Cur u)^2-k\langle(\Cur u)\tau,\partial_nu\rangle\biggr\}\mathrm{d}s\\[0.5em]
&\leq r\int_{\partial\omega_r}\biggl\{\frac{k}{2}(\Cur u)^2+k|(\Cur u)\tau||\partial_nu|\biggr\}\mathrm{d}s\\[0.5em]
&\leq Cr\int_{\partial\omega_r}\biggl\{\frac{k}{2}(\Cur u)^2+|\nabla u|^2\biggr\}\mathrm{d}s.
\end{align*}
Again we can find $C>0$ large enough so that
\begin{align*}
I_4+I_5+I_6\leq Cr\int_{\partial\omega_r}\frac{1}{2}\biggl\{|\nabla u|^2+k(\Cur u)^2+\frac{1}{2\e^2}\bigl(1-|u|^2\bigr)^2\biggr\}\mathrm{d}s=CF_{\Cur}(r).
\end{align*}
\underline{Estimates in $\omega_r$}:\\[0.5em]
The righthand side of (\ref{PohCurl}) can be written as the sum $J_4+J_5+J_6$ where
\begin{align*}
J_4&\coloneqq\int_{\omega_r}\biggl\{\frac{1}{2}|\nabla u|^2\Div X-\sum_{j,l}X_{x_j}^{l}\langle\partial_{x_j}u,\partial_{x_{l}}u\rangle\biggr\}\mathrm{d}x,\\[0.5em]
J_5&\coloneqq\int_{\omega_r}\biggl\{\frac{k}{2}(\Cur u)^2\Div X-k(\Cur u)\sum_{i=1}^2(-1)^{i-1}\langle\partial_{x_i}X,\nabla u^{3-i}\rangle\biggr\}\mathrm{d}x,\\[0.5em]
J_6&\coloneqq\frac{1}{4\e^2}\int_{\omega_r}\bigl(1-|u|^2\bigr)^2\Div X\,\mathrm{d}x.
\end{align*}
As before, the integrals $J_4$ and $J_6$ are estimated exactly like $J_1$ and $J_3$ of Lemma \ref{lem:bigestimate-divstrong} respectively. For $J_5$ observe that 
\begin{align*}
\sum_{i=1}^2(-1)^{i-1}\langle\partial_{x_i}X,\nabla u^{3-i}\rangle=\Cur u.
\end{align*}
This paired with the fact that $\Div X>2-r$ we obtain
\begin{align*}
J_5&\geq \int_{\omega_r}\biggl\{k(\Cur u)^2-\frac{k r}{2}(\Cur u)^2-k(\Cur u)\sum_{i=1}^2(-1)^{i-1}\langle\partial_{x_i}X,\nabla u^{3-i}\rangle\biggr\}\mathrm{d}x\\[0.5em]
&=-r\int_{\omega_r}\frac{k}{2}(\Cur u)^2\,\mathrm{d}x
\end{align*}
and therefore
\begin{equation*}
\frac{1}{4\e^2}\int_{\omega_r}\bigl(1-|u|^2\bigr)^2\,\mathrm{d}x-r\int_{\omega_r}\biggl\{\frac{1}{2}|\nabla u|^2+\frac{k}{2}(\Cur u)^2\biggr\}\mathrm{d}x\leq J_4+J_5+J_6\leq CF_{\Cur}(r).
\end{equation*}
\underline{Step 2: $x_0\in\Gamma$}\\[0.5em]
To prove \eqref{intboundstrong-curlstrong}, we take $X$ as in \eqref{Xnormal}--\eqref{DXbound} and preform estimates on $\Gamma_r$ and $\partial B_r(x_0)\cap\Omega$ separately.\\[0.5em]
\underline{Estimates Along $\Gamma_r$}:\\[0.5em]
We decompose and analyze the inner product
\begin{equation*}
\langle \partial_nu+k(\Cur u)\tau,X_{\tau}\partial_{\tau}u\rangle.
\end{equation*}
Using \eqref{eq:frenetdecomp} and the strong boundary data from \eqref{eq:ELcurlstrong},
\begin{equation*}
\partial_nu+k(\Cur u)\tau=\partial_nu_nn+u_{\tau}\partial_n\tau.
\end{equation*}
Using \eqref{eq:frenetdecomp} once more,
\begin{equation*}
\langle \partial_nu+k(\Cur u)\tau,X_{\tau}\partial_{\tau}u\rangle=X_{\tau}\langle \partial_nu_nn,-\kappa u_{\tau}n+\partial_{\tau}u_{\tau}\tau\rangle=-X_{\tau}\kappa u_{\tau}\partial_nu_n.
\end{equation*}
Using the assumed pointwise bounds and the smoothness of the boundary, the above satisfies
\begin{equation*}
|\langle \partial_nu+k(\Cur u)\tau,X_{\tau}\partial_{\tau}u\rangle|\leq |X_{\tau}|\frac{c}{\e}
\end{equation*}
where $c$ independent of $\e$, and therefore
\begin{equation*}
\left|\int_{\Gamma_r}\langle\partial_nu+k(\Cur u)\tau,X\cdot\nabla u\rangle\,\mathrm{d}s\right|\leq \int_{\Gamma_r}|X_{\tau}|\frac{c}{\e}\,\mathrm{d}s\leq \frac{Cr^2}{\e}
\end{equation*}
for $C>0$ independent of $\e$.\\\\
\underline{Estimates Along $\partial B_r(x_0)\cap\Omega$}:\\[0.5em]
The lefthand side of (\ref{PohCurl}) along $\partial B_r(x_0)\cap\Omega$ can be written $I_4+I_5+I_6$ where
\begin{align*}
I_4&=\int_{\partial B_r(x_0)\cap\Omega}\biggl\{\frac{1}{2}|\nabla u|^2\langle X,n\rangle-\langle\partial_nu,X\cdot\nabla u\rangle\biggr\}\mathrm{d}s,\\[0.5em]
I_5&=\int_{\partial B_r(x_0)\cap\Omega}\biggl\{\frac{k}{2}(\Cur u)^2\langle X,n\rangle-\langle k(\Cur u)\tau,X\cdot\nabla u\rangle\biggr\}\mathrm{d}s,\\[0.5em]
I_6&=\frac{1}{4\e^2}\int_{\partial B_r(x_0)\cap\Omega}\bigl(1-|u|^2\bigr)^2\langle X,n\rangle\,\mathrm{d}s.
\end{align*}
Since the estimates for $I_4$ and $I_6$ are identical to the estimates of $I_1$ and $I_3$ above, only $I_5$ is considered. Writing $X=\langle X,n\rangle n+\langle X,\tau\rangle\tau$, the second term in $I_5$ can be written
\begin{align*}
-\langle k(\Cur u)\tau,X\cdot\nabla u\rangle=-k\langle X,n\rangle\langle(\Cur u)\tau,\partial_nu\rangle-k\langle X,\tau\rangle\langle(\Cur u)\tau,\partial_{\tau}u\rangle.
\end{align*}
Using (\ref{Xbound}), Cauchy-Schwarz:
\begin{align*}
I_5&\leq k Cr\int_{\partial B_r(x_0)\cap\Omega}\biggl\{\frac{1}{2}(\Cur u)^2+|(\Cur u)\tau||\partial_nu|+|(\Cur u)\tau||\partial_{\tau}u|\biggr\}\mathrm{d}s\\[0.5em]
&\leq k Cr\int_{\partial B_r(x_0)\cap\Omega}\biggl\{\frac{1}{2}(\Cur u)^2+\frac{1}{2}(\Cur u)^2+\frac{1}{2}|\partial_nu|^2+\frac{1}{2}(\Cur u)^2+\frac{1}{2}|\partial_{\tau}u|^2\biggr\}\mathrm{d}s\\[0.5em]
&=Cr\int_{\partial B_r(x_0)\cap\Omega}\biggl\{\frac{3k}{2}(\Cur u)^2+\frac{k}{2}|\nabla u|^2\biggr\}\mathrm{d}s.
\end{align*}
Thus, for $C>0$ large enough we have
\begin{equation*}
I_4+I_5+I_6\leq Cr\int_{\partial B_r(x_0)\cap\Omega}\frac{1}{2}\biggl\{|\nabla u|^2+k(\Cur u)^2+\frac{1}{2\e^2}\bigl(1-|u|^2\bigr)^2\biggr\}\mathrm{d}s=CF_{\Cur}(r)
\end{equation*}
and therefore
\begin{equation*}
\int_{\partial\omega_r}\left\{\ecur(u)\langle X,n\rangle-\langle\partial_nu+k(\Cur u)\tau,X\cdot\nabla u\rangle\right\}\mathrm{d}s\leq C\left[F_{\Cur}(r)+\frac{r^2}{\e}\right].
\end{equation*}
\underline{Estimates in $\omega_r$}:\\[0.5em]
The righthand side of (\ref{PohCurl}) can be written as the sum of three integrals $J_4+J_5+J_6$ where
\begin{align*}
J_4&=\int_{\omega_r}\biggl\{\frac{1}{2}|\nabla u|^2\Div X-\sum_{j,l}X_{x_j}^{l}\langle\partial_{x_j}u,\partial_{x_{l}}u\rangle\biggr\}\mathrm{d}x,\\[0.5em]
J_5&=\int_{\omega_r}\biggl\{\frac{k}{2}(\Cur u)^2\Div X-k(\Cur u)\sum_{i=1}^2(-1)^{i-1}\langle\partial_{x_i}X,\nabla u^{3-i}\rangle\biggr\}\mathrm{d}x,\\[0.5em]
J_6&=\frac{1}{4\e^2}\int_{\omega_r}\bigl(1-|u|^2\bigr)^2\Div X\,\mathrm{d}x.
\end{align*}
The integrals $J_4$ and $J_6$ are estimated precisely like $J_1$ and $J_3$ so $J_5$ is the only integral that needs to be treated. Focusing on the last term in $J_5$, we add and subtract $(\Cur u)^2$ to obtain
\begin{align*}
(\Cur u)\sum_{i=1}^2&(-1)^{i-1}\langle\partial_{x_i}X,\nabla u^{3-i}\rangle\\
&=(\Cur u)^2+\sum_{i=1}^2\left[(X_{x_i}^i-1)(u_{x_i}^{3-i})^2-(X_{x_i}^i-1)u_{x_2}^1u_{x_1}^2\right]\\
&\qquad\qquad\qquad+X_{x_1}^2u_{x_1}^2u_{x_2}^2+X_{x_2}^1u_{x_1}^1u_{x_2}^1-X_{x_2}^1u_{x_1}^1u_{x_1}^2-X_{x_1}^2u_{x_2}^1u_{x_2}^2\\[0.5em]
&\leq(\Cur u)^2+\sum_{i=1}^2\left[|X_{x_i}^i-1||u_{x_i}^{3-i}|^2+|X_{x_i}^i-1||u_{x_2}^1||u_{x_1}^2|\right]\\
&\qquad\qquad\qquad+|X_{x_1}^2||u_{x_1}^2||u_{x_2}^2|+|X_{x_2}^1||u_{x_1}^1||u_{x_2}^1|+|X_{x_2}^1||u_{x_1}^1||u_{x_1}^2|\\
&\qquad\qquad\qquad\quad+|X_{x_1}^2||u_{x_2}^1||u_{x_2}^2|.
\end{align*}
By \eqref{DXbound} we have $|X_{x_j}^l-\delta_{jl}|\leq Cr$ on $\omega_r$ and applying Young's inequality to each of the derivative pairs, there exists $C>0$ such that
\begin{align*}
k(\Cur u)\sum_{i=1}^2(-1)^{i-1}\langle\partial_{x_i}X,\nabla u^{3-i}\rangle\leq k(\Cur u)^2+Ck r|\nabla u|^2.
\end{align*}
Finally, since $\Div X\geq 2-2Cr$ as before,
\begin{align*}
J_5&\geq\int_{\omega_r}\biggl\{\frac{k}{2}(\Cur u)^2\Div X-k(\Cur u)^2-Ck r|\nabla u|^2\biggr\}\mathrm{d}x\\[0.5em]
&\geq \int_{\omega_r}\biggl\{k(\Cur u)^2-k(\Cur u)^2-Ck r(\Cur u)^2-Ck r|\nabla u|^2\biggr\}\mathrm{d}x\\[0.5em]
&\geq -Cr\int_{\omega_r}k(\Cur u)^2\,\mathrm{d}x-Cr\int_{\omega_r}|\nabla u|^2\,\mathrm{d}x.
\end{align*}
Therefore by taking $C>0$ large enough
\begin{equation*}
\frac{1}{4\e^2}\int_{\omega_r}\bigl(1-|u|^2\bigr)^2\,\mathrm{d}x-Cr\int_{\omega_r}\frac{k}{2}(\Cur u)^2\,dx -Cr\int_{\omega_r}\frac{1}{2}|\nabla u|^2\,\mathrm{d}x\leq C\left[F_{\Cur}(r)+\frac{r^2}{\e}\right].
\end{equation*}
\end{proof}

\begin{proof}[Proof of Theorem \ref{thm:eta} - Curl Case] By the mean value theorem for integrals, there exists $r_{\e}\in(\e^{\gamma},\e^{\beta})$ such that $F_{\Cur}(r_{\e})\leq \eta(\gamma-\beta)^{-1}$. 
Then, replacing Lemma~\ref{lem:a-priori-div} by Theorem~\ref{thm:a-priori-curl} and the identity from Lemma~\ref{lem:bigestimate-divstrong} by Lemma~\ref{lem:bigestimate-curlstrong}, the proof proceeds as for the divergence penalized case.
\end{proof}


\section{Lower Bounds for the Energies, Convergence, \& Singular Sets}\label{sec:lowerbound}
The purpose of this section is to derive a lower bound for $\Ediv$ and $\Ecur$ on a set away from where $|u|$ is small. Then, we use these lower bounds to deduce convergence of minimizers to a limiting map as $\e\to 0$. Finally, we show that boundary vortices do not exist in the case of divergence penalization. \\

Recall the Ginzburg--Landau energy
\begin{equation*}
    G_{\e}(u)
    \ = \
    \frac{1}{2}\int_{\Omega}\Bigl(|\grad u|^2+\frac{1}{2\e^2}\bigl(1-|u|^2\bigr)^2\Bigr)\mathrm{d}x.
\end{equation*}
Then by the simple fact that $\Ediv$ and $\Ecur$ are bounded below by $G_{\e}$ for all $u\in\Ht$, any lower bound for $G_{\e}$ is a lower bound for $\Ediv$ and $\Ecur$.\\

Define the `bad set'
\begin{equation}\label{def:badseteps}
S_{\e}\coloneq\left\{x\in\overline{\Omega}:|u_{\e}(x)|<\frac{1}{2}\right\}.
\end{equation}
One of the fundamental consequences of $\eta$-compactness is the ability to establish the existence of a finite collection of disjoint balls, each with radius of order $\e$ that cover $S_{\e}$. We state the ball covering lemma here and refer the reader to \cite[Lemma 4.4]{ABG} for its proof, see also \cite{struwe1994asymptotic} and \cite{bbh}. It is important to note that the arguments used there are based on energy estimates that are directly applicable to our context.
\begin{lem}\label{lem:epcover}
There exists $\tilde{N}\in\mathbb{N}$ depending only on $\Omega$, a constant $\lambda>1$ independent of $\e$ and points $\{p_{\e,i}\}_{i=1}^{I_{\e}}\subset S_{\e}\cap\Omega$, $\{q_{\e,j}\}_{j=1}^{J_{\e}}\subset S_{\e}\cap\Gamma$ such that
\begin{enumerate}
    \item $I_{\e}+J_{\e}\leq\tilde{N}$,
    \item $S_{\e}\subset\bigcup_{i=1}^{I_{\e}}B_{\lambda\e}(p_{\e,i})\cup\bigcup_{j=1}^{J_{\e}}B_{\lambda\e}(q_{\e,j})$,
    \item $\{B_{\lambda\e}(p_{\e,i}),B_{\lambda\e}(q_{\e,j})\}_{i,j}$ are mutually disjoint with centers satisfying 
    \begin{equation*}
        |p_{\e,i}-p_{\e,j}|,\ |q_{\e,i}-q_{\e,j}|,\ |p_{\e,i}-q_{\e,j}|>8\lambda\e,
    \end{equation*}
    \item $\overline{B_{\lambda \e}(p_{\e,i})}\cap\Gamma=\varnothing$ for all $i=1,\ldots, I_{\e}$.
\end{enumerate}
\end{lem}
The balls forming the covering in Lemma \ref{lem:epcover} will be referred to as `bad balls'.

As in \cite[Lemma 3.2]{struwe1994asymptotic} one can obtain the following corollary, stating that the number of vortices and their location stabilize in the limit.

\begin{cor}\label{cor:sigcover}
    For any sequence of $\e\to 0$ there is a subsequence $\e_n\to0$, a constant $\sigma_0 > 0$ and a finite collection of points $\{p_{i}\}_{i=1}^{I}\subset\Omega$, $\{q_{j}\}_{j=1}^{J}\subset \Gamma$ (with $I$ and $J$ independent of $\e$) such that for any $0<\sigma<\sigma_0$ and for all $n\in\mathbb{N}$, the union of mutually disjoint sets 
\begin{equation}\label{def:sigmaballs}
\mathcal{S}_{\sigma}\coloneqq\bigcup_{i=1}^{I}B_{\sigma}(p_i)\cup\bigcup_{j=1}^J\left(B_{\sigma}(q_j)\cap\Omega\right)
\end{equation} 
    cover $S_{\e_n}$.
\end{cor} 
The final main result of this section before proving Theorem \ref{thm:maindiv} is as follows.
\begin{lem}\label{lem:sigmalb}
    Suppose $\e_n$ is the subsequence taken in Corollary \ref{cor:sigcover} and let $d_i$ and $D_j$ be the degree and boundary index for $u_{\e_n}$ about $\partial B_{\sigma}(p_i)$ and $\partial B_{\sigma}(q_j)\cap\Omega$ respectively. Then there exists a constant $C$, independent of $\e_n$ and $\sigma$, such that
    \begin{equation}\label{eq:Ssigmalowerbound}
        \Edivn(u_{\e_n};\mathcal{S}_{\sigma})\geq \pi\left(\sum_{i=1}^I|d_i|+\frac{1}{2}\sum_{j=1}^J|D_j|\right)|\ln \e_n|-C.
    \end{equation}
    The same result holds for $\Ecurn(u_{\e_n})$.
\end{lem}
To obtain the desired lower bound for $\Edivn$ and $\Ecurn$, we begin by analyzing the local cost of a vortex. Let
\begin{equation*}
    A_{r,R}(x_0)\coloneqq\omega_R(x_0)\setminus \overline{\omega_r(x_0)},\quad 0<r<R
\end{equation*}
denote an annular region.
\begin{lem}\label{lem:locallb}
    Suppose $x_0\in\overline{\Omega}$ and assume $1/2\leq |u|\leq 1$ in $A_{r,R}(x_0)$. Additionally, suppose $u_n=0$ on $\overline{ A_{r,R}(x_0)}\cap\Gamma$ and that there is some number $K$ satisfying
    \begin{align*}
\Ediv(u)\leq K|\ln\e|+K\quad \mbox{and}\quad \frac{1}{\e^2}\int_{\omega_{\e^{\gamma}}(x_0)}\left(1-|u|^2\right)^{2}\mathrm{d}x\leq K,
    \end{align*}
    where $\e^{\gamma}$ is as in Theorem \ref{thm:eta}. Then there is a constant $C$ independent of $\e$ such that:
\begin{enumerate}[(i)]
\item If $B_R(x_0)\subset\Omega$ and $d\coloneqq\deg(u;\partial B_r(x_0))\neq 0$,
\begin{equation}\label{ineq:lower1}
\int_{A_{r,R}(x_0)}|\nabla u|^2\,\mathrm{d}x \geq 2d^2\pi \ln\biggl(\frac{R}{r}\biggr)-C.
\end{equation}
\item If $x_0\in\Gamma$ and $D\neq 0$ is the boundary index on $\partial B_r\cap\Omega$,
\begin{equation}\label{ineq:lower2}
\int_{A_{r,R}(x_0)}|\nabla u|^2\,\mathrm{d}x \geq D^2\pi\ln\biggl(\frac{R}{r}\biggr)-C
\end{equation}
\end{enumerate}
The same results hold with $\Ediv$ replaced by $\Ecur$.
\end{lem}
\begin{proof}
Since $\Ediv$ and $\Ecur$ bound $G_{\e}$ from above, \cite[Theorem 5.2]{abv} gives conclusions (i) and (ii).
\end{proof}
The result of Lemma~\ref{lem:sigmalb} is in the spirit of \cite{sandier1998lower} and \cite{jerrard1999lower}. We follow the approach derived in \cite{sandier1998lower} in which a technique involving summing properties of the logarithm is employed. In short, the proof follows a two-step process where balls containing $S_{\e}$ are expanded and fused in such a way that the energy on these balls can be estimated from below while preserving the $|\ln \e|$ order of the local lower bound and simultaneously accounting for the sum of winding numbers about these balls. However, these techniques utilize Dirichlet boundary data, which we do not have, and so the proof must be modified.
\begin{proof}[Proof of Lemma \ref{lem:sigmalb}]
By Lemma \ref{lem:epcover}, $S_{\e}$ can be covered by a union of disjoint balls of radii of order $\e$. For a given $\e>0$, there will be $I_{\e}$ interior bad balls and $J_{\e}$ boundary bad balls. Suppose $d_{i,\e}$ denotes the degree of $u_{\e}$ on an interior bad ball and $D_{j,\e}$ denotes the boundary index of $u_{\e}$ on a boundary bad ball. Then \cite[Proposition 4.1]{abv} states 
\begin{equation}\label{eq:badballsum}
    \sum_{i=1}^{I_{\e}}d_{i,\e}+\frac{1}{2}\sum_{j=1}^{J_{\e}}D_{j,\e}=1.
\end{equation}
Moreover, by \cite[Lemma 4.2]{abv}, if $\mathcal{B}_R\cap\Omega$ is a boundary ball (of boundary index $\mathscr{D}$) enclosing $\mathcal{I}$ interior bad balls and $\mathcal{J}$ boundary bad balls, then 
\begin{equation}\label{eq:badballsum2}
2\sum_{i=1}^{\mathcal{I}}d_{i,\e}+\sum_{j=1}^{\mathcal{J}}D_{j,\e}=\mathscr{D}.
\end{equation}
Using these addition properties in conjunction with a modification of the aforementioned ball expansion/fusion argument (see \cite[Lemma 5.3]{abv} \& \cite[Lemma 7.1]{ABG}), which depends on Lemma~\ref{lem:locallb}, we obtain inequality \eqref{eq:Ssigmalowerbound} for $\Edivn$ and $\Ecurn$.
\end{proof}
\begin{proof}[Proof of Theorem \ref{thm:maindiv}]
By the upper bound \eqref{eq:Eupperbound} and lower bound \eqref{eq:Ssigmalowerbound} for $\Edivn$ and $\Ecurn$, we find that each degree $d_i$ and boundary index $D_j$ are uniformly bounded in $\e$ and therefore can be taken to be constant along a subsequence $\e_n\to 0$. It can be shown using the given energy estimates and \eqref{eq:badballsum} that
\begin{equation}\label{eq:finalsum}
\sum_{i=1}^I|d_i|+\frac{1}{2}\sum_{j=1}^J|D_j|=\sum_{i=1}^Id_i+\frac{1}{2}\sum_{j=1}^JD_j=1
\end{equation}\\
and so \eqref{eq:Ssigmalowerbound} is seen as the matching lower bound: 
\begin{equation}\label{eq:lastlowerbound}
    \Edivn(u_{\e_{n}};\mathcal{S}_{\sigma}),\ \Ecurn(u_{\e_{n}};\mathcal{S}_{\sigma})\geq \pi|\ln\e_{n}|-C'.
\end{equation}
Looking outside of $\mathcal{S}_{\sigma}$ on the punctured domain
\begin{equation}\label{def:Omegasigma}
\Omega_{\sigma}\coloneqq\Omega\setminus\overline{\mathcal{S}_{\sigma}}
\end{equation} 
the bounds \eqref{eq:Eupperbound} and \eqref{eq:lastlowerbound} can be combined (see \cite[Corollary 5.4]{abv}) to give the existence of a constant $C$ independent of $\e$ and $\sigma$ such that
\begin{equation}\label{eq:upperboundOmegasigma}
\Edivn(u_{\e};\Omega_{\sigma}),\ \Ecurn(u_{\e};\Omega_{\sigma})\leq \pi |\ln\sigma|+C.
\end{equation}
It is clear from \eqref{eq:Ssigmalowerbound} that all $\sigma$-balls constituting $\mathcal{S}_{\sigma}$ which satisfy $d_i=D_j=0$ do not contribute significant energy. Therefore, the associated balls can be seen to belong to the set where $u_{\e_n}$ converges. By relabeling the approximate vortices if necessary, we define
\begin{equation*}
\Sigma \coloneqq\{p_1,\ldots,p_I\}\cup\{q_1,\ldots,q_J\}
\end{equation*}
to be the collection of all $\sigma$-ball centers with nonzero degree or boundary index.\\

Using a refinement in the lower bound on annular regions outside $\mathcal{S}_{\sigma}$, it can also be shown that each degree and boundary index must be equal to unity. Thus, the summing condition \eqref{eq:finalsum} forces either $\Sigma=\{p\}\subset\Omega$ or $\Sigma=\{q_1,q_2\}\subset\Gamma$.\\

Upon taking an appropriate subsequence $\sigma_n\to 0$, \eqref{eq:upperboundOmegasigma} in combination of methods found in \cite{bbh,struwe1994asymptotic} allows us to conclude that $u_{\e_n}\w u_0$ weakly in $H^1_{loc}(\overline{\Omega}\setminus\Sigma;\mathbb{R}^2)$ as $\e_n\to 0$ where $u_0\in H^1(\overline{\Omega}\setminus\Sigma;\mathbb{S}^1)$.\\

The last step is to show $\Sigma=\{p\}\subset\Omega$ in the case of divergence penalization.\\

Suppose in order to derive a contradiction that there exists a boundary vortex at $x_0\in\partial\Omega$.
Without loss of generality, we assume that $x_0=0$.
Recall that for $r_2>r_1>0$ we use the notation $A_{r_1,r_2}=(B_{r_2}\setminus \overline{B_{r_1}})\cap\Omega$.

Following \cite[Section 4]{abv}, we write $u_\varepsilon(r,\theta)=|u_\varepsilon(r,\theta)|e^{i(\theta+\psi_\varepsilon(r,\theta))}$ on $A_{r_1,r_2}$, i.e.
\begin{align*}
u_\varepsilon(r,\theta)
\ &= \
|u_\varepsilon(r,\theta)|
\begin{pmatrix}
\cos(\theta + \psi_\varepsilon(r,\theta)) \\ 
\sin(\theta + \psi_\varepsilon(r,\theta))
\end{pmatrix}
\, ,
\end{align*}
where $\psi_\varepsilon$ is an $H^1-$function satisfying $\psi_\varepsilon\in \pi\mathbb{Z}$ on $\overline{A_{r_1,r_2}}\cap\partial\Omega$.
We find that
\begin{align*}
|\nabla u_\varepsilon|^2
\ &= \
|\nabla|u_\varepsilon||^2
+
|u_\varepsilon|^2 |\nabla(\theta+\psi_\varepsilon)|^2
\, ,
\end{align*}
which implies, using $|\nabla\theta|=\frac{1}{r}$ and the boundary condition for $\psi_\varepsilon$ that
\begin{align}\label{bdryvtx:grad_decomp}
\int_{A_{r_1,r_2}}|\nabla u_\varepsilon|^2 \dx x
\ &= \
\int_{A_{r_1,r_2}} \Big(|\nabla|u_\varepsilon||^2
+
|u_\varepsilon|^2 |\nabla(\theta+\psi_\varepsilon)|^2\Big) \dx x \nonumber \\
\ &= \
\int_{A_{r_1,r_2}} \Big(|\nabla|u_\varepsilon||^2
+
|u_\varepsilon|^2 |\nabla\theta|^2 
+
|u_\varepsilon|^2 |\nabla\psi_\varepsilon|^2
+
2\frac{|u_\varepsilon|^2}{r}\ee_\theta\cdot\nabla\psi_\varepsilon
\Big) \dx x  \\
\ &= \
\int_{A_{r_1,r_2}} \Big(|\nabla|u_\varepsilon||^2
+
\frac{|u_\varepsilon|^2}{r^2} 
+
|u_\varepsilon|^2 |\nabla\psi_\varepsilon|^2\Big) \dx x 
+
\mathcal{O}(|r_2-r_1|) \nonumber
\, .
\end{align}
Furthermore,
\begin{align}\label{bdryvtx:div_decomp}
|\mathrm{div}(u_\varepsilon)|^2
\ &= \
\Big(\frac{u_\varepsilon}{|u_\varepsilon|}\cdot\nabla|u_\varepsilon|\Big)^2
+
\Big(u_\varepsilon^\perp\cdot\nabla(\theta+\psi_\varepsilon)\Big)^2
+
2\Big(\frac{u_\varepsilon}{|u_\varepsilon|}\cdot\nabla|u_\varepsilon|\Big)\Big(u_\varepsilon^\perp\cdot\nabla(\theta+\psi_\varepsilon)\Big)
\, .
\end{align}

Let $r_\varepsilon>0$ such that $r_\varepsilon\to 0$, $|\ln(r_\varepsilon)|/|\ln(\varepsilon)|\to 0$.
Let $\theta_\varepsilon(r),\Theta_\varepsilon(r)$ be the minimal and maximal angle parametrizing $\overline{A_{\varepsilon,r_\varepsilon}}\cap\partial\Omega$ in polar coordinates.
Then the regularity of the boundary implies that $\Theta_\varepsilon-\theta_\varepsilon\to \pi$ uniformly in $r$ as $\varepsilon\to 0$.
We find, using $|\nabla\theta|=\frac{1}{r}$, that 
\begin{align*}
\int_{A_{\varepsilon,r_\varepsilon}} |u_\varepsilon|^2 |\nabla\theta|^2 \dx x
\ &= \
\int_{\varepsilon}^{r_\varepsilon} \int_{\theta_\varepsilon} ^{\Theta_\varepsilon} (|u_\varepsilon|^2 - 1)\frac{1}{r^2}r \dx\theta \dx r
+
\int_{\varepsilon}^{r_\varepsilon}\int_{\theta_\varepsilon} ^{\Theta_\varepsilon} \frac{1}{r^2}r \dx\theta \dx r
\, .
\end{align*}
The second term gives
\begin{align}\label{bdryvtx:grad_theta_pi_ln}
\int_{\varepsilon}^{r_\varepsilon}\int_{\theta_\varepsilon} ^{\Theta_\varepsilon} \frac{1}{r^2}r \dx r\dx\theta 
\ &= \
\int_{\varepsilon}^{r_\varepsilon} \frac{\Theta_\varepsilon(r)-\theta_\varepsilon(r)}{r} \dx r 
\ = \
\pi |\ln(\varepsilon)|
+
o(|\ln(\varepsilon)|)
\, ,
\end{align}
where we used that $\Theta_\varepsilon-\theta_\varepsilon\to\pi$ uniformly and $|\ln(r_\varepsilon)|/|\ln(\varepsilon)|\to 0$.
For the first term, using Cauchy-Schwarz inequality and the uniform bound on the potential \eqref{eq:uniform_bound_L4}, we find that
\begin{align*}
\bigg|
\int_{\varepsilon}^{r_\varepsilon}\int_{\theta_\varepsilon} ^{\Theta_\varepsilon} (|u_\varepsilon|^2 - 1)\frac{1}{r} \dx\theta\dx r
\bigg|
\ &\leq \
\int_{\varepsilon}^{r_\varepsilon}\int_{\theta_\varepsilon} ^{\Theta_\varepsilon} ||u_\varepsilon|^2 - 1|\frac{1}{\sqrt{r}}\, \frac{1}{\sqrt{r}} \dx\theta\dx r \\
\ &\leq \
\bigg(\int_{\varepsilon}^{r_\varepsilon}\int_{\theta_\varepsilon} ^{\Theta_\varepsilon} ||u_\varepsilon|^2 - 1|^2\frac{1}{\varepsilon^2} r\dx\theta\dx r\bigg)^\frac12
\bigg(\int_\varepsilon^{r_\varepsilon}\frac{\Theta_\varepsilon(r)-\theta_\varepsilon(r)}{r}\dx r\bigg)^\frac12 \\
\ &= \
o(|\ln(\varepsilon)|)
\, .
\end{align*}
Since $\int_\Omega |\nabla u_\varepsilon|^2 + |\mathrm{div}(u_\varepsilon)|^2 + \frac{1}{4\varepsilon^2}(1-|u_\varepsilon|^2)^2 \dx x\leq \pi|\ln(\varepsilon)| + C$ and by \eqref{bdryvtx:grad_decomp} and \eqref{bdryvtx:grad_theta_pi_ln} it holds
\begin{align}\label{bdryvtx:bound_o_ln}
\int_{A_{\varepsilon,r_\varepsilon}} 
\Big(
|\nabla|u_\varepsilon||^2
+
|u_\varepsilon|^2|\nabla\psi_\varepsilon|^2
+
|\mathrm{div}(u_\varepsilon)|^2
+ 
\frac{1}{4\varepsilon^2}(1-|u_\varepsilon|^2)^2
\Big)
\dx x
\ &= \
o(|\ln(\varepsilon)|)
\, .
\end{align}
Since all terms in the above are non-negative, this means in particular that $\int_{A_{\varepsilon,r_\varepsilon}}  |\mathrm{div}(u_\varepsilon)|^2\dx x = o(|\ln(\varepsilon)|)$, or using \eqref{bdryvtx:div_decomp},
\begin{align*}
\int_{A_{\varepsilon,r_\varepsilon}} 
\bigg(
\Big(\frac{u_\varepsilon}{|u_\varepsilon|}\cdot\nabla|u_\varepsilon|\Big)^2
+
\Big(u_\varepsilon^\perp\cdot\nabla(\theta+\psi_\varepsilon)\Big)^2
+
2\Big(\frac{u_\varepsilon}{|u_\varepsilon|}\cdot\nabla|u_\varepsilon|\Big)\Big(u_\varepsilon^\perp\cdot\nabla(\theta+\psi_\varepsilon)\Big)
\bigg)
\dx x
\ &= \
o(|\ln(\varepsilon)|)
\, .
\end{align*}
But \eqref{bdryvtx:bound_o_ln} implies that
\begin{align*}
\int_{A_{\varepsilon,r_\varepsilon}} 
\Big(\frac{u_\varepsilon}{|u_\varepsilon|}\cdot\nabla|u_\varepsilon|\Big)^2
\dx x
\ &= \
o(|\ln(\varepsilon)|)
\end{align*}
and by Cauchy-Schwarz, \eqref{bdryvtx:bound_o_ln} and the global energy bound
\begin{align*}
\bigg|\int_{A_{\varepsilon,r_\varepsilon}} 
\Big(\frac{u_\varepsilon}{|u_\varepsilon|}\cdot\nabla|u_\varepsilon|\Big)\Big(u_\varepsilon^\perp\cdot\nabla(\theta+\psi_\varepsilon)\Big)
\dx x
\bigg|
\ &\leq \
\int_{A_{\varepsilon,r_\varepsilon}} 
|\nabla|u_\varepsilon||\, |u_\varepsilon||\nabla(\theta+\psi_\varepsilon)|
\dx x \\
\ &\leq \
\bigg(\int_{A_{\varepsilon,r_\varepsilon}} 
|\nabla|u_\varepsilon||^2
\dx x \bigg)^\frac12
\bigg(\int_{A_{\varepsilon,r_\varepsilon}} 
|u_\varepsilon|^2|\nabla(\theta+\psi)|^2
\dx x \bigg)^\frac12 \\
\ &\leq \
o(\sqrt{|\ln(\varepsilon)|}) \sqrt{|\ln(\varepsilon)|}
\ = \
o(|\ln(\varepsilon)|)
\, .
\end{align*}
We compute
\begin{align*}
u_\varepsilon^\perp\cdot\nabla\theta
\ &= \
|u_\varepsilon|
\begin{pmatrix}
-\sin(\theta + \psi_\varepsilon) \\ 
\cos(\theta + \psi_\varepsilon)
\end{pmatrix}
\cdot
\frac{1}{r}
\begin{pmatrix}
-\sin(\theta) \\ \cos(\theta)
\end{pmatrix}
\ = \
\frac{|u_\varepsilon|}{r}\cos(\psi_\varepsilon)
\, .
\end{align*}
Thus we infer that
\begin{align*}
o(|\ln(\varepsilon)|)
\ &= \
\int_{A_{\varepsilon,r_\varepsilon}} 
\Big(u_\varepsilon^\perp\cdot\nabla(\theta+\psi_\varepsilon)\Big)^2
\dx x \\
\ &= \
\int_{A_{\varepsilon,r_\varepsilon}} 
\bigg(
\frac{|u_\varepsilon|^2}{r^2}\cos^2(\psi_\varepsilon)
+ |u_\varepsilon^\perp\cdot\nabla\psi_\varepsilon|^2
+ 2\frac{|u_\varepsilon|}{r}\cos(\psi_\varepsilon) \, (u_\varepsilon^\perp\cdot\nabla\psi_\varepsilon)
\bigg)
\dx x
\, .
\end{align*}
For the second and third term we observe that
\begin{align*}
\int_{A_{\varepsilon,r_\varepsilon}} 
|u_\varepsilon^\perp\cdot\nabla\psi_\varepsilon|^2
\dx x
\ &\leq \
\int_{A_{\varepsilon,r_\varepsilon}} 
|u_\varepsilon|^2|\nabla\psi_\varepsilon|^2
\dx x
\ = \
o(|\ln(\varepsilon)|)
\, , \\
\bigg|\int_{A_{\varepsilon,r_\varepsilon}} 
\frac{|u_\varepsilon|}{r}\cos(\psi_\varepsilon) \, (u_\varepsilon^\perp\cdot\nabla\psi_\varepsilon)
\dx x\bigg|
\ &\leq \
\bigg(\int_{A_{\varepsilon,r_\varepsilon}} 
\frac{|u_\varepsilon|^2}{r^2}
\dx x\bigg)^\frac12
\bigg(\int_{A_{\varepsilon,r_\varepsilon}} 
|u_\varepsilon|^2|\nabla\psi_\varepsilon|^2
\dx x\bigg)^\frac12
\ = \
o(|\ln(\varepsilon)|)
\, .
\end{align*}
Thus also the first term must be $o(|\ln(\varepsilon)|)$.
Hence we end up with
\begin{align*}
o(|\ln(\varepsilon)|)
\ &= \
\int_{A_{\varepsilon,r_\varepsilon}} 
\Big(
|\nabla|u_\varepsilon||^2
+
|u_\varepsilon|^2|\nabla\psi_\varepsilon|^2
+
\frac{|u_\varepsilon|^2}{r^2}\cos^2(\psi_\varepsilon)
+ 
\frac{1}{4\varepsilon^2}(1-|u_\varepsilon|^2)^2
\Big)
\dx x \\
\ &\geq \
\int_{A_{\varepsilon,r_\varepsilon}} 
\Big(
\frac{1}{r^2}|\partial_\theta|u_\varepsilon||^2
+
\frac{|u_\varepsilon|^2}{r^2}|\partial_\theta\psi_\varepsilon|^2
+
\frac{|u_\varepsilon|^2}{r^2}\cos^2(\psi_\varepsilon)
+ 
\frac{1}{4\varepsilon^2}(1-|u_\varepsilon|^2)^2
\Big)
\dx x \\
\ &= \
\int_\varepsilon^{r_\varepsilon}
\int_{\mathbb{S}^1\cap\frac{1}{r}\Omega}
\Big(
\frac{1}{r}|\partial_\theta|u_\varepsilon||^2
+
\frac{|u_\varepsilon|^2}{r}|\partial_\theta\psi_\varepsilon|^2
+
\frac{|u_\varepsilon|^2}{r}\cos^2(\psi_\varepsilon)
+ 
\frac{r}{4\varepsilon^2}(1-|u_\varepsilon|^2)^2
\Big)
\dx \mathcal{H}^1 \dx r
\, .
\end{align*}
Because of the $o(|\ln(\varepsilon)|)$ bound, there exists a set $S_\varepsilon\subset (\varepsilon,r_\varepsilon)$ with $|S_\varepsilon|=\frac12|r_\varepsilon-\varepsilon|$ such that
\begin{align*}
\frac{r}{\varepsilon^2}\int_{\mathbb{S}^1\cap\frac1r\Omega}  (1-|u_\varepsilon|^2)^2 \dx\mathcal{H}^1
\ &\leq \
\frac{o(|\ln(\varepsilon)|)}{|S_\varepsilon|}
\qquad
\text{for all }
r\in S_\varepsilon
\, .
\end{align*}
This allows us to conclude that for any radius in $S_\varepsilon$, the modulus of $u_\varepsilon$ must converge to $1$ in $L^4$.
Using this and the fact that all integrands are non-negative, there must exist a radius $s_\varepsilon\in S_\varepsilon$ such that
\begin{align*}
\int_{\mathbb{S}^1\cap \frac{1}{s_\varepsilon}\Omega} 
\Big(
|\partial_\theta|u_\varepsilon||^2
+
|u_\varepsilon|^2|\partial_\theta\psi_\varepsilon|^2
+
|u_\varepsilon|^2\cos^2(\psi_\varepsilon)
\Big)
\dx \mathcal{H}^1
\ \to \
0
\text{ as }
\varepsilon\to 0
\, .
\end{align*}
Using the first term and passing to a subsequence if necessary, we find that $|u_\varepsilon|\to 1$ uniformly on $\mathbb{S}^1\cap \frac{1}{s_\varepsilon}\Omega$.
Thus, using the second and third term, we deduce that $\psi_\varepsilon$ converges to a constant $\psi_0$ and since $\cos^2(\psi_0)=0$ we conclude that $\psi_0\in \frac{\pi}{2} + \pi\mathbb{Z}$.
However, such $\psi_0$ violates the tangential boundary condition, which forces $\psi_0\in\pi\mathbb{Z}$. This concludes the proof of Theorem \ref{thm:maindiv}.
\end{proof}

\begin{rem}
Note that we could do the same computation with slight modifications for curl instead of divergence without reaching a contradiction:
Since $|\mathrm{curl}(u_\varepsilon)|=|\mathrm{div}(u_\varepsilon^\perp)|$ in 2D, we would find at the end that $\sin^2(\psi_\varepsilon)$ must be converging zero, thus giving $\psi_0=\pi\mathbb{Z}$.
These values are compatible with the boundary condition.
\end{rem}

\appendix
\appendixpage
\addappheadtotoc

\section{Rigorous Calculations}\label{app:calculations}
In this Appendix we provide additional details of the rigorous calculation
provided in Subsection \ref{subsec:glueing}.
More specifically, we provide, for the interested reader, an extended
discussion of the calculations required to verify glueing the PDEs.\\

\noindent{}{\it Lemma \ref{lem:pdeext} calculations}\\[3pt]

Inserting the identities discussed before the statement of the lemma
in equation \eqref{eq:ELcurlweak} we obtain
\begin{align*}
    &\int_{\mathcal{U}_{j}}\!{}
    \Bigl[\nabla\widetilde{\psi}_{j}^{-1}(x)^{T}
    \nabla{}(\widetilde{u}_{\e})_{\tau}\bigl(\widetilde{\psi}_{j}^{-1}(x)\bigr)\Bigr]\cdot
    \Bigl[\nabla\widetilde{\psi}_{j}^{-1}(x)^{T}
    \nabla{}\widetilde{v}_{\tau}\bigl(\widetilde{\psi}_{j}^{-1}(x)\bigr)\Bigr]\\
    &+\int_{\mathcal{U}_{j}}\!{}
    \Bigl[\nabla\widetilde{\psi}_{j}^{-1}(x)^{T}
    \nabla{}(\widetilde{u}_{\e})_{n}\bigl(\widetilde{\psi}_{j}^{-1}(x)\bigr)\Bigr]\cdot
    \Bigl[\nabla\widetilde{\psi}_{j}^{-1}(x)^{T}
    \nabla{}\widetilde{v}_{n}\bigl(\widetilde{\psi}_{j}^{-1}(x)\bigr)\Bigr]\\
    &-\int_{\mathcal{U}_{j}}\!{}
    \Bigl[\nabla\widetilde{\psi}_{j}^{-1}(x)^{T}
    \nabla{}(\widetilde{u}_{\e})_{\tau}\bigl(\widetilde{\psi}_{j}^{-1}(x)\bigr)\Bigr]\cdot
    \biggl[\frac{\kappa\bigl((\widetilde{\psi}_{j}^{-1})^{1}(x)\bigr)
    \widetilde{v}_{n}(\widetilde{\psi}_{j}^{-1}(x))}
    {1-(\widetilde{\psi}_{j}^{-1})^{2}(x)\kappa\bigl((\widetilde{\psi}_{j}^{-1})^{1}(x)\bigr)}
    \tau\bigl(\widetilde{\psi}_{j}^{-1}(x)\bigr)\biggr]\\
    &-\int_{\mathcal{U}_{j}}\!{}
    \Bigl[\nabla\widetilde{\psi}_{j}^{-1}(x)^{T}
    \nabla{}(\widetilde{u}_{\e})_{n}\bigl(\widetilde{\psi}_{j}^{-1}(x)\bigr)\Bigr]\cdot
    \biggl[\frac{\kappa\bigl((\widetilde{\psi}_{j}^{-1})^{1}(x)\bigr)
    \widetilde{v}_{\tau}\bigl(\widetilde{\psi}_{j}^{-1}(x)\bigr)}
    {1-(\widetilde{\psi}_{j}^{-1})^{2}(x)\kappa\bigl((\widetilde{\psi}_{j}^{-1})^{1}(x)\bigr)}
    n\bigl(\widetilde{\psi}_{j}^{-1}(x)\bigr)\biggr]\\
    &+k\int_{\mathcal{U}_{j}}
    \Bigl[
    \partial_{y_{2}}(\widetilde{u}_{\e})_{\tau}\bigl(\widetilde{\psi}_{j}^{-1}(x)\bigr)
    \partial_{n}(\widetilde{\psi}_{j}^{-1})^{2}(x)
    -\partial_{y_{1}}(\widetilde{u}_{\e})_{n}\bigl(\widetilde{\psi}_{j}^{-1}(x)\bigr)
    \partial_{\tau}(\widetilde{\psi}_{j}^{-1})^{1}(x)
    \Bigr]
    \Bigl[\partial_{y_{2}}\widetilde{v}_{\tau}\bigl(\widetilde{\psi}_{j}^{-1}(x)\bigr)
    \partial_{n}(\widetilde{\psi}_{j}^{-1})^{2}(x)\Bigr]\\
    &-k\int_{\mathcal{U}_{j}}
    \Bigl[
    \partial_{y_{2}}(\widetilde{u}_{\e})_{\tau}\bigl(\widetilde{\psi}_{j}^{-1}(x)\bigr)
    \partial_{n}(\widetilde{\psi}_{j}^{-1})^{2}(x)
    -\partial_{y_{1}}(\widetilde{u}_{\e})_{n}\bigl(\widetilde{\psi}_{j}^{-1}(x)\bigr)
    \partial_{\tau}(\widetilde{\psi}_{j}^{-1})^{1}(x)
    \Bigr]
    \Bigl[\partial_{y_{1}}\widetilde{v}_{n}\bigl(\widetilde{\psi}_{j}^{-1}(x)\bigr)
    \partial_{\tau}(\widetilde{\psi}_{j}^{-1})^{1}(x)\Bigr]\\
    &+k\int_{\mathcal{U}_{j}}\!{}
    \biggl[\frac{\kappa\bigl((\widetilde{\psi}_{j}^{-1})^{1}(x)\bigr)(\widetilde{u}_{\e})_{\tau}\bigl(\widetilde{\psi}_{j}^{-1}(x)\bigr)}{1-(\widetilde{\psi}_{j}^{-1})^{2}(x)\kappa\bigl((\widetilde{\psi}_{j}^{-1})^{1}(x)\bigr)}\biggr]
    \Bigl[\partial_{y_{2}}\widetilde{v}_{\tau}\bigl(\widetilde{\psi}_{j}^{-1}(x)\bigr)\partial_{n}(\widetilde{\psi}_{j}^{-1})^{2}(x)\Bigr]\\
    &-k\int_{\mathcal{U}_{j}}\!{}
    \biggl[\frac{\kappa\bigl((\widetilde{\psi}_{j}^{-1})^{1}(x)\bigr)(\widetilde{u}_{\e})_{\tau}\bigl(\widetilde{\psi}_{j}^{-1}(x)\bigr)}{1-(\widetilde{\psi}_{j}^{-1})^{2}(x)\kappa\bigl((\widetilde{\psi}_{j}^{-1})^{1}(x)\bigr)}\biggr]
    \Bigl[\nabla[\widetilde{v}_{n}\bigl(\widetilde{\psi}_{j}^{-1}(x)\bigr)]\cdot{}\tau\bigl((\widetilde{\psi}_{j}^{-1})^{1}(x)\bigr)\Bigr]\\
    &+k\int_{\mathcal{U}_{j}}\!{}
    \biggl[\frac{\kappa\bigl((\widetilde{\psi}_{j}^{-1})^{1}(x)\bigr)(\widetilde{u}_{\e})_{\tau}\bigl(\widetilde{\psi}_{j}^{-1}(x)\bigr)}{1-(\widetilde{\psi}_{j}^{-1})^{2}(x)\kappa\bigl((\widetilde{\psi}_{j}^{-1})^{1}(x)\bigr)}\biggr]
    \biggl[\frac{\kappa\bigl((\widetilde{\psi}_{j}^{-1})^{1}(x)\bigr)\widetilde{v}_{\tau}\bigl(\widetilde{\psi}_{j}^{-1}(x)\bigr)}{1-(\widetilde{\psi}_{j}^{-1})^{2}(x)\kappa\bigl((\widetilde{\psi}_{j}^{-1})^{1}(x)\bigr)}\biggr]\\
    &+k\int_{\mathcal{U}_{j}}\!{}
    \Bigl[
    \partial_{y_{2}}(\widetilde{u}_{\e})_{\tau}\bigl(\widetilde{\psi}_{j}^{-1}(x)\bigr)
    \partial_{n}(\widetilde{\psi}_{j}^{-1})^{2}(x)
    -\partial_{y_{1}}(\widetilde{u}_{\e})_{n}\bigl(\widetilde{\psi}_{j}^{-1}(x)\bigr)
    \partial_{\tau}(\widetilde{\psi}_{j}^{-1})^{1}(x)
    \Bigr]\biggl[\frac{\kappa\bigl((\widetilde{\psi}_{j}^{-1})^{1}(x)\bigr)\widetilde{v}_{\tau}\bigl(\widetilde{\psi}_{j}^{-1}(x)\bigr)}{1-(\widetilde{\psi}_{j}^{-1})^{2}(x)\kappa\bigl((\widetilde{\psi}_{j}^{-1})^{1}(x)\bigr)}\biggr]
    \\
    =&\int_{\mathcal{U}_{j}}\!{}
    \frac{u_{\e}\cdot{}v}{\varepsilon^{2}}(1-|u_{\e}|^{2})
\end{align*}
for $v\in{}H_{T}^{1}(\Omega;\mathbb{R}^{2})$.
Integrating by parts and using that $\widetilde{v}_{n}=0$ on
$\Gamma=\partial\Omega$ we find that
\begin{align*}
    -k\int_{\mathcal{U}_{j}}\!{}
    \biggl[\frac{\kappa\bigl((\widetilde{\psi}_{j}^{-1})^{1}(x)\bigr)(\widetilde{u}_{\e})_{\tau}\bigl(\widetilde{\psi}_{j}^{-1}(x)\bigr)}{1-(\widetilde{\psi}_{j}^{-1})^{2}(x)\kappa\bigl((\widetilde{\psi}_{j}^{-1})^{1}(x)\bigr)}\biggr]
    \Bigl[\nabla[\widetilde{v}_{n}\bigl(\widetilde{\psi}_{j}^{-1}(x)\bigr)]\cdot{}\tau\bigl((\widetilde{\psi}_{j}^{-1})^{1}(x)\bigr)\Bigr]\\
    =k\int_{\mathcal{U}_{j}}\!{}
    \Div\biggl(\frac{\kappa\bigl((\widetilde{\psi}_{j}^{-1})^{1}(x)\bigr)(\widetilde{u}_{\e})_{\tau}\bigl(\widetilde{\psi}_{j}^{-1}(x)\bigr)}{1-(\widetilde{\psi}_{j}^{-1})^{2}(x)\kappa\bigl((\widetilde{\psi}_{j}^{-1})^{1}(x)\bigr)}
    \tau\bigl((\widetilde{\psi}_{j}^{-1})^{1}(x)\bigr)\biggr)
    \widetilde{v}_{n}\bigl(\widetilde{\psi}_{j}^{-1}(x)\bigr).
\end{align*}
Combining these calculations gives that for $j=1,2,\ldots,N$
we have
\begin{align*}
    &\int_{\mathcal{U}_{j}}\!{}
    \biggl[
    \frac{\partial_{y_{1}}(\widetilde{u}_{\e})_{\tau}\bigl(\widetilde{\psi}_{j}^{-1}(x)\bigr)
    \partial_{y_{1}}\widetilde{v}_{\tau}\bigl(\widetilde{\psi}_{j}^{-1}(x)\bigr)}
    {\bigl(1-(\widetilde{\psi}_{j}^{-1})^{2}(x)\kappa\bigl((\widetilde{\psi}_{j}^{-1})^{1}(x)\bigr)\bigr)^{2}}
    +
    \partial_{y_{2}}(\widetilde{u}_{\e})_{\tau}\bigl(\widetilde{\psi}_{j}^{-1}(x)\bigr)
    \partial_{y_{2}}\widetilde{v}_{\tau}\bigl(\widetilde{\psi}_{j}^{-1}(x)\bigr)
    \biggr]
    \\
    &+k\int_{\mathcal{U}_{j}}
    \biggl[
    \partial_{y_{2}}(\widetilde{u}_{\e})_{\tau}\bigl(\widetilde{\psi}_{j}^{-1}(x)\bigr)
    +\frac{\partial_{y_{1}}(\widetilde{u}_{\e})_{n}\bigl(\widetilde{\psi}_{j}^{-1}(x)\bigr)}{1-(\widetilde{\psi}_{j}^{-1})^{2}(x)\kappa\bigl((\widetilde{\psi}_{j}^{-1})^{1}(x)\bigr)}
    \biggr]
    \Bigl[\partial_{y_{2}}\widetilde{v}_{\tau}\bigl(\widetilde{\psi}_{j}^{-1}(x)\bigr)\Bigr]
    \\
    &+k\int_{\mathcal{U}_{j}}\!{}
    \biggl[\frac{\kappa\bigl((\widetilde{\psi}_{j}^{-1})^{1}(x)\bigr)(\widetilde{u}_{\e})_{\tau}\bigl(\widetilde{\psi}_{j}^{-1}(x)\bigr)}{1-(\widetilde{\psi}_{j}^{-1})^{2}(x)\kappa\bigl((\widetilde{\psi}_{j}^{-1})^{1}(x)\bigr)}\biggr]
    \Bigl[\partial_{y_{2}}\widetilde{v}_{\tau}\bigl(\widetilde{\psi}_{j}^{-1}(x)\bigr)\Bigr]
    \\
    =&\int_{\mathcal{U}_{j}}\!{}
    \frac{(\widetilde{u}_{\e})_{\tau}\bigl(\widetilde{\psi}_{j}^{-1}(x)\bigr)
    \widetilde{v}_{\tau}\bigl(\widetilde{\psi}_{j}^{-1}(x)\bigr)}
    {\varepsilon^{2}}(1-|u_{\e}|^{2})
    +\int_{\mathcal{U}_{j}}\!{}F_{j,\tau}\bigl(x,u_{\e}(x),\nabla{}u_{\e}(x)\bigr)
    \widetilde{v}_{\tau}\bigl(\widetilde{\psi}_{j}^{-1}(x)\bigr)
\end{align*}
and
\begin{align*}
    &\int_{\mathcal{U}_{j}}\!{}
   \biggl[
    \frac{\partial_{y_{1}}(\widetilde{u}_{\e})_{n}\bigl(\psi_{j}^{-1}(x)\bigr)
    \partial_{y_{1}}\widetilde{v}_{n}\bigl(\psi_{j}^{-1}(x)\bigr)}
    {\bigl(1-(\psi_{j}^{-1})^{2}(x)\kappa\bigl((\psi_{j}^{-1})^{1}(x)\bigr)\bigr)^{2}}
    +
    \partial_{y_{2}}(\widetilde{u}_{\e})_{n}\bigl(\psi_{j}^{-1}(x)\bigr)
    \partial_{y_{2}}\widetilde{v}_{n}\bigl(\psi_{j}^{-1}(x)\bigr)
    \biggr]\\
    &+k\int_{\mathcal{U}_{j}}
    \biggl[
    \partial_{y_{2}}(\widetilde{u}_{\e})_{\tau}\bigl(\widetilde{\psi}_{j}^{-1}(x)\bigr)
    +\frac{\partial_{y_{1}}(\widetilde{u}_{\e})_{n}\bigl(\widetilde{\psi}_{j}^{-1}(x)\bigr)}{1-(\widetilde{\psi}_{j}^{-1})^{2}(x)\kappa\bigl((\widetilde{\psi}_{j}^{-1})^{1}(x)\bigr)}
    \biggr]
    \biggl[\frac{\partial_{y_{1}}\widetilde{v}_{n}\bigl(\widetilde{\psi}_{j}^{-1}(x)\bigr)}{1-(\widetilde{\psi}_{j}^{-1})^{2}(x)\kappa\bigl((\widetilde{\psi}_{j}^{-1})^{1}(x)\bigr)}
    \biggr]\\
    =&\int_{\mathcal{U}_{j}}\!{}
    \frac{(\widetilde{u}_{\e})_{n}\bigl(\psi_{j}^{-1}(x)\bigr)
    \widetilde{v}_{n}\bigl(\psi_{j}^{-1}(x)\bigr)}
    {\varepsilon^{2}}(1-|u_{\e}|^{2})
    +\int_{\mathcal{U}_{j}}\!{}F_{j,n}\bigl(x,u_{\e}(x),\nabla{}u_{\e}(x)\bigr)
    \widetilde{v}_{n}\bigl(\psi_{j}^{-1}(x)\bigr)
    \, ,
\end{align*}
where $F_{j,\tau}$ and $F_{j,n}$ are determined by the remaining integrands
from the previous computations.
Specifically, we define $F_{j,\tau}$ and $F_{j,n}$ implicitly by
\begin{align}
    F_{j,\tau}&\bigl(x,u_{\e}(x),\nabla{}u_{\e}(x)\bigr) \nonumber \\
    &\coloneqq
    -\Bigl[\nabla\widetilde{\psi}_{j}^{-1}(x)^{T}
    \nabla{}(\widetilde{u}_{\e})_{n}\bigl(\widetilde{\psi}_{j}^{-1}(x)\bigr)\Bigr]\cdot
    \biggl[\frac{\kappa\bigl((\widetilde{\psi}_{j}^{-1})^{1}(x)\bigr)}
    {1-(\widetilde{\psi}_{j}^{-1})^{2}(x)\kappa\bigl((\widetilde{\psi}_{j}^{-1})^{1}(x)\bigr)}
    n\bigl(\widetilde{\psi}_{j}^{-1}(x)\bigr)\biggr]\nonumber \\
    &\quad+k
    \biggl[\frac{\kappa\bigl((\widetilde{\psi}_{j}^{-1})^{1}(x)\bigr)(\widetilde{u}_{\e})_{\tau}\bigl(\widetilde{\psi}_{j}^{-1}(x)\bigr)}{1-(\widetilde{\psi}_{j}^{-1})^{2}(x)\kappa\bigl((\widetilde{\psi}_{j}^{-1})^{1}(x)\bigr)}\biggr]
    \biggl[\frac{\kappa\bigl((\widetilde{\psi}_{j}^{-1})^{1}(x)\bigr)}{1-(\widetilde{\psi}_{j}^{-1})^{2}(x)\kappa\bigl((\widetilde{\psi}_{j}^{-1})^{1}(x)\bigr)}\biggr]\label{app:def-Fjt}\\
    &\quad+k\Bigl[
    \partial_{y_{2}}(\widetilde{u}_{\e})_{\tau}\bigl(\widetilde{\psi}_{j}^{-1}(x)\bigr)
    \partial_{n}(\widetilde{\psi}_{j}^{-1})^{2}(x)
    -\partial_{y_{1}}(\widetilde{u}_{\e})_{n}\bigl(\widetilde{\psi}_{j}^{-1}(x)\bigr)
    \partial_{\tau}(\widetilde{\psi}_{j}^{-1})^{1}(x)
    \Bigr]\biggl[\frac{\kappa\bigl((\widetilde{\psi}_{j}^{-1})^{1}(x)\bigr)}{1-(\widetilde{\psi}_{j}^{-1})^{2}(x)\kappa\bigl((\widetilde{\psi}_{j}^{-1})^{1}(x)\bigr)}\biggr] \nonumber
\end{align}
and
\begin{align}
    F_{j,n}&\bigl(x,u_{\e}(x),\nabla{}u_{\e}(x)\bigr) \nonumber \\
    &\coloneqq
    -\Bigl[\nabla\widetilde{\psi}_{j}^{-1}(x)^{T}
    \nabla{}(\widetilde{u}_{\e})_{\tau}\bigl(\widetilde{\psi}_{j}^{-1}(x)\bigr)\Bigr]\cdot
    \biggl[\frac{\kappa\bigl((\widetilde{\psi}_{j}^{-1})^{1}(x)\bigr)}
    {1-(\widetilde{\psi}_{j}^{-1})^{2}(x)\kappa\bigl((\widetilde{\psi}_{j}^{-1})^{1}(x)\bigr)}
    \tau\bigl(\widetilde{\psi}_{j}^{-1}(x)\bigr)\biggr] \label{app:def-Fjn} \\
    &\quad+k\Div\biggl(\frac{\kappa\bigl((\widetilde{\psi}_{j}^{-1})^{1}(x)\bigr)(\widetilde{u}_{\e})_{\tau}\bigl(\widetilde{\psi}_{j}^{-1}(x)\bigr)}{1-(\widetilde{\psi}_{j}^{-1})^{2}(x)\kappa\bigl((\widetilde{\psi}_{j}^{-1})^{1}(x)\bigr)}
    \tau\bigl((\widetilde{\psi}_{j}^{-1})^{1}(x)\bigr)\biggr) \nonumber
\end{align}
respectively.
Notice that $F_{j,\tau}$ and $F_{j,n}$, for $j=1,2,\ldots,N$ satisfy
\eqref{eq:carathmapApp} since all constituent terms also satisfy this.

Next, we use the previous computations to show that the extension $U_{\e}$
satisfies a PDE of similar character.
Similar to \cite{bcs24}, we introduce a few preliminary
definitions and calculations before preceding.
First, we notice that we may reduce the problem to demonstrating the glueing for functions whose support is contained in
$\widetilde{\mathcal{U}}_{j}$ for $j=1,2,\ldots,N$ by appealing to a partition of unity subordinate to
$\{\widetilde{\mathcal{U}}_{j}\}_{j=1}^{N}$.
In particular, we may assume that our test function satisfies $v\in{}H_{0}^{1}(\widetilde{\mathcal{U}}_{j};\mathbb{R}^{2})$ for some
$j=1,2,\ldots,N$.
Next, for $j=1,2,\ldots,N$ we define the function
$\sigma_{j}\colon\widetilde{\mathcal{U}}_{j}\to\widetilde{\mathcal{U}}_{j}$ by
\begin{equation*}
    \sigma_{j}(x)\coloneqq
    \begin{cases}
        x& \text{for }x\in\widetilde{\mathcal{U}}_{j}\cap\Omega,\\
        R(x)&
        \text{for }x\in\widetilde{\mathcal{U}}_{j}\setminus\Omega,
    \end{cases}
\end{equation*}
where $R$ is as defined in \eqref{eq:BdMirror}.
This function will be used on the enlarged coordinate chart $\widetilde{\mathcal{U}}_{j}$ to swap between an exterior point to $\Omega$ into its interior counterpart while leaving interior points invariant.
This function will be imperative for the glueing argument as we
intend to make use of the structure of the extension as well as
information valid in the interior.
We also introduce, for $j=1,2,\ldots,N$,
the function $\mathfrak{R}_{j}$, defined on
$\widetilde{\mathcal{U}}_{j}\times\mathbb{R}^{2}$ by
\begin{equation*}
    \mathfrak{R}_{j}(x,z)\coloneqq
    \Bigl[\tau\bigl((\widetilde{\psi}_{j}^{-1})^{1}(x)\bigr)
    \tau\bigl((\widetilde{\psi}_{j}^{-1})^{1}(x)\bigr)^{T}-
    n\bigl((\widetilde{\psi}_{j}^{-1})^{1}(x)\bigr)
    n\bigl((\widetilde{\psi}_{j}^{-1})^{1}(x)\bigr)^{T}
    \Bigr]z,
\end{equation*}
which corresponds to the function $R$ written in the original coordinates.
We observe that
\begin{equation*}
    \nabla\sigma_{j}(\sigma_{j}(x))\nabla\sigma_{j}(\sigma_{j}(x))^{T}=
    \begin{cases}
        I_{2}&\text{for }x\in\widetilde{\mathcal{U}}_{j}\cap\Omega,\\
        \nabla\widetilde{\psi}_{j}\bigl(\widetilde{\psi}_{j}^{-1}(x)\bigr)
        \mathcal{M}(x)
        \nabla\widetilde{\psi}_{j}\bigl(\widetilde{\psi}_{j}^{-1}(x)\bigr)^{T}&\text{for }x\in\widetilde{\mathcal{U}}_{j}\setminus\Omega,
    \end{cases}
\end{equation*}
where
\begin{equation*}
    \mathcal{M}(x)\coloneqq
    \begin{pmatrix}
            \frac{1}{\bigl(1-|(\widetilde{\psi}_{j}^{-1})^{2}(x)|
            \kappa((\widetilde{\psi}_{j}^{-1})^{1}(x))\bigr)^{2}}& 0\\
            0& 1
    \end{pmatrix}.
\end{equation*}
Corresponding to this matrix we introduce the inner product defined for $x\in\widetilde{\mathcal{U}}_{j}$, for $j=1,2,\ldots,N$, by
\begin{equation*}
    \bigl<v,w\bigr>_{j}\coloneqq{}
    |\det(\nabla\sigma_{j}(x))|
    v^{T}\nabla\sigma_{j}\bigl(\sigma_{j}(x)\bigr)
    \nabla\sigma_{j}\bigl(\sigma_{j}(x)\bigr)^{T}w
\end{equation*}
where $v,w\in\mathbb{R}^{2}$.
This inner product will enter when verifying that the PDEs glue properly and to simplify the notation hereafter.
In addition, we introduce the following \emph{distortion factor}
\begin{equation*}
    \mathcal{D}_{j}(x)\coloneqq
    \begin{cases}
        1& \text{for }x\in\widetilde{\mathcal{U}}_{j}\cap\Omega,\\
        \frac{1-(\widetilde{\psi}_{j}^{-1})^{2}(x)\kappa((\widetilde{\psi}_{j}^{-1})^{1}(x))}
        {1+(\widetilde{\psi}_{j}^{-1})^{2}(x)\kappa((\widetilde{\psi}_{j}^{-1})^{1}(x))},
        &\text{for }x\in\widetilde{\mathcal{U}}_{j}\setminus\Omega,
    \end{cases}
\end{equation*}
that accounts for the deformation due to the change of variables from outside to the inside of the domain.
Using this distortion factor, we define
\begin{equation*}
    \Curj(w)(x)\coloneqq
    |\det(\nabla\sigma_{j}(x))|^{\frac{1}{2}}
    \biggl[
    \partial_{n}\Bigl(w(x)\cdot\tau\bigl((\widetilde{\psi}_{j}^{-1})^{1}(x)\bigr)\Bigr)
    -\mathcal{D}_{j}(x)\partial_{\tau}\Bigl(w(x)\cdot{}n\bigl((\widetilde{\psi}_{j}^{-1})^{1}(x)\bigr)\Bigr)
    \biggr]
\end{equation*}
for $x\in\widetilde{\mathcal{U}}_{j}$
for $j=1,2,\ldots,N$ and functions,
$w$, of appropriate regularity.
In particular, we use this notation to denote a quantity resembling curl
but including a compensating distortion factor.
Finally, for notational convenience, we let
$\mathcal{G}_{j}\colon\widetilde{\Omega}\times{}B_{r_{1}}(0)
\to{}M_{2\times2}(\mathbb{R})$, for $j=1,2,\ldots,N$,
denote the matrix-valued functions given by
\begin{equation*}
    \mathcal{G}_j(x,y)\coloneqq
    \begin{pmatrix}
        \frac{1}{(1-y_{2}\kappa(y_{1}))^{2}}& 0\\
        0& 1
    \end{pmatrix}
    \, ,
\end{equation*}
where $\kappa(y_1)$ denotes the curvature of $\Gamma$ at the point
satisfying $x=\widetilde{\psi}_j(y_1,0)$.
With this notation in place we are ready to glue the PDEs together.
Using the local tangent-normal decompositions for $U_{\e}$ and $v$ for
$x\in\widetilde{\mathcal{U}}_{j}\setminus\Omega$ we have that
\begin{align*}
    &\sum_{i=1}^{2}
    \Bigl[\nabla\sigma_{j}\bigl(\sigma_{j}(x)\bigr)
    \nabla\sigma_{j}\bigl(\sigma_{j}(x)\bigr)^{T}
    \nabla{}U_{\e}^{i}(x)\Bigr]\cdot{}\nabla{}v^{i}(x)\\
    =&\nabla(\widetilde{u}_{\e})_{\tau}\bigl(\widetilde{\psi}_{j}^{-1}(R(x))\bigr)^{T}
    \bigl[[I_{2}-2\mathbf{e}_{2}\mathbf{e}_{2}^{T}]\mathcal{G}_{j}(x,\widetilde{\psi}_{j}^{-1}(R(x)))\bigr]
    \nabla\widetilde{v}_{\tau}\bigl(\widetilde{\psi}_{j}^{-1}(x)\bigr)\\
    &-\nabla\widetilde({u}_{\e})_{n}\bigl(\widetilde{\psi}_{j}^{-1}(R(x))\bigr)^{T}
    \bigl[[I_{2}-2\mathbf{e}_{2}\mathbf{e}_{2}^{T}]\mathcal{G}_{j}(x,\widetilde{\psi}_{j}^{-1}(R(x)))\bigr]
    \nabla\widetilde{v}_{n}\bigl(\widetilde{\psi}_{j}^{-1}(x)\bigr)\\
    &+\kappa\bigl((\widetilde{\psi}_{j}^{-1})^{1}(x)\bigr)
    \widetilde{v}_{n}\bigl(\widetilde{\psi}_{j}^{-1}(x)\bigr)
    \nabla(\widetilde{u}_{\e})_{\tau}\bigl(\widetilde{\psi}_{j}^{-1}(R(x))\bigr)^{T}
    \bigl[[I_{2}-2\mathbf{e}_{2}\mathbf{e}_{2}^{T}]\mathcal{G}_{j}(x,\widetilde{\psi}_{j}^{-1}(R(x)))\bigr]
    \mathbf{e}_{1}\\
    &-\kappa\bigl((\widetilde{\psi}_{j}^{-1})^{1}(x)\bigr)
    \widetilde{v}_{\tau}\bigl(\widetilde{\psi}_{j}^{-1}(x)\bigr)
    \nabla(\widetilde{u}_{\e})_{n}\bigl(\widetilde{\psi}_{j}^{-1}(R(x))\bigr)^{T}
    \bigl[[I_{2}-2\mathbf{e}_{2}\mathbf{e}_{2}^{T}]\mathcal{G}_{j}(x,\widetilde{\psi}_{j}^{-1}(R(x)))\bigr]
    \mathbf{e}_{1}\\
    &-\sum_{i=1}^{2}
    \frac{\kappa\bigl((\widetilde{\psi}_{j}^{-1})^{1}(x)\bigr)
    n^{i}\bigl((\widetilde{\psi}_{j}^{-1})^{1}(x)\bigr)}
    {\bigl(1+(\widetilde{\psi}_{j}^{-1})^{2}(x)\kappa\bigl((\widetilde{\psi}_{j}^{-1})^{1}(x)\bigr)\bigr)^{2}}
    (\widetilde{u}_{\e})_{\tau}\bigl(\widetilde{\psi}_{j}^{-1}(R(x))\bigr)
    \Bigl[
    \nabla\widetilde{\psi}_{j}\bigl(\widetilde{\psi}_{j}^{-1}(x)\bigr)\mathbf{e}_{1}\Bigr]\cdot{}\nabla{}v(x)\\
    &+\sum_{i=1}^{2}
    \frac{\kappa\bigl((\widetilde{\psi}_{j}^{-1})^{1}(x)\bigr)
    \tau^{i}\bigl((\widetilde{\psi}_{j}^{-1})^{1}(x)\bigr)}
    {\bigl(1+(\widetilde{\psi}_{j}^{-1})^{2}(x)\kappa\bigl((\widetilde{\psi}_{j}^{-1})^{1}(x)\bigr)\bigr)^{2}}
    (\widetilde{u}_{\e})_{n}\bigl(\widetilde{\psi}_{j}^{-1}(R(x))\bigr)
    \Bigl[
    \nabla\widetilde{\psi}_{j}\bigl(\widetilde{\psi}_{j}^{-1}(x)\bigr)\mathbf{e}_{1}\Bigr]\cdot{}\nabla{}v(x).
\end{align*}
A similar expression holds over $\widetilde{\mathcal{U}}_{j}\cap\Omega$
due to the earlier calculation for $u_{\e}$.
Notice that the last two terms, containing a gradient of $v$ but not of
$u_{\e}$, will have no 
boundary terms after integrating by parts since
$(\widetilde{u}_{\e})_{n}=0$
on $\Gamma$, $v=0$ on $\partial\widetilde{\mathcal{U}}_{j}$, and since
the integrals involving $(\widetilde{u}_{\e})_{\tau}$ and $v$ will cancel after
an integration by parts.
Multiplying by $|\det(\nabla\sigma_{j}(x))|$, integrating over
$\widetilde{\mathcal{U}}_{j}\setminus\Omega$, applying the Change of
Variables Theorem with $\sigma_{j}$, and using that $\sigma_{j}^{2}=\text{Id}$
we obtain, for the highest order terms, that
\begin{align*}
    &
    \int_{\widetilde{\mathcal{U}}_{j}\setminus\Omega}\!{}|\det(\nabla\sigma_{j}(x))|
    \nabla(\widetilde{u}_{\e})_{\tau}\bigl(\widetilde{\psi}_{j}^{-1}(R(x))\bigr)^{T}
    \bigl[[I_{2}-2\mathbf{e}_{2}\mathbf{e}_{2}^{T}]\mathcal{G}_{j}(x,\widetilde{\psi}_{j}^{-1}(R(x)))\bigr]
    \nabla\widetilde{v}_{\tau}\bigl(\widetilde{\psi}_{j}^{-1}(x)\bigr)\\
    &-
    \int_{\widetilde{\mathcal{U}}_{j}\setminus\Omega}\!{}|\det(\nabla\sigma_{j}(x))|
    \nabla(\widetilde{u}_{\e})_{n}\bigl(\widetilde{\psi}_{j}^{-1}(R(x))\bigr)^{T}
    \bigl[[I_{2}-2\mathbf{e}_{2}\mathbf{e}_{2}^{T}]\mathcal{G}_{j}(x,\widetilde{\psi}_{j}^{-1}(R(x)))\bigr]
    \nabla\widetilde{v}_{n}\bigl(\widetilde{\psi}_{j}^{-1}(x)\bigr)\\
    =&\int_{\widetilde{\mathcal{U}}_{j}\cap\Omega}\!{}
    \nabla(\widetilde{u}_{\e})_{\tau}\bigl(\widetilde{\psi}_{j}^{-1}(x)\bigr)
    \bigl[[I_{2}-2\mathbf{e}_{2}\mathbf{e}_{2}^{T}]\mathcal{G}_{j}(x,\widetilde{\psi}_{j}^{-1}(x))\bigr]
    \nabla\widetilde{v}_{\tau}\bigl(\widetilde{\psi}_{j}^{-1}(R(x))\bigr)\\
    &-\int_{\widetilde{\mathcal{U}}_{j}\cap\Omega}\!{}
    \nabla(\widetilde{u}_{\e})_{n}\bigl(\widetilde{\psi}_{j}^{-1}(x)\bigr)
    \bigl[[I_{2}-2\mathbf{e}_{2}\mathbf{e}_{2}^{T}]\mathcal{G}_{j}(x,\widetilde{\psi}_{j}^{-1}(x))\bigr]
    \nabla\widetilde{v}_{n}\bigl(\widetilde{\psi}_{j}^{-1}(R(x))\bigr).
\end{align*}
Noticing that
\begin{equation*}
    \widetilde{\psi}_{j}^{-1}(R(x))=
    [I_{2}-\mathbf{e}_{2}\mathbf{e}_{2}^{T}]
    \widetilde{\psi}_{j}^{-1}(x)
\end{equation*}
we can rewrite the previous equation as
\begin{align*}
    &\int_{\widetilde{\mathcal{U}}_{j}\cap\Omega}\!{}
    \nabla(\widetilde{u}_{\e})_{\tau}\bigl(\widetilde{\psi}_{j}^{-1}(x)\bigr)
    \bigl[\mathcal{G}_{j}(x,\widetilde{\psi}_{j}^{-1}(x))\bigr]
    \nabla(\widetilde{v}^{R})_{\tau}\bigl(\widetilde{\psi}_{j}^{-1}(x)\bigr)\\
    &-\int_{\widetilde{\mathcal{U}}_{j}\cap\Omega}\!{}
    \nabla(\widetilde{u}_{\e})_{n}\bigl(\widetilde{\psi}_{j}^{-1}(x)\bigr)
    \bigl[\mathcal{G}_{j}(x,\widetilde{\psi}_{j}^{-1}(x))\bigr]
    \nabla(\widetilde{v}^{R})_{n}\bigl(\widetilde{\psi}_{j}^{-1}(x)\bigr)
\end{align*}
where for a function
$w\in{}H^{1}\bigl(B_{r_{1}}(0);\mathbb{R}^{2}\bigr)$ we have set
$w^{R}(y)\coloneqq{}w\bigl([I_{2}-2\mathbf{e}_{2}\mathbf{e}_{2}^{T}]y\bigr)$.
Next, we compute $\Curj$ in tangent-normal coordinates for
$x\in\widetilde{\mathcal{U}}_{j}$ to obtain
$\widetilde{\mathcal{U}}_{j}\setminus\Omega$, that
\begin{align*}
     \Curj(U_{\e})(x)=&
    |\det(\nabla\sigma_{j}(x))|^{\frac{1}{2}}\biggl[
    \partial_{n}\Bigl[
    (\widetilde{u}_{\e})_{\tau}\bigl(\widetilde{\psi}_{j}^{-1}(R(x))\bigr)\Bigr]
    +\mathcal{D}_{j}(x)
    \partial_{\tau}\Bigl[
    (\widetilde{u}_{\e})_{n}\bigl(\widetilde{\psi}_{j}^{-1}(R(x))\bigr)\Bigr]
    \biggr]\\
    =&|\det(\nabla\sigma_{j}(x))|^{\frac{1}{2}}\biggl[
    \partial_{y_{2}}(\widetilde{u}_{\e})_{\tau}\bigl(\widetilde{\psi}_{j}^{-1}(R(x))\bigr)+
    \frac{\partial_{y_{1}}(\widetilde{u}_{\e})_{n}\bigl(\widetilde{\psi}_{j}^{-1}(R(x))\bigr)}
    {1+(\widetilde{\psi}_{j}^{-1})^{2}(x)\kappa\bigl((\widetilde{\psi}_{j}^{-1})^{1}(x)\bigr)}
    \biggr].
\end{align*}
Note that since nothing specific to the minimizer was used in
this computation then a similar computation holds for $v$.
Using this we see, after integrating over
$\widetilde{\mathcal{U}}_{j}\setminus\Omega$, that
\begin{equation*}
    \int_{\widetilde{\mathcal{U}}_{j}\setminus\Omega}\!{}
    \Curj(U_{\e})\Curj(v)
    =\int_{\widetilde{\mathcal{U}}_{j}\setminus\Omega}\!{}
    |\det(\nabla\sigma_{j}(x))|^{\frac{1}{2}}\biggl[
    \partial_{y_{2}}(\widetilde{u}_{\e})_{\tau}\bigl(\widetilde{\psi}_{j}^{-1}(R(x))\bigr)+
    \frac{\partial_{y_{1}}(\widetilde{u}_{\e})_{n}\bigl(\widetilde{\psi}_{j}^{-1}(R(x))\bigr)}
    {1+(\widetilde{\psi}_{j}^{-1})^{2}(x)\kappa\bigl((\widetilde{\psi}_{j}^{-1})^{1}(x)\bigr)}
    \biggr]
    \Curj(v).
\end{equation*}
Next, introducing a change of variables using $\sigma_{j}$ leads to
\begin{align*}
    &\int_{\widetilde{\mathcal{U}}_{j}\cap\Omega}\!{}
    |\det(\nabla\sigma_{j}(x))|^{\frac{-1}{2}}
    \biggl[
    \partial_{y_{2}}(\widetilde{u}_{\e})_{\tau}\bigl(\widetilde{\psi}_{j}^{-1}(x))+
    \frac{\partial_{y_{1}}(\widetilde{u}_{\e})_{n}\bigl(\widetilde{\psi}_{j}^{-1}(x)\bigr)}
    {1-(\widetilde{\psi}_{j}^{-1})^{2}(x)\kappa\bigl((\widetilde{\psi}_{j}^{-1})^{1}(x)\bigr)}
    \bigr)
    \biggr]
    \Curj(v)(\sigma_{j}(x))\\
    =&-\int_{\widetilde{\mathcal{U}}_{j}\cap\Omega}\!{}
    \biggl[
    \partial_{y_{2}}(\widetilde{u}_{\e})_{\tau}\bigl(\widetilde{\psi}_{j}^{-1}(x)\bigr)+
    \frac{\partial_{y_{1}}(\widetilde{u}_{\e})_{n}\bigl(\widetilde{\psi}_{j}^{-1}(x)\bigr)}
    {1-(\widetilde{\psi}_{j}^{-1})^{2}(x)\kappa\bigl((\widetilde{\psi}_{j}^{-1})^{1}(x)\bigr)}
    \biggr]
    \partial_{y_{2}}\widetilde{v}_{\tau}\bigl(\widetilde{\psi}_{j}^{-1}(R(x))\bigr)
    \\
    &+\int_{\widetilde{\mathcal{U}}_{j}\cap\Omega}\!{}
    \biggl[
    \partial_{y_{2}}(\widetilde{u}_{\e})_{\tau}\bigl(\widetilde{\psi}_{j}^{-1}(x)\bigr)+
    \frac{\partial_{y_{1}}(\widetilde{u}_{\e})_{n}\bigl(\widetilde{\psi}_{j}^{-1}(x)\bigr)}
    {1-(\widetilde{\psi}_{j}^{-1})^{2}(x)\kappa\bigl((\widetilde{\psi}_{j}^{-1})^{1}(x)\bigr)}
    \biggr]
    \frac{\partial_{y_{1}}\widetilde{v}_{n}\bigl(\widetilde{\psi}_{j}^{-1}(R(x))\bigr)}{1-(\widetilde{\psi}_{j}^{-1})^{2}(x)\kappa\bigl((\widetilde{\psi}_{j}^{-1})^{1}(x)\bigr)}.
\end{align*}
Notice that we may rewrite this as
\begin{align*}
    -&\int_{\widetilde{\mathcal{U}}_{j}\cap\Omega}\!{}
    \biggl[
    \partial_{y_{2}}(\widetilde{u}_{\e})_{\tau}\bigl(\widetilde{\psi}_{j}^{-1}(x)\bigr)+
    \frac{\partial_{y_{1}}(\widetilde{u}_{\e})_{n}\bigl(\widetilde{\psi}_{j}^{-1}(x)\bigr)}
    {1-(\widetilde{\psi}_{j}^{-1})^{2}(x)\kappa\bigl((\widetilde{\psi}_{j}^{-1})^{1}(x)\bigr)}
    \biggr]
    \partial_{y_{2}}(\widetilde{v}^{R})_{\tau}\bigl(\widetilde{\psi}_{j}^{-1}(x)\bigr)
    \\
    &+\int_{\widetilde{\mathcal{U}}_{j}\cap\Omega}\!{}
    \biggl[
    \partial_{y_{2}}(\widetilde{u}_{\e})_{\tau}\bigl(\widetilde{\psi}_{j}^{-1}(x)\bigr)+
    \frac{\partial_{y_{1}}(\widetilde{u}_{\e})_{n}\bigl(\widetilde{\psi}_{j}^{-1}(x)\bigr)}
    {1-(\widetilde{\psi}_{j}^{-1})^{2}(x)\kappa\bigl((\widetilde{\psi}_{j}^{-1})^{1}(x)\bigr)}
    \biggr]
    \frac{\partial_{y_{1}}(\widetilde{v}^{R})_{\tau}\bigl(\widetilde{\psi}_{j}^{-1}(x)\bigr)}
    {1-(\widetilde{\psi}_{j}^{-1})^{2}(x)\kappa\bigl((\widetilde{\psi}_{j}^{-1})^{1}(x)\bigr)}.
\end{align*}
Finally, we notice that by a change of variables using
$\sigma_{j}$ we obtain
\begin{align*}
    &k\int_{\widetilde{\mathcal{U}}_{j}\setminus\Omega}\!{}
    |\det(\nabla\sigma_{j}(x))|\mathcal{D}_{j}(x)
    \biggl[\frac{\kappa\bigl((\widetilde{\psi}_{j}^{-1})^{1}(x)\bigr)(\widetilde{U}_{\e})_{\tau}\bigl(\widetilde{\psi}_{j}^{-1}(x)\bigr)}{1-(\widetilde{\psi}_{j}^{-1})^{2}(x)\kappa\bigl((\widetilde{\psi}_{j}^{-1})^{1}(x)\bigr)}\biggr]
    \Bigl[\partial_{y_{2}}\widetilde{v}_{\tau}\bigl(\widetilde{\psi}_{j}^{-1}(x)\bigr)\Bigr]
    \\
    =&k\int_{\widetilde{\mathcal{U}}_{j}\setminus\Omega}\!{}
    |\det(\nabla\sigma_{j}(x))|
    \biggl[\frac{\kappa\bigl((\widetilde{\psi}_{j}^{-1})^{1}(x)\bigr)(\widetilde{u}_{\e})_{\tau}\bigl(\widetilde{\psi}_{j}^{-1}(R(x))\bigr)}{1+(\widetilde{\psi}_{j}^{-1})^{2}(x)\kappa\bigl((\widetilde{\psi}_{j}^{-1})^{1}(x)\bigr)}\biggr]
    \Bigl[\partial_{y_{2}}\widetilde{v}_{\tau}\bigl(\widetilde{\psi}_{j}^{-1}(x)\bigr)\Bigr]
    \\
    =&-k\int_{\widetilde{\mathcal{U}}_{j}\cap\Omega}\!{}
    \biggl[\frac{\kappa\bigl((\widetilde{\psi}_{j}^{-1})^{1}(x)\bigr)(\widetilde{u}_{\e})_{\tau}\bigl(\widetilde{\psi}_{j}^{-1}(x)\bigr)}{1-(\widetilde{\psi}_{j}^{-1})^{2}(x)\kappa\bigl((\widetilde{\psi}_{j}^{-1})^{1}(x)\bigr)}\biggr]
    \Bigl[\partial_{y_{2}}(\widetilde{v}^{R})_{\tau}\bigl(\widetilde{\psi}_{j}^{-1}(x)\bigr)\Bigr].
\end{align*}
Combining our previous work now gives
\begin{align*}
    &\sum_{i=1}^{2}\int_{\widetilde{\mathcal{U}}_{j}}\!{}
    \bigl<\nabla{}U_{\e}^{i},\nabla{}v^{i}\bigr>_{j}
    +k\int_{\widetilde{\mathcal{U}}_{j}}\!{}
    (\Curj U_{\e})(\Curj v)
    \nonumber\\
    &+k\int_{\widetilde{\mathcal{U}}_{j}}\!{}
    |\det(\nabla\sigma_{j}(x))|\mathcal{D}_{j}(x)
    \biggl[\frac{\kappa\bigl((\widetilde{\psi}_{j}^{-1})^{1}(x)\bigr)U_{\e}(x)\cdot\tau\bigl((\widetilde{\psi}_{j}^{-1})^{1}(x)\bigr)}
    {1-(\widetilde{\psi}_{j}^{-1})^{2}(x)\kappa\bigl((\widetilde{\psi}_{j}^{-1})^{1}(x)\bigr)}\biggr]
    \nabla\bigl[v(x)\cdot{}n\bigl(\widetilde{\psi}_{j}^{-1}(x)\bigr)\bigr]\cdot{}n\bigl(\widetilde{\psi}_{j}^{-1}(x)\bigr)
    \nonumber\\
    =&\int_{\widetilde{\mathcal{U}}_{j}\cap\Omega}\!{}
    \nabla(\widetilde{u}_{\e})_{\tau}\bigl(\widetilde{\psi}_{j}^{-1}(x)\bigr)^{T}
    \bigl[\mathcal{G}_{j}(x,\widetilde{\psi}_{j}^{-1}(x))\bigr]
    \bigl[\nabla\widetilde{v}_{\tau}\bigl(\widetilde{\psi}_{j}^{-1}(x)\bigr)+
    \nabla\widetilde{v}_{\tau}^{R}\bigl(\widetilde{\psi}_{j}^{-1}(x)\bigr)\bigr]\\
    &+
    \int_{\widetilde{\mathcal{U}}_{j}\cap\Omega}\!{}
    \nabla(\widetilde{u}_{\e})_{n}\bigl(\widetilde{\psi}_{j}^{-1}(x)\bigr)^{T}
    \bigl[\mathcal{G}_{j}(x,\widetilde{\psi}_{j}^{-1}(x))\bigr]
    \bigl[\nabla\widetilde{v}_{n}\bigl(\widetilde{\psi}_{j}^{-1}(x)\bigr)-
    \nabla\widetilde{v}_{n}^{R}\bigl(\widetilde{\psi}_{j}^{-1}(x)\bigr)\bigr]\\
    &+k\int_{\widetilde{\mathcal{U}}_{j}\cap\Omega}\!{}
    \biggl[
    \partial_{y_{2}}(\widetilde{u}_{\e})_{\tau}\bigl(\widetilde{\psi}_{j}^{-1}(x)\bigr)+
    \frac{\partial_{y_{1}}(\widetilde{u}_{\e})_{n}\bigl(\widetilde{\psi}_{j}^{-1}(x)\bigr)}
    {1-(\widetilde{\psi}_{j}^{-1})^{2}(x)\kappa\bigl((\widetilde{\psi}_{j}^{-1})^{1}(x)\bigr)}
    \biggr]
    \bigl[\partial_{y_{2}}\widetilde{v}_{\tau}\bigl(\widetilde{\psi}_{j}^{-1}(x)\bigr)-\partial_{y_{2}}(\widetilde{v}^{R})_{\tau}\bigl(\widetilde{\psi}_{j}^{-1}(x)\bigr)
    \bigr]
    \\
    &+k\int_{\widetilde{\mathcal{U}}_{j}\cap\Omega}\!{}
    \biggl[
    \partial_{y_{2}}(\widetilde{u}_{\e})_{\tau}\bigl(\widetilde{\psi}_{j}^{-1}(x)\bigr)+
    \frac{\partial_{y_{1}}(\widetilde{u}_{\e})_{n}\bigl(\widetilde{\psi}_{j}^{-1}(x)\bigr)}
    {1-(\widetilde{\psi}_{j}^{-1})^{2}(x)\kappa\bigl((\widetilde{\psi}_{j}^{-1})^{1}(x)\bigr)}
    \biggr]
    \frac{\partial_{y_{1}}\widetilde{v}_{\tau}\bigl(\widetilde{\psi}_{j}^{-1}(x)\bigr)+\partial_{y_{1}}(\widetilde{v}^{R})_{\tau}\bigl(\widetilde{\psi}_{j}^{-1}(x)\bigr)}
    {1-(\widetilde{\psi}_{j}^{-1})^{2}(x)\kappa\bigl((\widetilde{\psi}_{j}^{-1})^{1}(x)\bigr)}
    \\
    &+k\int_{\widetilde{\mathcal{U}}_{j}\cap\Omega}\!{}
    \biggl[\frac{\kappa\bigl((\widetilde{\psi}_{j}^{-1})^{1}(x)\bigr)(\widetilde{u}_{\e})_{\tau}\bigl(\widetilde{\psi}_{j}^{-1}(x)\bigr)}{1-(\widetilde{\psi}_{j}^{-1})^{2}(x)\kappa\bigl((\widetilde{\psi}_{j}^{-1})^{1}(x)\bigr)}\biggr]
    \Bigl[\partial_{y_{2}}\widetilde{v}_{\tau}\bigl(\widetilde{\psi}_{j}^{-1}(x)\bigr)
    -\partial_{y_{2}}(\widetilde{v}^{R})_{\tau}\bigl(\widetilde{\psi}_{j}^{-1}(x)\bigr)\Bigr]
    \\
    &+\int_{\widetilde{\mathcal{U}}_{j}\setminus\Omega}\!{}
    \widetilde{F}_{j}\bigl(x,U_{\e}(x),\nabla{}U_{\e}(x)\bigr)\cdot
    v(x).
\end{align*}
Since the operation
\begin{equation*}
    w\mapsto{}\mathfrak{R}_{j}(x,w(\sigma_{j}(x)))
\end{equation*}
is an involution for each
$x\in\widetilde{\mathcal{U}}_{j}$ then
we may define the even part relative to this involution, denoted
$w_{E}$, by
\begin{equation*}
    w_{E}(x)\coloneqq
    \frac{w(x)+\mathfrak{R}_{j}(w(\sigma_{j}(x)))}{2}.
\end{equation*}
Observe that
\begin{equation*}
    (w_{E})_{\tau}(x)=
    \frac{\widetilde{w}_{\tau}(\widetilde{\psi}_{j}^{-1}(x))
    +\widetilde{w}_{\tau}(\widetilde{\psi}_{j}^{-1}(R(x)))}{2},
    \hspace{15pt}
    (w_{E})_{n}(x)=
    \frac{\widetilde{w}_{n}(\widetilde{\psi}_{j}^{-1}(x))
    -\widetilde{w}_{n}(\widetilde{\psi}_{j}^{-1}(R(x)))}{2}
\end{equation*}
and that for $x\in\Gamma$ satisfying
$x=\widetilde{\psi}_{j}(y_{1},0)$
\begin{equation*}
    (w_{E})_{n}(x)
    =\frac{\widetilde{w}_{n}(\widetilde{\psi}_{j}^{-1}(x))
    -\widetilde{w}_{n}(\widetilde{\psi}_{j}^{-1}(x))}{2}
    =\frac{\widetilde{w}_{n}(y_{1},0)
    -\widetilde{w}_{n}(y_{1},0)}{2}
    =0.
\end{equation*}
Thus, $w_{E}$ only has tangential part along $\Gamma$.
Using this notation in the previous calculation combined with
the PDE satisfied by $u_{\e}$ we find that
\begin{align*}
    &
    \sum_{i=1}^{2}\int_{\widetilde{\mathcal{U}}_{j}}\!{}
    \bigl<\nabla{}U_{\e}^{i}(x),\nabla{}v^{i}(x)\bigr>_{j}
    +k\int_{\widetilde{\mathcal{U}}_{j}}\!{}
    \Curj(U_{\e})(x)\Curj(v)(x)\\
    &+k\int_{\widetilde{\mathcal{U}}_{j}}\!{}
    |\det(\nabla\sigma_{j}(x))|\mathcal{D}_{j}(x)
    \biggl[\frac{\kappa\bigl((\widetilde{\psi}_{j}^{-1})^{1}(x)\bigr)U_{\e}(x)\cdot\tau\bigl((\widetilde{\psi}_{j}^{-1})^{1}(x)\bigr)}
    {1-(\widetilde{\psi}_{j}^{-1})^{2}(x)\kappa\bigl((\widetilde{\psi}_{j}^{-1})^{1}(x)\bigr)}\biggr]
    \nabla\bigl[v(x)\cdot{}n\bigl(\widetilde{\psi}_{j}^{-1}(x)\bigr)\bigr]\cdot{}n\bigl(\widetilde{\psi}_{j}^{-1}(x)\bigr)
    \\
    =&2\int_{\widetilde{\mathcal{U}}_{j}\cap\Omega}\!{}
        \frac{U_{\e}\cdot{}v_{E}}{\varepsilon^{2}}(1-|U_{\e}|^{2})
        +\int_{\widetilde{\mathcal{U}}_{j}}\!{}
        \widetilde{\mathcal{F}}_{j}\bigl(x,U_{\e}(x),\nabla{}U_{\e}(x)\bigr)\cdot
        v(x)
\end{align*}
where $\widetilde{\mathcal{F}}_{j}$ combines $F_{j,\tau}$,
$F_{j,n}$, and $\widetilde{F}_{j}$.
Noting that
\begin{equation*}
    U_{\e}(x)\cdot\bigl[\mathfrak{R}_{j}\bigl(x,v(\sigma_{j}(x))\bigr)\bigr]=\mathfrak{R}_{j}(x,U_{\e}(x))\cdot{}v(\sigma_{j}(x))
    \, ,
\end{equation*}
and using the changing variables $\sigma_{j}$ gives
\begin{align*}
    &\sum_{i=1}^{2}\int_{\widetilde{\mathcal{U}}_{j}}\!{}
    \bigl<\nabla{}U_{\e}^{i},\nabla{}v^{i}\bigr>_{j}
    +k\int_{\widetilde{\mathcal{U}}_{j}}\!{}
    (\Curj U_{\e})(\Curj v)
    \nonumber\\
    &+k\int_{\widetilde{\mathcal{U}}_{j}}\!{}
    |\det(\nabla\sigma_{j}(x))|\mathcal{D}_{j}(x)
    \biggl[\frac{\kappa\bigl((\widetilde{\psi}_{j}^{-1})^{1}(x)\bigr)U_{\e}(x)\cdot\tau\bigl((\widetilde{\psi}_{j}^{-1})^{1}(x)\bigr)}
    {1-(\widetilde{\psi}_{j}^{-1})^{2}(x)\kappa\bigl((\widetilde{\psi}_{j}^{-1})^{1}(x)\bigr)}\biggr]
    \nabla\bigl[v(x)\cdot{}n\bigl(\widetilde{\psi}_{j}^{-1}(x)\bigr)\bigr]\cdot{}n\bigl(\widetilde{\psi}_{j}^{-1}(x)\bigr)
    \nonumber\\
    =&\int_{\widetilde{\mathcal{U}}_{j}\cap\Omega}\!{}
    \frac{U_{\e}(x)\cdot{}v(x)}{\varepsilon^{2}}(1-|U_{\e}(x)|^{2})
    \nonumber\\
    &+\int_{\widetilde{\mathcal{U}}_{j}\setminus\Omega}\!{}
    |\det(\nabla\sigma_{j}(x))|
    \frac{\mathfrak{R}_{j}\bigl(x,U_{\e}(x)\bigr)\cdot{}v(x)}{\varepsilon^{2}}(1-|U_{\e}(x)|^{2})
    \nonumber\\
    &+\int_{\widetilde{\mathcal{U}}_{j}}\!{}\widetilde{\mathcal{F}}_{j}
    \bigl(x,U_{\e}(x),\nabla{}U_{\e}(x)\bigr)\cdot{}v(x).
\end{align*}

\bibliographystyle{alpha}
\bibliography{references}

\end{document}